\documentclass[12pt]{amsart}
\usepackage{amsmath,amsthm,amsfonts,amscd,amssymb,eucal,latexsym,mathrsfs, appendix, yhmath}
\usepackage[numbers,sort&compress]{natbib}
\usepackage[all,cmtip]{xy}

\newtheorem{theorem}{Theorem}[section]
\newtheorem{corollary}[theorem]{Corollary}
\newtheorem{lemma}[theorem]{Lemma}
\newtheorem{proposition}[theorem]{Proposition}

\theoremstyle{definition}
\newtheorem{definition}[theorem]{Definition}
\newtheorem{remark}[theorem]{Remark}
\newtheorem{notation}[theorem]{Notation}
\newtheorem{example}[theorem]{Example}
\newtheorem{question}[theorem]{Question}

\newtheorem{conjecture}[theorem]{Conjecture}

\numberwithin{equation}{section}

\newcommand{\htopol}{{\text{\rm h}}_{\text{\rm top}}}

\newcommand{\cB}{{\mathcal B}}

\newcommand{\cL}{{\mathcal L}}

\newcommand{\cM}{{\mathcal M}}

\newcommand{\Cb}{{\mathbb C}}

\newcommand{\Zb}{{\mathbb Z}}
\newcommand{\Tb}{{\mathbb T}}
\newcommand{\Qb}{{\mathbb Q}}
\newcommand{\Rb}{{\mathbb R}}
\newcommand{\Nb}{{\mathbb N}}

\newcommand{\fM}{{\mathfrak M}}
\newcommand{\fN}{{\mathfrak N}}

\newcommand{\supp}{{\rm supp}}

\newcommand{\IE}{{\rm IE}}

\newcommand{\oA}{{\boldsymbol{A}}}

\newcommand{\ox}{{\boldsymbol{x}}}
\newcommand{\oy}{{\boldsymbol{y}}}

\newcommand{\M}{M}

\newcommand{\rh}{{\rm h}}

\allowdisplaybreaks

\begin{document}

\title{Homoclinic groups, IE groups, and expansive algebraic actions}

\author{Nhan-Phu Chung}
\author{Hanfeng Li}
\address{\hskip-\parindent
Nhan-Phu Chung, Department of Mathematics, SUNY at Buffalo,
Buffalo, NY \text{14260-2900, U.S.A.} 
\break
{\it Current address:} Max Planck Institute-Mathematics in the Sciences, Inselstra{\ss}e 22, 04103 Leipzig, Germany}
\email{chung@mis.mpg.de}

\address{\hskip-\parindent
Hanfeng Li, Chongqing Institute of Mathematics, Center of Mathematics, Chongqing \text{University,
Chongqing 401331, China}
\break
Department of Mathematics, SUNY at Buffalo,
Buffalo, NY 14260-2900, U.S.A.}
\email{hfli@math.buffalo.edu}


\date{April 15, 2014}

\begin{abstract}
We give algebraic characterizations for expansiveness of algebraic actions of countable groups. The notion of
$p$-expansiveness is introduced for algebraic actions, and we show that for countable amenable groups,
a finitely presented algebraic action is $1$-expansive exactly when it has finite entropy. We also study the local entropy theory for actions of
countable amenable groups on compact groups by automorphisms, and show that the IE group determines the Pinsker factor for such actions.
For an expansive algebraic action of  a polycyclic-by-finite group on $X$, it is shown that the entropy of the action is equal to the entropy of the induced action on the Pontryagin dual of the homoclinic group, the homoclinic group is a dense subgroup of the IE group, the homoclinic group is nontrivial exactly when the action has positive entropy, and the homoclinic group is dense in $X$
exactly when the action has completely positive entropy.
\end{abstract}

\maketitle

\section{Introduction} \label{S-introduction}

 Given a (continuous) action of a countable discrete amenable group $\Gamma$ on a compact metrizable space $X$, one has the topological entropy $\htopol(X)$ of the action, lying in $[0, +\infty]$. Besides being an invariant of the action, the entropy also gives us a lot of information about the action itself.
 Indeed, the intuition about the entropy is that the larger the entropy is, the more complicated the action is. Thus it is very natural to ask for the relation between entropy properties of the action and the asymptotic behavior of orbits of the action, i.e. the asymptotics  of $\rho(sx, sy)$ as elements $s$ of $\Gamma$ go to infinity, where $\rho$ is a compatible metric on $X$ and $x, y\in X$. 

 A well-known result in this direction is that of Blanchard et al. \cite{BGKM} in the case of $\Gamma=\Zb$. They showed that positive entropy implies Li-Yorke chaos.
That is, if the $\Zb$-action is generated by a homeomorphism $T:X\rightarrow X$ and $\htopol(X)>0$, then there exists an uncountable subset $Z$ of $X$
such that for any distinct $x, y$ in $Z$ one has $\limsup_{n\to +\infty}\rho(T^nx, T^ny)>0$ and $\liminf_{n\to +\infty}\rho(T^nx, T^ny)=0$.

In this article we concentrate on the phenomenon $\lim_{s\to \infty}\rho(sx, sy)=0$ for points under
the action of a general group $\Gamma$. A pair $(x, y)$ of points in $X$ satisfying
$\lim_{s\to \infty}\rho(sx, sy)=0$ is called {\it asymptotic} or {\it homoclinic}.
In the case $\Gamma=\Zb$ and the action is generated by a homeomorphism
$T:X\rightarrow X$, a pair $(x, y)$ of points in $X$ satisfying
$\lim_{n\to +\infty}\rho(T^nx, T^ny)=0$ is called {\it positively asymptotic} or {\it positively homoclinic}.
Thus, one of the questions we want to address is the relation between $\htopol(X)>0$ and the existence of non-diagonal asymptotic pairs.

A positive result on this question is that of Blanchard et al. \cite{BHR}. They showed that in the case $\Gamma=\Zb$, when $\htopol(X)>0$, there exist non-diagonal {\it positively} asymptotic pairs. On the other hand, Lind and Schmidt constructed examples of $\Zb$-actions (actually toral automorphisms) which have positive entropy but no non-diagonal asymptotic pairs \cite[Example 3.4]{LS99}. Thus one has to add further conditions.

The condition we are going to add is expansiveness. An action $\Gamma \curvearrowright X$ is called {\it expansive} if there exists $r>0$ such that
$\sup_{s\in \Gamma}\rho(sx, sy)\ge r$ for all distinct $x, y$ in $X$. For example, for any $k\in \Nb$ and $A\in M_k(\Zb)$ being invertible in $M_k(\Zb)$,
the toral automorphism of $\Rb^k/\Zb^k=(\Rb/\Zb)^k$ defined by $x+\Zb^k\mapsto Ax+\Zb^k$ for $x\in \Rb^k$ is expansive if and only if $A$ has no eigenvalues with absolute value $1$ \cite[page 143]{Walters}. Bryant showed that for expansive $\Zb$-actions, when $X$ is infinite, there are non-diagonal positively asymptotic pairs \cite[Theorem 2]{Bryant}.
Schmidt showed that when $\Gamma=\Zb^d$ for some $d\in \Nb$, every subshift of finite type (which is always expansive) with positive entropy has non-diagonal asymptotic pairs \cite[Proposition 2.1]{Schmidt95}.
These results lead us to ask  the following question:

\begin{question} \label{Q-homoclinic vs IE}
Let a countable discrete amenable group $\Gamma$ act on a compact metrizable space $X$  expansively.
If $\htopol(X)>0$, then must there be a non-diagonal asymptotic pair in $X$?
\end{question}

Despite all the  evidence above, we do not know the answer to Question~\ref{Q-homoclinic vs IE} even in the case $\Gamma=\Zb$.
One of the main results of this article is that Question~\ref{Q-homoclinic vs IE} has an affirmative answer for algebraic actions of polycyclic-by-finite groups.

Actions of countable discrete groups $\Gamma$ on compact (metrizable) groups $X$ by (continuous) automorphisms
are a rich class of dynamical systems, and have drawn much attention since the beginning of ergodic theory. Among such actions, the so called {\it algebraic actions}, meaning that $X$ is abelian in which case the action is completely determined by the module structure of the Pontryagin dual $\widehat{X}$ of $X$ over the integral group ring $\Zb\Gamma$ of $\Gamma$, is especially important because of the beautiful interplay between dynamics, Fourier analysis, and commutative or noncommutative algebra.

The $\Zb$-actions on compact groups by automorphisms are well understood now (cf. \cite{Lind77, MT78, Yuz1, Yuz2}). After investigation during the last few decades, much is also known for such actions of $\Zb^d$ (cf. \cite{EW05, KKS, KS, KS00, LS99, LS02, LSV, LSW, RS, Schmidt, Schmidt01, SW}). The fact that the integral group ring of $\Zb^d$ is a commutative factorial Noetherian ring plays a vital role for such study, as it makes the machinery of commutative algebra available.
In the last several years, much progress has been made towards  understanding the algebraic actions of general countable groups $\Gamma$ (cf. \cite{BG, Bow4, Den, DS, ER, KerLi10, Li, MW}). It is somehow surprising that operator algebras, especially the group $C^*$-algebras or group von Neumann algebras of $\Gamma$, turn out to be important for such a study.

Let a countable group $\Gamma$ act on a compact  group $X$ by automorphisms, and denote by $e_X$ the identity element of $X$.
A point $x\in X$ is called {\it homoclinic} if the pair $(x, e_X)$ is asymptotic, i.e. $sx\to e_X$ when $\Gamma \ni s\to \infty$. When $\Gamma$ is amenable, a point $x\in X$ is called {\it IE} if, for any neighborhoods $U_1$ and $U_2$ of $x$ and $e_X$ respectively, there exists $c>0$ such that for any sufficiently left invariant nonempty finite set $F\subseteq \Gamma$ one can find some $F'\subseteq F$ with $|F'|\ge c|F|$ being an independence set for $(U_1, U_2)$ in the sense that for any map $\sigma: F'\rightarrow \{1, 2\}$ one has $\bigcap_{s\in F'}s^{-1}U_{\sigma(s)}\neq \emptyset$.
The set of all homoclinic points (resp. IE points), denoted by $\Delta(X)$ (resp. $\IE(X)$), is a $\Gamma$-invariant subgroup of $X$.
It is easy to see that $\Delta(X)$ describes all the asymptotic pairs of $X$ in the sense that
a pair $(x, y)$ of points in $X$ is asymptotic if and only if $xy^{-1}$ lies in $\Delta(X)$.
A group $\Gamma$ is called {\it polycyclic-by-finite} \cite[page 422]{Passman} if there is a sequence of subgroups $\Gamma=\Gamma_1\rhd \Gamma_2\rhd\cdots \rhd \Gamma_n=\{e_\Gamma\}$ such
that $\Gamma_j/\Gamma_{j+1}$ is finite or cyclic for every $j=1, \dots, n-1$. The polycyclic-by-finite groups are exactly the virtually solvable groups each of whose subgroups is finitely generated (cf. \cite[pages 2 and 4]{Segal}). One of our main results is

\begin{theorem} \label{T-main}
Let $\Gamma$ be a polycyclic-by-finite group.
 Let $\Gamma$ act on a compact abelian group $X$ expansively by automorphisms.
Then the following hold:
\begin{enumerate}
\item Let $G$ be a  $\Gamma$-invariant subgroup of $\Delta(X)$ such that
$G$ and $\Delta(X)$ have the same closure. Treat $G$ as a discrete abelian group and consider the induced $\Gamma$-action on the Pontryagin dual $\widehat{G}$.
Then the actions $\Gamma\curvearrowright X$ and $\Gamma \curvearrowright \widehat{G}$ have the same entropy.

\item $\Delta(X)$ is a dense subgroup of $\IE(X)$.

\item The action has positive entropy if and only if $\Delta(X)$ is nontrivial.

\item The action has completely positive entropy with respect to the normalized Haar measure of $X$ if and only if $\Delta(X)$ is dense in $X$.
\end{enumerate}
\end{theorem}

Note that the assertion (3) of Theorem~\ref{T-main} answers Question~\ref{Q-homoclinic vs IE} affirmatively for algebraic actions of polycyclic-by-finite groups.
In the case $\Gamma=\Zb^d$ for some $d\in \Nb$, the assertions (3) and (4) are the main results of Lind and Schmidt in \cite{LS99}, and the assertion (1) was proved by Einsiedler and Schmidt for $G=\Delta(X)$ \cite{ES}. As we mentioned above, the work in \cite{LS99, ES} relies heavily on the machinery of commutative algebra.
When $\Gamma$ is nonabelian, we do not have such tools available anymore. Instead, the $\ell^1$-group algebra $\ell^1(\Gamma)$ and the group von Neumann  algebra of $\Gamma$ play a crucial role.

This paper has three parts:
expansive algebraic actions of countable groups, the local entropy theory for actions of countable amenable groups on compact groups by automorphisms, and duality for algebraic actions of countable amenable groups. Theorem~\ref{T-main} is the outcome of these three parts.

In Section~\ref{S-expansive} we give algebraic characterizations for expansiveness of algebraic actions of a countable group $\Gamma$. Given a matrix $A\in M_k(\Zb\Gamma)$ being invertible in $M_k(\ell^1(\Gamma))$, the canonical $\Gamma$-action on the Pontryagin dual $X_A$ of the $\Zb\Gamma$-module
$(\Zb\Gamma)^k/(\Zb\Gamma)^kA$ is expansive. Dynamically, our characterization says that the expansive algebraic actions of $\Gamma$ are exactly the restriction to closed $\Gamma$-invariant subgroups of $X_A$ for all such $A$.

The notion of {\it $p$-expansiveness} is introduced in Section~\ref{S-p version} for algebraic actions of a countable group $\Gamma$, for $1\le p\le +\infty$.
This yields a hierarchy of expansiveness: $q$-expansiveness implies $p$-expansiveness for $p<q$ and $+\infty$-expansiveness is exactly the ordinary expansiveness.
We show that, for an algebraic action of $\Gamma$ on $X$, if $\Gamma$ is amenable and
$\widehat{X}$ is a finitely presented $\Zb\Gamma$-module, then the action has finite entropy if and only if
it is $1$-expansive. This relies on a result of Elek \cite{Elek03} about the analytic zero divisor conjecture, the proof of which uses the group von Neumann algebra of $\Gamma$.

In Section~\ref{S-p-homoclinic} we study the group $\Delta^p(X)$ of {\it $p$-homoclinic} points for an algebraic action of a countable group $\Gamma$ and $1\le p<+\infty$. They are subgroups of $\Delta(X)$.  For expansive algebraic actions, using our characterization for such actions, we show that $\Delta^1(X)=\Delta(X)$.

Section~\ref{S-specification} gives another application of our characterization for expansive algebraic actions of a countable group $\Gamma$. We show that various specification properties are equivalent and imply that $\Delta(X)$ is dense in $X$ for such actions.

Initiated by Blanchard \cite{Blanchard}, the local entropy theory for continuous actions of a countable amenable group $\Gamma$ on compact spaces developed quickly during the last two decades (cf. \cite{BGH, BGKM, BHMMR, BHR, Glasner, GY, HLSY, HMRY, HY06, HY09, HYZ}), and has been found to be related to combinatorial independence \cite{Ind, MInd}, which appeared first in Rosenthal's work on Banach spaces containing $\ell^1$ \cite{ell1}. We develop the local entropy theory for actions of a countable amenable group $\Gamma$ on compact groups $X$ by automorphisms in Sections~\ref{S-IE} and \ref{S-Pinsker}. It turns out that $\IE(X)$ determines the local entropy theory and the Pinsker factor for such actions.
Furthermore, when $X$ is abelian, one has $\Delta^1(X)\subseteq \IE(X)$. In particular, the ``if'' parts of the assertions (3) and (4) of Theorem~\ref{T-main} actually hold for all countable amenable groups $\Gamma$. This also enables us to give a partial answer to a question of Deninger about the Fuglede-Kadison determinant (see Corollary~\ref{C-answer to Deninger}), which is an application to the study of the group von Neumann algebra of $\Gamma$.
We also show that, for finite entropy actions on compact groups by automorphisms, having completely positive entropy is equivalent to having a unique maximal measure.

For an algebraic action of a countable group $\Gamma$ on $X$, we treat $\widehat{X}$ and $\Delta^p(X)$ as a dual pair of discrete abelian groups, with the expectation that
the dynamical properties of the $\Gamma$-actions on $X$ and $\widehat{\Delta^p(X)}$ would be reflected in each other.
We discuss the relation of the entropy properties for such pairs of actions in Section~\ref{S-duality}, for countable amenable groups $\Gamma$.
It turns out that $\Delta^1(X)$ and $\Delta^2(X)$ are more closely related to the entropy properties of
$\Gamma \curvearrowright X$ than any other $\Delta^p(X)$ or $\Delta(X)$. In particular, when $\widehat{X}$ is a finitely presented $\Zb\Gamma$-module,
the entropy of $\Gamma \curvearrowright X$ is bounded below by that for $\Gamma \curvearrowright \widehat{\Delta^1(X)}$. This depends on the equivalence between $1$-expansiveness and finite entropy mentioned above. Actually we establish a general result (see Theorem~\ref{T-dual entropy for Noetherian 1-homoclinic})
stating the relation between $\Delta^1(X)$ and the entropy properties of $\Gamma \curvearrowright X$, much as the assertions of Theorem~\ref{T-main}. Then
Theorem~\ref{T-main} is just a consequence of this result and our characterizations of expansive algebraic actions.

In fact Theorem~\ref{T-main} holds whenever $\Gamma$ is amenable and $\Zb\Gamma$ is left Noetherian. It is known that a polycyclic-by-finite group is amenable and its integral group ring is left Noetherian \cite{Hall} \cite[Theorem 10.2.7]{Passman}. On the other hand, it is a long standing open question whether $\Zb\Gamma$ is left Noetherian implies that $\Gamma$ is polycyclic-by-finite.

\medskip

\noindent{\it Acknowledgements.}
The second named author was partially supported by NSF
grants DMS-0701414 and DMS-1001625. He is grateful to Doug Lind and Klaus Schmidt for
very interesting discussions. We thank David Kerr for helpful comments.

\section{Preliminary} \label{S-preliminary}

In this section we set up some notations and recall some basic facts about group rings, algebraic actions, entropy theory, and local entropy theory.

Throughout this paper, $\Gamma$ will be a
countable discrete group. All compact spaces are assumed to be metrizable and all automorphisms of compact groups are assumed to be continuous.
For a group $G$, we write $e_G$ for the identity element of $G$. When $G$ is abelian, sometimes we also write $0_G$.

\subsection{Group Rings} \label{SS-group ring}

For a unital ring $R$, we denote by $R\Gamma$ the group ring of $\Gamma$ with coefficients in $R$. It consists of finitely supported $R$-valued functions
$f$
on $\Gamma$,  which we shall write as $\sum_{s\in \Gamma}f_ss$. The algebraic structure of $R\Gamma$ is defined by $(\sum_{s\in \Gamma} f_s s)+(\sum_{s\in \Gamma} g_s s)=\sum_{s\in \Gamma}(f_s+g_s)s$ and $(\sum_{s\in \Gamma}f_ss)(\sum_{s\in \Gamma}g_ss)=\sum_{s\in \Gamma}(\sum_{t\in \Gamma}f_tg_{t^{-1}s})s$.

We denote by $\ell^{\infty}(\Gamma)$ the Banach space of all bounded $\Rb$-valued functions on $\Gamma$, equipped with the $\ell^{\infty}$-norm
$\| \cdot \|_{\infty}$. We also denote by $\ell^1(\Gamma)$ the Banach algebra of all absolutely summable $\Rb$-valued functions on $\Gamma$, equipped with
the $\ell^1$-norm $\|\cdot \|_1$. Note that $\ell^1(\Gamma)$ has a canonical algebra structure extending that of $\Rb\Gamma$, and is a Banach algebra.
We shall write $f\in \ell^1(\Gamma)$ as $\sum_{s\in \Gamma}f_ss$. Note that $\ell^1(\Gamma)$ has an involution $f\mapsto f^*$ defined by
$(\sum_{s\in \Gamma}f_ss)^*=\sum_{s\in \Gamma}f_ss^{-1}$.

For each $k\in \Nb$, we endow $\Rb^k$ the supremum norm $\|\cdot \|_\infty$. For each $1\le p\le +\infty$, we endow
$\ell^p(\Gamma, \Rb^k)=(\ell^p(\Gamma))^k$ with the $\ell^p$-norm
\begin{align} \label{E-infinity norm}
\| (f_1, \dots, f_k)\|_p=\|\Gamma\ni s\mapsto \|(f_1(s), \dots, f_k(s))\|_\infty\|_p.
\end{align}
We shall write elements of $(\ell^p(\Gamma))^k$ as row vectors.

The algebraic structures of $\Zb\Gamma$ and $\ell^1(\Gamma)$ also extend to some other situations naturally. For example, $(\Rb/\Zb)^\Gamma$ becomes
a right $\Zb\Gamma$-module naturally.

For any $n, k\in \Nb$, we also endow $M_{n\times k}(\ell^1(\Gamma))$ with the norm
$$\|(f_{i, j})_{1\le i\le n, 1\le j\le k}\|_1:=\sum_{1\le i\le n, 1\le j\le k}\|f_{i, j}\|_1.$$
The involution of $\ell^1(\Gamma)$ also extends naturally to an isometric linear map $M_{n\times k}(\ell^1(\Gamma))\rightarrow M_{k\times n}(\ell^1(\Gamma))$
by
$$ (f_{i, j})_{1\le i\le n, 1\le j\le k}^*:=(f_{j, i}^*)_{1\le i\le k, 1\le j\le n}.$$
To simplify the notation, we shall write $M_k(\cdot)$ for $M_{k\times k}(\cdot)$. Note that $M_k(\ell^1(\Gamma))$  is a Banach algebra.

\subsection{Algebraic Actions} \label{SS-algebraic actions}

By an {\it algebraic action} of $\Gamma$, we mean an action of $\Gamma$ on a compact abelian group by automorphisms.

For a locally compact abelian group $X$, we denote by $\widehat{X}$ its Pontryagin dual. Then for any compact abelian group $X$, there is a natural one-to-one correspondence between algebraic actions of $\Gamma$ on $X$  and actions of $\Gamma$ on $\widehat{X}$ by automorphisms. There is also a natural one-to-one correspondence between the latter and left $\Zb\Gamma$-module structure on $\widehat{X}$.
Thus, when we have an algebraic action of $\Gamma$ on $X$, we shall talk about the left $\Zb\Gamma$-module $\widehat{X}$. And when we have a left $\Zb\Gamma$-module $W$, we shall treat $W$ as a discrete abelian group and talk about the algebraic action of $\Gamma$ on $\widehat{W}$.

Note that for each $k\in \Nb$, we may identify the Pontryagin dual $\widehat{(\Zb\Gamma)^k}$  of $(\Zb\Gamma)^k$ with $((\Rb/\Zb)^k)^\Gamma=((\Rb/\Zb)^\Gamma)^k$ naturally. Under this identification,
the canonical action of $\Gamma$ on $\widehat{(\Zb\Gamma)^k}$ is just the left shift action on $((\Rb/\Zb)^k)^\Gamma$. If $J$ is a left $\Zb\Gamma$-submodule
of $(\Zb\Gamma)^k$, then $\widehat{(\Zb\Gamma)^k/J}$ is identified with
$$ \{(x_1, \dots, x_k)\in ((\Rb/\Zb)^\Gamma)^k: x_1g_1^*+\cdots+x_kg_k^*=0_{(\Rb/\Zb)^\Gamma}, \mbox{ for all } (g_1, \dots, g_k)\in J\}.$$

We denote by $\rho$ the canonical metric on $\Rb/\Zb$ defined by
$$\rho(t + \Zb,  s+\Zb):=\min_{m\in \Zb}|t-s-m|.$$
For $k\in \Nb$, we denote by $\rho_\infty$ the metric on $(\Rb/\Zb)^k$ defined by
\begin{align} \label{E-metric on torus}
\rho_\infty((t_1, \dots, t_k), (s_1, \dots, s_k)):=\max_{1\le j\le k}\rho(t_j, s_j).
\end{align}

An action of $\Gamma$ on a compact space $X$ is called {\it expansive} if there is a constant $c>0$ such that $\sup_{s\in \Gamma} \rho(sx, sy)>c$ for all distinct $x, y$ in $X$, where $\rho$ is a compatible metric on $X$. It is easy to see that the definition does not depend on the choice of $\rho$.
If $\widehat{X}=(\Zb\Gamma)^k/J$ for some $k\in \Nb$ and some left $\Zb\Gamma$-submodule $J$ of $(\Zb\Gamma)^k$, then the $\Gamma$-action on $X$ is expansive exactly when there exists $c>0$ such that the only $x\in X$ satisfying
$$ \sup_{s\in \Gamma}\rho_\infty(x_s, 0_{(\Rb/\Zb)^k})<c$$
is $0_X$.

We shall need the following result \cite[Propositions 2.2 and Corollary 2.16]{Schmidt}.

\begin{proposition} \label{P-expansive to subshift}
Let $\Gamma$ act on a compact abelian group $X$ expansively by automorphisms. Then $\widehat{X}$ is a finitely generated left $\Zb\Gamma$-module.
\end{proposition}

\subsection{Entropy} \label{SS-entropy}

We refer the reader to \cite{JMO, OW} for details on the entropy theory of countable amenable groups.
Throughout this subsection $\Gamma$ will be a countable amenable group.

Let $\Gamma$ act on a compact space $X$ continuously. Fix a compatible metric $\rho$ on $X$ and a left F{\o}lner sequence $\{F_n\}_{n\in \Nb}$ in $\Gamma$, i.e.,
each $F_n$ is a nonempty finite subset of $\Gamma$ and $\frac{|KF_n\setminus F_n|}{|F_n|}\to 0$ as $n\to \infty$ for every finite set $K\subseteq \Gamma$.
For  a finite subset $F$ of $\Gamma$ and $\varepsilon>0$, we say that a set $Z\subseteq X$ is {\it $(\rho, F, \varepsilon)$-separated} if for any distinct
$y, z\in Z$ one has $\max_{s\in F}\rho(sy, sz)>\varepsilon$.  Denote by $N_{\rho, F, \varepsilon}(X)$ the maximal cardinality of $(\rho, F, \varepsilon)$-separated subsets of $X$. Then the {\it topological entropy} of the action $\Gamma\curvearrowright X$ is defined as
\begin{align} \label{E-entropy}
 \htopol(X)=\sup_{\varepsilon>0}\limsup_{n\to \infty}\frac{\log N_{\rho, F_n, \varepsilon}(X)}{|F_n|},
\end{align} 
and does not depend on the choice of the F{\o}lner sequence $\{F_n\}_{n\in \Nb}$ and the metric $\rho$. 

For a measure-preserving action of $\Gamma$ on a probability measure space $(X, \cB_X, \mu)$, one also has the {\it measure entropy} or {\it Kolmogorov-Sinai entropy} $\rh_\mu(X)$ defined.

When $\Gamma$ acts on a compact group $X$ by automorphisms, the topological entropy and the measure entropy with respect to the normalized Haar measure
$\mu_X$ on $X$ coincide \cite[Theorem 2.2]{Den}. Thus we shall simply denote by $\rh(X)$ this common value, and refer to it the entropy of the action.

One has the following Yuzvinski\u{i} addition formula \cite[Corollary 6.3]{Li}:

\begin{proposition} \label{P-addition formula}
Let $\Gamma$ act on a compact group $X$ by automorphisms. Let $Y$ be a closed $\Gamma$-invariant normal subgroup of $X$. Consider the restriction of the $\Gamma$-action to $Y$ and the induced $\Gamma$-action on $X/Y$. Then
$$ \rh(X)=\rh(Y)+\rh(X/Y).$$
\end{proposition}

\subsection{Local Entropy Theory} \label{SS-local}

The local entropy theory was initiated by Blanchard \cite{Blanchard}. We refer the readers to \cite{Glasner, GY} for a nice account for the case $\Gamma=\Zb$.
In \cite{Ind, MInd} Kerr and the second named author gave a systematic combinatorial approach to the local entropy theory for general countable amenable groups. Here we follow the terminologies in \cite{Ind, MInd}.
Throughout this subsection, $\Gamma$ will be a countable amenable group.

\begin{definition} \label{D-IE}
Let $\Gamma$ act on a compact space $X$ continuously. For a tuple $\oA=(A_1, \dots, A_k)$ of subsets of $X$, we say that a finite set $F\subseteq \Gamma$ is an {\it independence set for $\oA$} if for every function $\sigma:F\rightarrow \{1, 2, \dots, k\}$ one has $\bigcap_{s\in F}s^{-1}A_{\sigma(s)}\neq \emptyset$.
We call a tuple $\ox=(x_1, \dots, x_k)\in X^k$ an {\it IE-tuple} if for every product neighborhood $U_1\times \cdots \times U_k$ of $\ox$, there exist a nonempty finite set $K\subseteq \Gamma$ and $c, \varepsilon>0$ such that for any finite set $F\subseteq \Gamma$ with $|KF\setminus F|\le \varepsilon|F|$ the tuple $(U_1,\dots, U_k)$ has an independence set $F'\subseteq F$ with $|F'|\ge c|F|$. We denote the set of IE-tuples of length $k$ by $\IE_k(X)$.
\end{definition}

We need the following properties of IE-tuples \cite[Proposition 3.9, Proposition 3.12, Theorem 3.15]{Ind}. For a continuous action of $\Gamma$ on a compact space $X$, we denote by $\cM(X, \Gamma)$ the set of $\Gamma$-invariant Borel probability measures on $X$.

\begin{theorem} \label{T-basic IE}
Let $\Gamma$ act on compact spaces $X$ and $Y$ continuously. Let $k\in \Nb$. Then the following hold:
\begin{enumerate}
\item $\IE_k(X)$ is a closed $\Gamma$-invariant subset of $X^k$, for the product $\Gamma$-action on $X^k$.

\item $\IE_2(X)$ has non-diagonal elements if and only if $\rh(X)>0$.

\item $\IE_1(X)$ is the closure of $\bigcup_{\mu \in \cM(X, \Gamma)}\supp(\mu)$.

\item Let $\pi:X\rightarrow Y$ be a $\Gamma$-equivariant continuous surjective map. Then $(\pi \times \cdots \times \pi)(\IE_k(X))=\IE_k(Y)$.

\item $\IE_k(X\times Y)=\IE_k(X)\times \IE_k(Y)$, where we take the product $\Gamma$-action on $X\times Y$, and identify $(X\times Y)^k$ with $X^k\times Y^k$ naturally.
\end{enumerate}
\end{theorem}

\begin{definition} \label{D-measure IE}
Let $\Gamma$ act on a compact space $X$ continuously and let $\mu\in \cM(X, \Gamma)$. For a tuple $\oA=(A_1, \dots, A_k)$ of subsets of $X$ and a subset $D$ of $X$, we say that a finite set $F\subseteq \Gamma$ is an {\it independence set for $\oA$ relative to $D$} if for every function $\sigma:F\rightarrow \{1, 2, \dots, k\}$ one has $D\cap \bigcap_{s\in F}s^{-1}A_{\sigma(s)}\neq \emptyset$.
We call a tuple $\ox=(x_1, \dots, x_k)\in X^k$ a {\it $\mu$-IE-tuple} if for every product neighborhood $U_1\times \cdots \times U_k$ of $\ox$, there exist
$c, \delta>0$ such that for any nonempty finite set $K\subseteq \Gamma$ and $\varepsilon>0$ one can find
a nonempty finite set $F\subseteq \Gamma$ with $|KF\setminus F|\le \varepsilon|F|$ such that for every Borel set $D\subseteq X$ with $\mu(D)\ge 1-\delta$ the tuple $(U_1,\dots, U_k)$ has an independence set $F'\subseteq F$ relative to $D$ with $|F'|\ge c|F|$. We denote the set of $\mu$-IE-tuples of length $k$ by $\IE_k^\mu(X)$.
\end{definition}

We need the following properties of measure IE-tuples \cite[Proposition 2.16, Theorem 2.21, Theorem 2.30]{MInd}:

\begin{theorem} \label{T-basic measure IE}
Let $\Gamma$ act on compact spaces $X$ and $Y$ continuously. Let $\mu\in \cM(X, \Gamma)$ and $\nu\in \cM(Y, \Gamma)$. Let $k\in \Nb$. Then the following hold:
\begin{enumerate}
\item $\IE_k^\mu(X)$ is a closed $\Gamma$-invariant subset of $X^k$, for the product $\Gamma$-action on $X^k$.

\item $\IE_2^\mu(X)$ has non-diagonal elements if and only if $\rh_\mu(X)>0$.

\item $\IE_1^\mu(X)=\supp(\mu)$.

\item Let $\pi:X\rightarrow Y$ be a $\Gamma$-equivariant continuous surjective map. Then $(\pi \times \cdots \times \pi)(\IE_k^\mu(X))=\IE_k^{\pi_*(\mu)}(Y)$.

\item $\IE_k^{\mu \times \nu}(X\times Y)=\IE_k^\mu(X)\times \IE_k^\nu(Y)$, where we take the product $\Gamma$-action on $X\times Y$, and identify $(X\times Y)^k$ with $X^k\times Y^k$ naturally.

\item $\IE_k(X)$ is the closure of $\bigcup_{\mu \in \cM(X, \Gamma)}\IE_k^\mu(X)$.
\end{enumerate}
\end{theorem}

\section{Algebraic Characterizations of Expansive Algebraic Actions} \label{S-expansive}

In this section we prove the following algebraic characterizations for expansiveness of algebraic actions.
Throughout this section $\Gamma$ will be a countable discrete group. For a unital ring $R$, a right $R$-module $\fM$ and
a left $R$-module $\fN$, we denote by $\fM\otimes_R\fN$ the tensor product of $\fM$ and $\fN$ \cite[Section 19]{AF}, which is an abelian group.

\begin{theorem} \label{T-algebraic characterization of expansive}
Let $\Gamma$ act on a compact abelian group $X$ by automorphisms. Then  the following are equivalent:
\begin{enumerate}
\item the action is expansive;

\item the left $\Zb\Gamma$-module $\widehat{X}$ is finitely generated, and if we identify $\widehat{X}$ with
$(\Zb\Gamma)^k/J$ for some $k\in \Nb$ and some left $\Zb\Gamma$-submodule $J$ of $(\Zb\Gamma)^k$, then there exists $A\in M_k(\Zb\Gamma)$ being invertible
in $M_k(\ell^1(\Gamma))$ such that the rows of $A$ are contained in $J$;

\item there exist some $k\in \Nb$, some left $\Zb\Gamma$-submodule $J$ of $(\Zb\Gamma)^k$, and some  $A\in M_k(\Zb\Gamma)$ being invertible in $M_k(\ell^1(\Gamma))$ such that
the left $\Zb\Gamma$-module $\widehat{X}$ is isomorphic to $(\Zb\Gamma)^k/J$ and the rows of $A$ are contained in $J$;

\item  the  left $\Zb\Gamma$-module $\widehat{X}$ is finitely generated, and $\ell^1(\Gamma)\otimes_{\Zb\Gamma}\widehat{X}=\{0\}$.
\end{enumerate}
\end{theorem}

Previously, characterizations of expansiveness for algebraic actions have been obtained in various special cases, such as
the case $\Gamma=\Zb^d$ for $d\in \Nb$ \cite{Schmidt90}, the case $\Gamma$ is abelian \cite{Miles}, the case $\widehat{X}=\Zb\Gamma/J$ for a finitely generated left ideal $J$ of $\Zb\Gamma$ \cite{ER}, the case $X$ is connected and finite-dimensional \cite{Bhattacharya},  and the case $\widehat{X}=\Zb\Gamma/\Zb \Gamma f$ for some $f\in \Zb\Gamma$ \cite{DS}.

When $\Gamma$ is abelian, we have the following characterization of expansive algebraic actions.

\begin{corollary} \label{C-characterization of expansive for abelian}
Suppose that $\Gamma$ is abelian. Let $\Gamma$ act on a compact abelian group by automorphisms. Then the action is expansive if and only if $\widehat{X}$ is a finitely generated $\Zb\Gamma$-module and there exists $f\in \Zb\Gamma$ being invertible in $\ell^1(\Gamma)$ such that $f\widehat{X}=\{0\}$.
\end{corollary}
\begin{proof} Suppose that the action is expansive. By Theorem~\ref{T-algebraic characterization of expansive} we can write the $\Zb\Gamma$-module $\widehat{X}$ as $(\Zb\Gamma)^k/J$ for some $k\in \Nb$ and some left $\Zb\Gamma$-submodule $J$ of $(\Zb\Gamma)^k$,
and find $A\in M_k(\Zb\Gamma)$ being invertible in $M_k(\ell^1(\Gamma))$ such that the rows of $A$ are contained in $J$.

Since $\Gamma$ is abelian, $\ell^1(\Gamma)$ is a commutative algebra. Thus we can talk about the determinant $\det(B)\in \ell^1(\Gamma)$ for any $B\in M_k(\ell^1(\Gamma))$, and $B$ is invertible in $M_k(\ell^1(\Gamma))$ exactly when $\det(B)$ is invertible in $\ell^1(\Gamma)$. Furthermore, the formula of $A^{-1}$ in terms of $\det (A)$ and the minors of $A$ shows that  $A^{-1}$ is of the form $(\det(A))^{-1}B$ for some $B\in M_k(\Zb\Gamma)$.

For any $b\in (\Zb\Gamma)^k$, we have
$$ \det(A) b=\det(A) b A^{-1}A=b (\det(A) \cdot A^{-1})A=b B A\in J.$$
Thus
$\det(A)\widehat{X}=\{0\}$. This proves the ``only if'' part.

Now suppose that $\widehat{X}$ is a finitely generated $\Zb\Gamma$-module and there exists $f\in \Zb\Gamma$ being invertible in $\ell^1(\Gamma)$ such that $f\widehat{X}=\{0\}$. For any $h\in \ell^1(\Gamma)$ and $a\in \widehat{X}$ we have
$$ h\otimes a=hf^{-1}\otimes fa=0$$
in $\ell^1(\Gamma)\otimes_{\Zb\Gamma}\widehat{X}$. Thus $\ell^1(\Gamma)\otimes_{\Zb\Gamma}\widehat{X}=\{0\}$. By Theorem~\ref{T-algebraic characterization of expansive} the action is expansive. This proves the ``if'' part.
\end{proof}

From  Corollary~\ref{C-characterization of expansive for abelian} we have the following consequence, which can also be deduced from \cite[Theorem 3.3]{Miles}.

\begin{corollary} \label{C-abelian expansive quotient}
Suppose that $\Gamma$ is abelian. If $\Gamma$ acts expansively on a compact abelian group $X$ by automorphisms and $Y$ is a closed $\Gamma$-invariant subgroup of $X$ with $\widehat{X/Y}$ being a finitely generated $\Zb\Gamma$-module, then the induced $\Gamma$-action on $X/Y$ is expansive.
\end{corollary}

\begin{remark} \label{R-quotient not expansive}
Corollary~\ref{C-abelian expansive quotient} fails for any $\Gamma$ containing a free subgroup with $2$ generators. Indeed, if elements $s$ and $t$ of $\Gamma$ generate a free subgroup $\Gamma'$, then
$\Rb\Gamma (2-s)\cap \Rb\Gamma (2-t)=\{0\}$. This can be proved first for the case $\Gamma'=\Gamma$ using the arguments in the proof of \cite[Corollary 10.3.7.(iv)]{Passman}, and then extended to the general case using the fact that $\Rb\Gamma$ is a free right $\Rb\Gamma'$-module. It follows that the left $\Zb\Gamma$-submodule $\fM$ of $\Zb\Gamma/\Zb\Gamma (2-t)$ generated by $2-s+\Zb\Gamma (2-t)$ is isomorphic to $\Zb\Gamma$. Note that $2-t$ is invertible in $\ell^1(\Gamma)$, thus the canonical action of
$\Gamma$ on $X=\widehat{\Zb\Gamma/\Zb\Gamma (2-t)}$ is expansive. But $\widehat{\fM}$ is of the form $X/Y$ for some closed invariant subgroup $Y$ of $X$, and the canonical action of $\Gamma$ on $\widehat{\fM}=\widehat{\Zb\Gamma}=(\Rb/\Zb)^\Gamma$ is not expansive.
\end{remark}

For $\Gamma=\Zb^d$ with $d\in\Nb$, the group ring $\Zb \Gamma$ is the ring of Laurent polynomials with $d$ variables and integer coefficients, which is Noetherian \cite[Corollary IV.4.2]{Lang}. Thus, from Corollary~\ref{C-abelian expansive quotient} we obtain a new proof of the following result of Klaus Schmidt \cite[Corollary 3.11]{Schmidt90}:

\begin{corollary} \label{C-expansive quotient}
Let $d\in \Nb$. If $\Zb^d$ acts expansively on a compact abelian group $X$ by automorphisms and $Y$ is a closed $\Zb^d$-invariant subgroup of $X$, then the induced $\Zb^d$-action on $X/Y$ is expansive.
\end{corollary}

We remark that, based on Corollary~\ref{C-expansive quotient}, Schmidt also showed that if  $\Zb^d$ acts expansively on a compact (not necessarily abelian) group $X$ by automorphisms and $Y$ is a closed $\Zb^d$-invariant normal subgroup of $X$, then the induced $\Zb^d$-action on $X/Y$ is expansive \cite[Corollary 6.15]{Schmidt}.

Recall that a unital ring $R$ is said to be {\it left Noetherian} if every left ideal of $R$ is finitely generated. In general, we have

\begin{conjecture} \label{Conj-expansive quotient}
Suppose that $\Gamma$ is amenable and $\Zb\Gamma$ is left Noetherian. If  $\Gamma$ acts expansively on a compact group $X$ by automorphisms and $Y$ is a closed $\Gamma$-invariant normal subgroup of $X$, then the induced $\Gamma$-action on $X/Y$ is expansive.
\end{conjecture}

The proof of \cite[Corollary 6.15]{Schmidt} shows that, in order to prove Conjecture~\ref{Conj-expansive quotient}, it suffices to consider the case $X$ is abelian.

We start preparation for the proof of  Theorem~\ref{T-algebraic characterization of expansive}.
We describe first a class of expansive algebraic actions, which are analogues of the expansive principal algebraic actions studied in \cite{DS}. Conditions (2) and (3) of Theorem~\ref{T-algebraic characterization of expansive} state that these actions are the largest expansive algebraic actions in the sense that every expansive algebraic action is the restriction of one of these actions to a closed invariant subgroup.

\begin{lemma} \label{L-key expansive}
Let $k\in \Nb$, and  $A\in M_k(\Zb\Gamma)$ be invertible in $M_k(\ell^1(\Gamma))$.
Then the canonical action
of $\Gamma$ on $X_A:=\widehat{(\Zb\Gamma)^k/(\Zb\Gamma)^kA}$ is expansive.
\end{lemma}
\begin{proof}
Let $x\in X_A$ be nonzero. Take $y\in ([-1/2, 1/2]^k)^\Gamma\subseteq (\ell^{\infty}(\Gamma))^k$ with $P(y)=x$, where $P$ denotes the canonical map
 $(\ell^\infty(\Gamma))^k\rightarrow ((\Rb/\Zb)^k)^\Gamma$. Then $y\neq 0$, and $yA^*\in \ell^{\infty}(\Gamma, \Zb^k)$. Since $A$ is invertible in $M_k(\ell^1(\Gamma))$, so is $A^*$. Thus we have $yA^*\neq 0$, and hence $\|yA^*\|_\infty\ge 1$.
Note that $\|yA^*\|_\infty\le \|y\|_\infty \cdot \|A^*\|_1=\|y\|_\infty\cdot \|A\|_1$. Therefore $\|y\|_\infty\ge \|A\|_1^{-1}$.
Then $\sup_{s\in \Gamma}\rho_{\infty}(x_s, 0_{(\Rb/\Zb)^k})\ge \|A\|_1^{-1}$, where $\rho_\infty$ is the metric on $(\Rb/\Zb)^k$ defined in \eqref{E-metric on torus}, and hence the canonical action of $\Gamma$ on $X_A$ is expansive.
\end{proof}

The following lemma is inspired by a question raised by Doug Lind and Klaus Schmidt,
and uses the technique of \cite[Theorem 3.2]{DS}.

\begin{lemma} \label{L-kernel}
Let $k, n\in \Nb$, and $B\in M_{n\times k}(\Zb\Gamma)$.
Then the following are equivalent:
\begin{enumerate}
\item  the canonical action of $\Gamma$ on $\widehat{(\Zb\Gamma)^k/(\Zb\Gamma)^nB}$ is expansive;

\item the linear map $\varphi: (\ell^\infty(\Gamma))^k\rightarrow (\ell^\infty(\Gamma))^n$ sending $y$ to $yB^*$ is injective;

\item the linear map $\psi: (\ell^1(\Gamma))^n\rightarrow (\ell^1(\Gamma))^k$ sending $z$ to $zB$ has dense image;

\item there exists $A\in M_k(\Zb\Gamma)$ being invertible in $M_k(\ell^1(\Gamma))$ such that the rows of $A$ are contained in $(\Zb\Gamma)^nB$.
\end{enumerate}
\end{lemma}
\begin{proof} (1)$\Rightarrow$(2):
Let $y\in (\ell^\infty(\Gamma))^k$ with $yB^*=0$. Then $(\lambda y)B^*=0$ for every $\lambda \in \Rb$. If we denote by $P$ the canonical map
$(\ell^\infty(\Gamma))^k\rightarrow ((\Rb/\Zb)^k)^\Gamma$, then $P(\lambda y)\in \widehat{(\Zb\Gamma)^k/(\Zb\Gamma)^nB}$ for every $\lambda \in \Rb$. Since
the canonical action of $\Gamma$ on $\widehat{(\Zb\Gamma)^k/(\Zb\Gamma)^nB}$ is expansive, we conclude that $y=0$.

(2)$\Rightarrow$(3):
Note that we can identify $(\ell^\infty(\Gamma))^n$ and $(\ell^\infty(\Gamma))^k$ with the dual spaces of $(\ell^1(\Gamma))^n$ and $(\ell^1(\Gamma))^k$ respectively in a canonical way. For instance, for $y=(y_1, \dots, y_k)\in (\ell^{\infty}(\Gamma))^k$
and $z=(z_1, \dots, z_k)\in (\ell^1(\Gamma))^k$, the pairing
is given by
\begin{align} \label{E-pairing}
 \left<y, z\right>:=(yz^*)_{e_\Gamma}=(\sum_{1\le j\le k}y_jz^*_j)_{e_\Gamma}.
\end{align}
Under such identification $\varphi$ is the dual of $\psi$. If the image of $\psi$
is not dense in $(\ell^1(\Gamma))^k$, then by the Hahn-Banach theorem we can find some nonzero bounded linear functional $y$ on $(\ell^1(\Gamma))^k$, vanishing
at the image of $\psi$. Since $y$ is an element of $(\ell^\infty(\Gamma))^k$, this means  $yB^*=0$, which contradicts (2). Thus the image of $\psi$ is dense in $(\ell^1(\Gamma))^k$.

(3)$\Rightarrow$(4):
Since $M_k(\ell^1(\Gamma))$ is a Banach algebra, the set of invertible matrices in $M_k(\ell^1(\Gamma))$ is open. Thus we can find $z_1, \dots, z_k\in (\ell^1(\Gamma))^n$ such
that the matrix
\begin{eqnarray*}
\begin{pmatrix} z_1 \\ \vdots \\ z_k
\end{pmatrix}\cdot B
\end{eqnarray*}
in $M_k(\ell^1(\Gamma))$ is close enough to the identity matrix for being invertible. Since $\Qb\Gamma$ is dense in $\ell^1(\Gamma)$, we may require that $z_1, \dots, z_k\in (\Qb \Gamma)^n$. Then we can find some $N\in \Nb$ such that $Nz_1, \dots, Nz_k\in (\Zb \Gamma)^k$.
Now we can set
\begin{eqnarray*}
A=N\begin{pmatrix} z_1 \\ \vdots \\ z_k
\end{pmatrix}\cdot B.
\end{eqnarray*}

(4)$\Rightarrow$(1):  Define $X_A$ as in Lemma~\ref{L-key expansive}. Then $\widehat{(\Zb\Gamma)^k/(\Zb\Gamma)^nB}\subseteq X_A$. Since the canonical action of $\Gamma$ on $X_A$ is expansive by
Lemma~\ref{L-key expansive}, its restriction on $\widehat{(\Zb\Gamma)^k/(\Zb\Gamma)^nB}$ is also expansive.
\end{proof}

\begin{lemma} \label{L-torsion}
Let $k\in \Nb$ and $J$ be a left $\Zb\Gamma$-submodule of $(\Zb\Gamma)^k$. Then the following are equivalent:
\begin{enumerate}
\item there exists $A\in M_k(\Zb\Gamma)$ being invertible in $M_k(\ell^1(\Gamma))$ such that the rows of $A$ are contained in $J$;

\item $\ell^1(\Gamma)\otimes_{\Zb\Gamma}((\Zb\Gamma)^k/J)=\{0\}$;

\item the left $\ell^1(\Gamma)$-module $(\ell^1(\Gamma))^k$ is generated by $J$.
\end{enumerate}
\end{lemma}
\begin{proof}
(1)$\Rightarrow$(2):
Write the $j$-th row of $A$ as $g_j$.
Denote by $e_1, \dots, e_k$ the standard basis of $(\Zb\Gamma)^k$. Let $w\in \ell^1(\Gamma)$, $1\le i\le k$, and $f\in \Zb\Gamma$. Then in $\ell^1(\Gamma)\otimes_{\Zb\Gamma}((\Zb\Gamma)^k/J)$ we have
\begin{align*}
w\otimes (fe_i+J)&=\sum_{1\le m\le k}wf\delta_{i, m}\otimes (e_m+J) \\
&=\sum_{1\le m\le k}wf(\sum_{1\le j\le k}(A^{-1})_{i, j}A_{j, m})\otimes (e_m+J)\\
&=\sum_{1\le j\le k}wf(A^{-1})_{i, j}\otimes \sum_{1\le m\le k}(A_{j, m}e_m+J)\\
&= \sum_{1\le j\le k}wf(A^{-1})_{i, j}\otimes (g_j+J)=0.
\end{align*}

(2)$\Rightarrow$(3): Denote by $J'$ the left $\ell^1(\Gamma)$-submodule of $(\ell^1(\Gamma))^k$ generated by $J$. Consider the map
$\ell^1(\Gamma)\times ((\Zb\Gamma)^k/J)\rightarrow (\ell^1(\Gamma))^k/J'$ sending $(w, f+J)$ to $wf+J'$. It induces a group homomorphism
$\varphi: \ell^1(\Gamma)\otimes_{\Zb\Gamma} ((\Zb\Gamma)^k/J)\rightarrow (\ell^1(\Gamma))^k/J'$ defined by $\varphi(w\otimes (f+J))=wf+J'$
for all $w\in \ell^1(\Gamma)$ and $f\in (\Zb\Gamma)^k$. Clearly $\varphi$ is surjective. Since $\ell^1(\Gamma)\otimes_{\Zb\Gamma} ((\Zb\Gamma)^k/J)=\{0\}$,
we conclude that $(\ell^1(\Gamma))^k/J'=\{0\}$. That is, $(\ell^1(\Gamma))^k=J'$.

(3)$\Rightarrow$(1): The condition (3) says that every $g\in (\ell^1(\Gamma))^k$ can be written as $a_1f_1+\dots +a_nf_n$ for some
$n\in \Nb$, $a_1, \dots, a_n\in \ell^1(\Gamma)$, and $f_1, \dots, f_n\in J$. Taking $g$ to the standard basis of the $\ell^1(\Gamma)$-module
$(\ell^1(\Gamma))^k$,
we find some $B\in M_{k\times n}(\ell^1(\Gamma))$ for some $n\in \Nb$, and some $f_1, \dots, f_n\in J$ such that
$BC$ is the identity matrix in $M_k(\ell^1(\Gamma))$, where
\begin{eqnarray*}
C=\begin{pmatrix} f_1 \\ \vdots \\ f_n
\end{pmatrix}.
\end{eqnarray*}
Since $M_k(\ell^1(\Gamma))$ is a Banach algebra, the set of invertible matrices in $M_k(\ell^1(\Gamma))$ is open.
Approximating $B$ by some $B'\in M_{k\times n}(\Qb\Gamma)$, we may assume that $B'C$ is invertible in $M_k(\ell^1(\Gamma))$.
Take some suitable $N\in \Nb$ such that $NB'\in M_{k\times n}(\Zb \Gamma)$. Then we may set $A=(NB')C$.
\end{proof}

We are ready to prove Theorem~\ref{T-algebraic characterization of expansive}.

\begin{proof}[Proof of Theorem~\ref{T-algebraic characterization of expansive}]
(1)$\Rightarrow$(2): By Proposition~\ref{P-expansive to subshift} we know that $\widehat{X}$ is a finitely generated left $\Zb\Gamma$-module.

Identify $\widehat{X}$ with $(\Zb\Gamma)^k/J$ for some $k\in \Nb$ and some left $\Zb\Gamma$-submodule $J$ of $(\Zb\Gamma)^k$.
Denote by $\Omega$ the set of finitely generated $\Zb\Gamma$-submodules of $J$. We claim that there exists some $\omega \in \Omega$ such that the canonical action of $\Gamma$ on $\widehat{(\Zb\Gamma)^k/\omega}$ is expansive.
Suppose that this fails. Let $\omega\in \Omega$.
By Lemma~\ref{L-kernel} we can find some nonzero $y^\omega=(y^\omega_1, \dots, y^\omega_k)\in (\ell^\infty(\Gamma))^k$
such that $\left<y^\omega, f\right>=0$ for every $f$ in a finite generating set $W$ of $\omega$, where the pairing $\left<y^\omega, f\right>$ is given by the equation (\ref{E-pairing}).
Since $\omega$ is a left $\Zb\Gamma$-module, we get $y^\omega f^*=0$ for all $f\in W$, and hence $y^\omega f^*=0$ for all $f\in \omega$.
Replacing $y^\omega$ by $s y^\omega$ for some $s \in \Gamma$, we may assume that
$\|y^\omega_{e_{\Gamma}}\|_\infty\ge \|y^\omega\|_\infty/2$. Replacing  $y^\omega$ by $\lambda y^\omega$ for some $\lambda \in \Rb$, we may assume that $\|y^\omega_{e_{\Gamma}}\|_\infty=1/2$.
Then $y^\omega\in ([-1,1]^k)^\Gamma$. Note that $\Omega$ is directed by inclusion.
Since the space $([-1,1]^k)^\Gamma$ is compact under the product topology, we can take a limit point $z$ of the net $(y^\omega)_{\omega \in \Omega}$.
Then $zf^*=0$ for every $f\in J$, and hence $P(\lambda z)\in \widehat{(\Zb\Gamma)^k/J}$ for every $\lambda \in \Rb$, where $P$ denotes the canonical map from $(\ell^{\infty}(\Gamma))^k$ to $((\Rb/\Zb)^\Gamma)^k$.
We also have $\|z_{e_{\Gamma}}\|_\infty=1/2$, and hence $P(\lambda z)\neq 0$ for all $0<\lambda\le 1$. This contradicts the expansiveness of  the canonical action of $\Gamma$ on $\widehat{(\Zb\Gamma)^k/J}$. Thus our claim holds.

So let $f_1, \dots, f_n\in J$ such that the canonical action of $\Gamma$ on $\widehat{(\Zb\Gamma)^k/\omega}$ is expansive for
$\omega=\Zb\Gamma f_1+\dots+\Zb\Gamma f_n$. By Lemma~\ref{L-kernel} we can find some
$A\in M_k(\Zb\Gamma)$ being invertible in $M_k(\ell^1(\Gamma))$ such that the rows of $A$ are contained in $\omega$, and hence in $J$.
This proves (2).

(2)$\Rightarrow$(3) is trivial.

(3)$\Rightarrow$(1) can be proved as in the proof of (4)$\Rightarrow$(1) of Lemma~\ref{L-kernel}.

(3)$\Leftrightarrow$(4) follows from Lemma~\ref{L-torsion}.
\end{proof}

\section{P-expansiveness} \label{S-p version}

In this section we study $p$-expansiveness, a weak version of expansiveness. Throughout this section $\Gamma$ will be a countable discrete group.

\begin{notation} \label{N-function}
Let $\Gamma$ act on a compact abelian group $X$ by automorphisms. For every $x\in X$ and $\varphi\in \widehat{X}$, we denote by $\Psi_{x, \varphi}$ the function
on $\Gamma$ defined by
$$ \Psi_{x, \varphi}(s)=\left<sx, \varphi\right>-1, \quad s\in \Gamma,$$
where $\Tb$ denotes the unit circle in $\Cb$ and $\left< \cdot, \cdot\right>:X\times \widehat{X}\rightarrow \Tb$ denotes the canonical pairing between $X$ and $\widehat{X}$.
\end{notation}

\begin{definition} \label{D-p-expansive}
Let $1\le p\le +\infty$. We say that an action of $\Gamma$ on a compact abelian group $X$ by automorphisms is {\it $p$-expansive} if there exist a finite subset $W$ of $\widehat{X}$ and $\varepsilon>0$ such that $e_X$ is the only point $x$ in $X$ satisfying
$$ \sum_{\varphi \in W}\|\Psi_{x, \varphi}\|_p<\varepsilon.$$
\end{definition}

In the following we collect some basic properties of $p$-expansiveness. Assertion (4) below justifies our terminology of $p$-expansiveness.
Recall that for a unital ring $R$, a left $R$-module $\fM$ is called {\it finitely presented} if $\fM=R^k/J$ for some $k\in \Nb$ and some finitely generated left $R$-submodule $J$ of $R^k$ \cite[Definition 4.25]{Lam}. If $R$ is left Noetherian, then every finitely generated left $R$-module is finitely presented \cite[Proposition 4.29]{Lam}.

\begin{proposition} \label{P-basic p-expansive}
Let $\alpha$ be an action of $\Gamma$ on a compact abelian group by automorphisms. Let $1\le p\le +\infty$. Then the following hold:
\begin{enumerate}
\item If $\alpha$ is $p$-expansive, then it is $q$-expansive for all $1\le q\le p$.

\item If $\alpha$ is $p$-expansive, then $\widehat{X}$ is a finitely generated left $\Zb\Gamma$-module.

\item If $\alpha$ is $p$-expansive, then for any finite subset $W$ of $\widehat{X}$ generating $\widehat{X}$ as a left $\Zb\Gamma$-module, there exists $\varepsilon>0$ such that $e_X$ is the only point $x$ in $X$ satisfying
$ \sum_{\varphi \in W}\|\Psi_{x, \varphi}\|_p<\varepsilon$.

\item $\alpha$ is $\infty$-expansive if and only if it is expansive.

\item Suppose that $\widehat{X}$ is a finitely presented left $\Zb\Gamma$-module. Write $\widehat{X}$ as $(\Zb\Gamma)^k/(\Zb\Gamma)^nA$ for some $k, n\in \Nb$ and
$A\in M_{n\times k}(\Zb\Gamma)$. Then $\alpha$ is $p$-expansive if and only if
    the linear map $(\ell^p(\Gamma))^k\rightarrow (\ell^p(\Gamma))^n$ sending $a$ to $aA^*$ is injective.
\end{enumerate}
\end{proposition}
\begin{proof}
(1). This follows from the fact that for any $f\in \ell^q(\Gamma)$ with $\|f\|_q\le 1$, one has $\|f\|_p^p\le \|f\|_q^q$ when $p<+\infty$ and $\|f\|_\infty\le \|f\|_q$ when $p=+\infty$.

(2). Suppose that $\alpha$ is $p$-expansive. Let $W$ and $\varepsilon$ be as in Definition~\ref{D-p-expansive}. Denote by $G$ the $\Zb\Gamma$-submodule of $\widehat{X}$ generated by $W$. If $x\in X$ satisfies $\left<x, \psi\right>=1$ for all $\psi \in G$, then $\Psi_{x, \varphi}=0$ for all $\varphi \in W$ and hence $x=e_X$.
By Pontryagin duality we have $G=\widehat{X}$.

(3). This follows from the fact that for any $x\in X$, $\varphi, \psi\in \widehat{X}$, and $a, b\in \Zb\Gamma$ one has
\begin{align} \label{E-basic p-expansive 2}
\|\Psi_{x, a\varphi+b\psi}\|_p\le \|a\|_1\|\Psi_{x, \varphi}\|_p+\|b\|_1\|\Psi_{x, \psi}\|_p.
\end{align}

To prove the assertions (4) and (5), we observe a general fact first. Suppose that $\widehat{X}$ is a finitely generated left $\Zb\Gamma$-module, and write $\widehat{X}$ as $(\Zb\Gamma)^k/J$ for some $k\in \Nb$ and some left $\Zb\Gamma$-submodule $J$ of $(\Zb\Gamma)^k$.
For each $x\in X$ denote by $\Phi_x$ the function $s\mapsto \rho_\infty(x_s, 0_{(\Rb/\Zb)^k})$ on $\Gamma$, where $\rho_\infty$ is the canonical metric on $(\Rb/\Zb)^k$ defined in
the equation \eqref{E-metric on torus}. Denote by $e_1, \dots, e_k$ the standard basis
of $(\Zb\Gamma)^k$. Set $W=\{e_1+J, \dots, e_k+J\}$.
Note that there exists $C>0$ such that
$$ C|t|\le |e^{2\pi i t}-1|\le C^{-1}|t|$$
for all $t\in [-1/2, 1/2]$. It follows that there exists $C_1>0$ such that
\begin{align} \label{E-basic p-expansive}
 C_1\|\Phi_x\|_p\le \sum_{\varphi \in W}\|\Psi_{x, \varphi}\|_p\le C_1^{-1}\|\Phi_x\|_p
 \end{align}
for all $x\in X$.

(4). By the assertion (2) and Proposition~\ref{P-expansive to subshift} both $\infty$-expansiveness and expansiveness imply that $\widehat{X}$ is a finitely generated left $\Zb\Gamma$-module.
Now the assertion (4) follows from the $p=\infty$ case of the inequalities \eqref{E-basic p-expansive} and the assertion (3).

(5). Denote by $P$ the canonical map $\ell^\infty(\Gamma, \Rb^k)\rightarrow ((\Rb/\Zb)^k)^\Gamma$.
Suppose that the linear map $(\ell^p(\Gamma))^k\rightarrow (\ell^p(\Gamma))^n$ sending $a$ to $aA^*$ is not injective. Take a nonzero $a\in (\ell^p(\Gamma))^k$ with $aA^*=0$. Then for any $\lambda\in \Rb$, one has $\lambda a A^*=0$ and hence $P(\lambda a)\in X$. When $\lambda \to 0$, one has $\|\Phi_{P(\lambda a)}\|_p\to 0$, and hence $\sum_{\varphi \in W}\|\Psi_{P(\lambda a), \varphi}\|_p\to 0$. Since $a\neq 0$, when $|\lambda|$ is sufficiently small and nonzero,
$P(\lambda a)\neq e_X$. Thus $\alpha$ is not $p$-expansive. This proves the ``only if'' part.

Suppose that $\alpha$ is not $p$-expansive. Then we can find a nonzero $x\in X$ such that $\|\Phi_x\|_p<\|A\|_1^{-1}$. Take a lift $\tilde{x}$ of $x$ in
$\ell^p(\Gamma, \Rb^k)$ with $\|\tilde{x}\|_p=\|\Phi_x\|_p$. Then $\tilde{x}A^*\in \ell^\infty(\Gamma, \Zb^n)$, and
$$\|\tilde{x}A^*\|_\infty\le \|\tilde{x}A^*\|_p\le \|\tilde{x}\|_p\|A^*\|_1=\|\Phi_x\|_p\|A\|_1<1.$$
It follows that $\tilde{x}A^*=0$. Since $P(\tilde{x})=x$ is nonzero, $\tilde{x}\neq 0$. This proves the ``if'' part.
\end{proof}

\begin{notation} \label{N-principal}
For $f\in \Zb\Gamma$, we denote by $\alpha_f$ the canonical $\Gamma$-action on $X_f:=\widehat{\Zb\Gamma/\Zb\Gamma f}$. We also denote by $C_0(\Gamma)$ the space
of $\Cb$-valued functions on $\Gamma$ vanishing at infinity.
\end{notation}

By Proposition~\ref{P-basic p-expansive}.(5), for any $f\in \Zb\Gamma$ and $1\le p\le +\infty$, the action $\alpha_f$ is $p$-expansive if and only if for any nonzero $g\in \ell^p(\Gamma)$ one has $gf^*\neq 0$. The latter is related to the {\it analytic zero divisor conjecture} and has been studied extensively in \cite{Linnell91, Linnell92, Linnell93, Linnell98, LP, Puls}.

\begin{example} \label{Ex-elementary amenable}
Recall that the class of {\it elementary amenable groups} is the smallest class of groups containing all finite groups and all abelian groups and is closed under taking subgroups, quotient groups, extensions, and inductive limits. Suppose that $\Gamma$ is torsion free and elementary amenable. Then for any nonzero $f\in \Cb\Gamma$ and nonzero $g\in \ell^2_\Cb(\Gamma)$, one has $gf\neq 0$ \cite[Theorem 2]{Linnell91}. Thus $\alpha_f$ is $2$-expansive for every nonzero $f\in \Zb\Gamma$. On the other hand, by \cite[Theorem 3.2]{DS}, for $f\in \Zb\Gamma$, $\alpha_f$ is expansive exactly when $f$ is invertible in $\ell^1(\Gamma)$.
\end{example}

\begin{example} \label{Ex- p-expansive not expansive}
Suppose that $s\in \Gamma$ has infinite order. Denote by $\Gamma'$ the subgroup of $\Gamma$ generated by $s$.
For any nonzero $f\in \Cb\Gamma'$ and nonzero $g\in C_0(\Gamma')$, one has $gf\neq 0$ \cite[Theorem 5.1]{Linnell98}.
Using the cosets decomposition of $\Gamma$, it follows that for any nonzero $f\in \Cb\Gamma'$ and nonzero $g\in C_0(\Gamma)$, one has $gf\neq 0$.
Thus,
for any nonzero $f\in \Zb\Gamma'$ the action $\alpha_f$ is $p$-expansive for all $1\le p<+\infty$.
Note that $s-1$ is not invertible in $\ell^1(\Gamma)$, thus $\alpha_{s-1}$ is not expansive by  \cite[Theorem 3.2]{DS}.
\end{example}

\begin{example} \label{Ex- Z^d}
Let $\Gamma=\Zb^d$ for some $d\in \Nb$ with $d\ge 2$. One may identify $\Zb\Gamma$ with the ring $\Zb[u_1^{\pm 1},\dots,  u_d^{\pm 1}]$ of Laurent polynomials with integer coefficients
in the variables $u_1, \dots, u_d$. For any nonzero $f\in \Cb\Gamma$ and nonzero $g\in \ell^{\frac{2d}{d-1}}(\Gamma)$, one has $gf\neq 0$ \cite[Theorem 2.1]{LP}. Thus $\alpha_f$ is $\frac{2d}{d-1}$-expansive for every nonzero $f\in \Zb\Gamma$. Set $h=2d-1-\sum_{j=1}^d(u_j+u_j^{-1})\in \Zb\Gamma$. Then there exists a nonzero $g\in \ell^\infty(\Gamma)$ such that $gh=0$ and $g\in \ell^p(\Gamma)$ for all $\frac{2d}{d-1}<p\le +\infty$ \cite[Section 7]{Puls}. Thus, for any $\frac{2d}{d-1}<p\le +\infty$, the action $\alpha_h$ is not $p$-expansive.
\end{example}

\begin{example}  \label{Ex-free group}
Let $\Gamma$ be a free group with canonical generators $s_1, \dots, s_d$, for $d\ge 2$.
For any nonzero $f\in \Cb\Gamma$, one has $gf\neq 0$ for every nonzero $g\in \ell^2_\Cb(\Gamma)$ \cite[Theorem 2]{Linnell92}. Thus
the action $\alpha_f$ is $2$-expansive for every nonzero $f\in \Zb\Gamma$. Endow $\Gamma$ with the word length with respect to $s_1, \dots, s_d$.
For each $n\in \Zb_{\ge 0}$ denote by $\chi_n$ the sum of the elements in $\Gamma$ with length $n$.
Set $g=\sum_{n=0}^\infty(-1)^n(2d-1)^{-n}\chi_{2n}\in \ell^\infty(\Gamma)$.
Then $g\chi_1=0$ and $g\in \ell^p(\Gamma)$ for all $2<p\le +\infty$ \cite[Example 6.5]{LP}. Thus, for any $2<p\le +\infty$,
the action $\alpha_{\chi_1}$ is not $p$-expansive.
\end{example}

\begin{example} \label{Ex- Z^d uniform}
Let $\Gamma=\Zb^d$ for some $d\in \Nb$ with $d\ge 2$. Denote by $P$ the natural projection $\Rb^d\rightarrow (\Rb/\Zb)^d=\widehat{\Gamma}$. For each $f\in \Cb\Gamma$, via the pairing between $\Gamma$ and $\widehat{\Gamma}$ we may think of $f$ as a function on $\widehat{\Gamma}$. Denote by $Z(f)$ the zero set of $f$ as a function on $\widehat{\Gamma}$. For $f=\sum_{s\in \Gamma}\lambda_ss\in \Cb\Gamma$, set $\bar{f}=\sum_{s\in \Gamma}\overline{\lambda_s}s$. The set $Z(\bar{f})$ is contained in the image of a finite union of  hyperplanes in $\Rb^d$ under $P$ if and only if $gf\neq 0$ for all nonzero $g\in C_0(\Gamma)$ \cite[Theorem 2.2]{LP}.
Thus, for $f\in \Zb\Gamma$, if $Z(f)$ is contained in the image of a finite union of  hyperplanes in $\Rb^d$ under $P$ (for instance when $Z(f)$ is a finite set), then so is $Z(f^*)$, and hence $\alpha_f$ is $p$-expansive for all $1\le p<+\infty$.
\end{example}

A finitely generated elementary amenable group either contains a nilpotent subgroup with finite index or has a free subsemigroup with two generators \cite[Theorem 3.2']{Chou}.

\begin{example} \label{Ex-principal action for free subsemigroup}
Suppose that $\Gamma$ has a free subsemigroup generated by two elements $s$ and $t$. Set $f=\pm 3-(1+s-s^2)t\in \Zb\Gamma$.
 Then $f$ is not invertible in
$\ell^1(\Gamma)$ \cite[Example A.1]{Li}, and thus by \cite[Theorem 3.2]{DS} $\alpha_f$ is not expansive. An argument similar to that in \cite[Example 7.2]{LS99}
shows that for any $g\in C_0(\Gamma)$, if $gf^*=0$, then $g=0$. Thus
$\alpha_f$ is $p$-expansive for all $1\le p<+\infty$.
\end{example}

It is well known that when $\Gamma$ is amenable, every continuous expansive $\Gamma$-action on a compact space has finite topological entropy (the case $\Gamma=\Zb$ is proved in \cite[Theorem 7.11]{Walters};  the proof there also works for general countable amenable groups $\Gamma$).
Next we show that
$1$-expansiveness and $2$-expansiveness characterize finite entropy for finitely presented algebraic actions of countable amenable groups.
The definition of entropy is recalled in Section~\ref{SS-entropy}.

\begin{theorem} \label{T-finite entropy vs 1-expansive}
Suppose that $\Gamma$ is  amenable. Let $\alpha$ be an action of $\Gamma$ on a compact abelian group $X$ by automorphisms.
Suppose that $\widehat{X}$ is a finitely presented left $\Zb\Gamma$-module, and write $\widehat{X}$ as $(\Zb\Gamma)^k/(\Zb\Gamma)^nA$ for
some $k, n\in \Nb$ and $A\in M_{n\times k}(\Zb\Gamma)$. Then the following are
equivalent:
\begin{enumerate}
\item $\rh(X)<+\infty$.

\item $\alpha$ is $1$-expansive.

\item $\alpha$ is $2$-expansive.

\item The linear map $(\Rb \Gamma)^k\rightarrow (\Rb\Gamma)^n$ sending $a$ to $aA^*$ is injective.

\item The additive map $(\Zb \Gamma)^k\rightarrow (\Zb\Gamma)^n$ sending $a$ to $aA^*$ is injective.
\end{enumerate}
\end{theorem}

We need some preparation for the proof of Theorem~\ref{T-finite entropy vs 1-expansive}.
First we need the following result of Elek \cite{Elek03}. He assumed $\Gamma$ to be finitely generated and $k=1$, which are unnecessary.

\begin{theorem} \label{T-Elek}
 Suppose that $\Gamma$ is amenable. If $A\in M_k(\Cb \Gamma)$ for some $k\in \Nb$ and $aA=0$ for some nonzero $a\in (\ell^2_\Cb(\Gamma))^k$, then $bA=0$ for some nonzero $b\in (\Cb\Gamma)^k$.
\end{theorem}

We remark that the proof of Theorem~\ref{T-Elek} uses the group von Neumann algebra of $\Gamma$.

\begin{lemma} \label{L-separated subset}
Let $K$ and $F$ be nonempty finite subsets of $\Gamma$. Then there exists a finite subset $F_1$ of $F$ with
$\frac{|F_1|}{|F|}\ge \frac{1}{2|K|+1}$ and $((F_1F_1^{-1})\setminus \{e_{\Gamma}\})\subseteq \Gamma\setminus K$.
\end{lemma}
\begin{proof} Let $F_1$ be a maximal subset of $F$ subject to the condition that $s'\not\in Ks$ for all distinct $s, s'\in F_1$.
Then $F\subseteq (\{e_\Gamma\}\cup K\cup K^{-1})F_1$. Thus $|F|\le |\{e_\Gamma\}\cup K\cup K^{-1}|\cdot |F_1| \le (2|K|+1)|F_1|$.
\end{proof}

\begin{lemma} \label{L-kernel to infinite entropy}
Suppose that $\Gamma$ is amenable.
Let $\Gamma$ act on a compact abelian group $X$ by automorphisms. Suppose that $\widehat{X}=(\Zb\Gamma)^k/J$ for some $k\in \Nb$ and some left $\Zb\Gamma$-submodule $J$ of $(\Zb\Gamma)^k$, and that there exists $0\neq a\in (\Rb \Gamma)^k$ satisfying $ab^*=0$ for all $b\in J$. Then
$\rh(X)=+\infty$.
\end{lemma}
\begin{proof} Denote by $K$ the support of $a$ as a $\Rb^k$-valued function on $\Gamma$.
Fix a compatible metric $\rho$ on $X$. Denote by $P$ the natural projection $\ell^\infty(\Gamma, \Rb^k)\rightarrow ((\Rb/\Zb)^k)^\Gamma$.
Then $P(\lambda a)\in X$ for every $\lambda \in \Rb$. If $\lambda_1, \lambda_2\in (0, 1/\|a\|_\infty)$ are distinct, then $P(\lambda_1a)\neq P(\lambda_2 a)$.

Let $M\in \Nb$. Take distinct $\lambda_1, \dots, \lambda_M\in (0, 1/\|a\|_\infty)$. For each $1\le j\le M$, set
$$Y_j=\{x\in X: x_s=(P(\lambda_ja))_s \mbox{ for all } s\in K\}.$$ Then the sets $Y_1,\dots, Y_M$ are pairwise disjoint closed subsets of $X$.
Thus we can find $\varepsilon>0$ such
that if $x\in Y_i$ and $y\in Y_j$ for some $1\le i<j\le M$, then $\rho(x, y)>\varepsilon$.

Let $F$ be a nonempty finite subset of $\Gamma$. By Lemma~\ref{L-separated subset} we can find a
finite subset $F_1$ of $F$ with
$\frac{|F_1|}{|F|}\ge \frac{1}{2|KK^{-1}|+1}$ and $((F_1F_1^{-1})\setminus \{e_{\Gamma}\})\subseteq \Gamma\setminus KK^{-1}$.
Then the sets $s^{-1}K$ for $s\in F_1$ are pairwise disjoint. For each $\sigma\in \{1, \dots, M\}^{F_1}$, set
$$ x_\sigma=\sum_{s\in F_1}s^{-1}P(\lambda_{\sigma(s)}a).$$
Then $sx_\sigma \in Y_{\sigma(s)}$ for every $s\in F_1$. Thus the set $\{x_\sigma: \sigma \in \{1, \dots, M\}^{F_1}\}$ is $(\rho, F, \varepsilon)$-separated.
Therefore
$$ N_{\rho, F, \varepsilon}(X)\ge M^{|F_1|}\ge M^{|F|/(2|KK^{-1}|+1)}.$$
It follows that $\rh(X)\ge \frac{1}{2|KK^{-1}|+1}\log M$. Since $M$ is arbitrary, we get $\rh(X)=+\infty$.
\end{proof}

The following lemma is well known, and can be proved by a simple volume comparison argument (see for example the proof of \cite[Lemma 4.10]{Pisier}).

\begin{lemma} \label{L-volume to covering}
Let $V$ be a finite-dimensional normed space over $\Rb$. Let $\varepsilon>0$. Then any $\varepsilon$-separated subset of the unit ball of $V$ has cardinality at most
$(1+\frac{2}{\varepsilon})^{\dim V}$.
\end{lemma}

We are ready to prove Theorem~\ref{T-finite entropy vs 1-expansive}.

\begin{proof}[Proof of Theorem~\ref{T-finite entropy vs 1-expansive}]
(1)$\Rightarrow$(4) follows from Lemma~\ref{L-kernel to infinite entropy}.

(4)$\Rightarrow$(1): Note that the definitions of $(\rho, F, \varepsilon)$-separated sets and $N_{\rho, F, \varepsilon}(X)$ in Section~\ref{SS-entropy} extend directly to any continuous pseudometric $\rho$ on $X$. The equation \eqref{E-entropy} holds for any continuous pseudometric $\rho$ on $X$ which is {\it dynamically generating} in the sense that for any distinct $x, y\in X$ one has $\rho(sx, sy)>0$ for some $s\in \Gamma$ \cite[Proposition 2.3]{Den} \cite[Theorem 4.2]{Li}.
Recall the canonical metric $\rho_\infty$  on $(\Rb/\Zb)^k$ defined by the equation \eqref{E-metric on torus}.
Define a continuous pseudometric $\rho'$ on $X$ by $\rho'(x, y)=\rho_\infty(x_{e_\Gamma}, y_{e_\Gamma})$ for all $x, y\in X$. Clearly $\rho'$ is dynamically generating. 

Denote by $K$ the union of $\{e_\Gamma\}$ and the support of $A$ as a $M_{n\times k}(\Zb)$-valued function on $\Gamma$. Let $\varepsilon>0$ and $F$ be a nonempty finite subset of $\Gamma$. Take a $(\rho', F, \varepsilon)$-separated subset $E\subseteq X$ with
$|E|=N_{\rho', F, \varepsilon}(X)$. For each $x\in E$ denote by $\tilde{x}$ the element in $([0, 1)^k)^\Gamma$ such that $x$ is the image of $\tilde{x}$ under the natural map $([0, 1)^k)^\Gamma\rightarrow ((\Rb/\Zb)^k)^\Gamma$. Then $\tilde{x}A^*\in \ell^\infty(\Gamma, \Zb^n)$ and $\|\tilde{x}A^*\|_\infty\le \|\tilde{x}\|_\infty \|A^*\|_1\le \|A\|_1$.

For a finite subset $W$ of $\Gamma$, we shall identify $(\Rb^k)^W$ and $(\Rb^n)^W$ as linear subspaces of $\ell^\infty(\Gamma, \Rb^k)=(\ell^\infty(\Gamma))^k$
and $\ell^\infty(\Gamma, \Rb^n)$ respectively naturally, and denote by $p_W$ the restriction maps  $\ell^\infty(\Gamma, \Rb^k)\rightarrow (\Rb^k)^W$ and $\ell^\infty(\Gamma, \Rb^n)\rightarrow (\Rb^n)^W$.

Set $F'=\{s\in F: s^{-1}K\subseteq F^{-1}\}$.  We define a map $\psi: E\rightarrow ((\Zb \cap [-\|A\|_1, \|A\|_1])^n)^{(F')^{-1}}$ sending $x$ to
$p_{(F')^{-1}}(\tilde{x}A^*)$. Let $a\in ((\Zb \cap [-\|A\|_1, \|A\|_1])^n)^{(F')^{-1}}$. We shall give an upper bound for $|\psi^{-1}(a)|$.

Consider the linear map $\xi: (\Rb^k)^{F^{-1}}\rightarrow
 (\Rb^n)^{F^{-1}K^{-1}}$ sending $z$ to $zA^*$. By (4) $\xi$ is injective. Thus
 $$\dim \ker (p_{(F')^{-1}}\circ \xi)\le \dim ((\Rb^n)^{F^{-1}K^{-1}\setminus (F')^{-1}})=n|F^{-1}K^{-1}\setminus (F')^{-1}|.$$

 For each $x\in E$, set $x'=p_{F^{-1}}(\tilde{x})\in (\Rb^k)^{F^{-1}}$.
 Note that $\tilde{x}A^*=x'A^*$ on $(F')^{-1}$. Fix $y\in \psi^{-1}(a)$. Then $(x'-y')A^*=0$ on $(F')^{-1}$ for all $x\in \psi^{-1}(a)$.
 Thus $x'-y'\in \ker (p_{(F')^{-1}}\circ \xi)$ for all $x\in \psi^{-1}(a)$. Since $E$ is $(\rho', F, \varepsilon)$-separated, we see that the set $\{\frac{x'-y'}{2}: x\in \psi^{-1}(a)\}$ is
 $\frac{\varepsilon}{2}$-separated under the $\ell^\infty$-norm, and is clearly contained in the unit ball of $(\Rb^k)^{F^{-1}}$ with respect to the $\ell^\infty$-norm. By Lemma~\ref{L-volume to covering} we have
 $$ |\psi^{-1}(a)|\le (1+\frac{4}{\varepsilon})^{\dim \ker (p_{(F')^{-1}}\circ \xi)}\le (1+\frac{4}{\varepsilon})^{n|F^{-1}K^{-1}\setminus (F')^{-1}|}.$$
 Therefore
 \begin{align*}
  N_{\rho', F, \varepsilon}(X)=|E|&\le (2\|A\|_1+1)^{n|F'|}(1+\frac{4}{\varepsilon})^{n|F^{-1}K^{-1}\setminus (F')^{-1}|}\\
  &\le (2\|A\|_1+1)^{n|F|}(1+\frac{4}{\varepsilon})^{n|F^{-1}K^{-1}\setminus (F')^{-1}|}.
 \end{align*}
When $F$ becomes sufficiently left invariant, $|F^{-1}K^{-1}\setminus (F')^{-1}|/|F|$ becomes arbitrarily small.  It follows that
$$ \rh(X)\le n\log (2\|A\|_1+1)<+\infty.$$

(4)$\Rightarrow$(3): Suppose that (3) fails. By Proposition~\ref{P-basic p-expansive}.(5) we have $aA^*=0$ for some nonzero $a\in (\ell^2(\Gamma))^k$.
 Then $aA^*A=0$. By Theorem~\ref{T-Elek} we have $bA^*A=0$ for some nonzero $b\in (\Cb\Gamma)^k$. Replacing $b$ by its real part or imaginary part, we may assume that $b\in (\Rb\Gamma)^k$.

 Note that for any $w\in (\ell^2(\Gamma))^n$ and $w'\in (\ell^2(\Gamma))^k$ we have
$$ \left< wA,  w'\right>=\left<w, w'A^*\right>,$$
where we take $(\ell^2(\Gamma))^k$ and $(\ell^2(\Gamma))^n$ as the direct sum of copies of the Hilbert space $\ell^2(\Gamma)$.
Thus $\left<bA^*, bA^*\right>=\left<bA^*A, b\right>=0$, and hence $bA^*=0$. Therefore (4) fails.

(3)$\Rightarrow$(2)$\Rightarrow$(4) follows from Proposition~\ref{P-basic p-expansive}.

(4)$\Rightarrow$(5) is trivial.

(5)$\Rightarrow$(4): Suppose that (4) fails. Then $aA^*=0$ for some nonzero $a\in (\Rb \Gamma)^k$. Denote by $F$ (resp. $K$) the support of $a$ (resp. $A^*$) as a $\Rb^k$-valued (resp. $M_{k\times n}(\Zb)$-valued) function on $\Gamma$. Consider the equation $bA^*=0$ for $b\in (\Rb^F)^k$. One can interpret it as a system of integer-coefficients homogeneous linear equations indexed by $FK\times \{1, \dots, n\}$ which has variables indexed by $F\times \{1, \dots, k\}$. These linear equations have a nonzero solution given by $a$, thus has a nonzero integral solution. Therefore (5) fails.
\end{proof}

\section{P-Homoclinic Group} \label{S-p-homoclinic}

In this section we study $p$-homoclinic points. Throughout this section $\Gamma$ will be a countable discrete group.

When $\Gamma$ acts on a compact group $X$ by automorphisms, a point $x\in X$ is said to be {\it homoclinic} \cite{LS99} if $sx\to e_X$ as $\Gamma \ni s\to \infty$.
The set of all homoclinic points, denoted by $\Delta(X)$, is a $\Gamma$-invariant normal subgroup of $X$. Note that
when $X$ is abelian, a point $x\in X$ is homoclinic exactly when $\Psi_{x, \varphi}\in C_0(\Gamma)$ for every $\varphi \in \widehat{X}$,
where $\Psi_{x, \varphi}$ and $C_0(\Gamma)$ are defined in Notations~\ref{N-function} and \ref{N-principal} respectively.

\begin{definition} \label{D-p-homoclinic}
Let $\Gamma$ act on a compact abelian group $X$ by automorphisms. Let $1\le p<+\infty$. We say that $x\in X$ is {\it $p$-homoclinic} if
$\Psi_{x, \varphi}\in \ell^p(\Gamma)$ for every $\varphi\in \widehat{X}$. We denote by $\Delta^p(X)$ the set of all $p$-homoclinic points of $X$.
We also say $x\in X$ is {\it $\infty$-homoclinic} if it is homoclinic, and set $\Delta^\infty(X)=\Delta(X)$.
\end{definition}

The set $\Delta^1(X)$ was studied in \cite{SV, LSV} for algebraic $\Zb^d$-actions. In Section~\ref{S-duality} we shall see that both $\Delta^1(X)$ and $\Delta^2(X)$ play important roles in the study of  entropy theory for algebraic actions.

We collect some basic properties of the $p$-homoclinic group.

\begin{proposition} \label{P-basic p-homoclinic}
Let $\Gamma$ act on a compact abelian group $X$ by automorphisms. Let $1\le p\le +\infty$. Then the following hold:
\begin{enumerate}
\item For any $p<q\le +\infty$, one has $\Delta^p(X)\subseteq \Delta^q(X)$.

\item  $\Delta^p(X)$ is a $\Gamma$-invariant subgroup of $X$.

\item If $\Gamma$ acts on a compact abelian group $Y$ by automorphisms and $\Phi:X\rightarrow Y$ is a continuous $\Gamma$-equivariant group homomorphism, then
$\Phi(\Delta^p(X))\subseteq \Delta^p(Y)$.

\item If $\widehat{X}$ is a finitely generated left $\Zb\Gamma$-module and
we write $\widehat{X}$ as $(\Zb\Gamma)^k/J$ for some $k\in \Nb$ and
some left $\Zb\Gamma$-submodule $J$ of $(\Zb\Gamma)^k$,
then $\Delta^p(X)$ consists exactly of the elements $x\in X$
for which the function $s\mapsto \rho_\infty(x_s, 0_{(\Rb/\Zb)^k})$ on $\Gamma$ is in $\ell^p(\Gamma)$ when $p<+\infty$ or in $C_0(\Gamma)$ when $p=+\infty$,
where $\rho_\infty$ is the canonical  metric on $(\Rb/\Zb)^k$ defined by \eqref{E-metric on torus}.

\item If $\alpha$ is $p$-expansive, then $\Delta^p(X)$ is countable.

\item If $\Zb\Gamma$ is left Noetherian, and $\alpha$ is $p$-expansive, then $\Delta^p(X)$ is a finitely generated left $\Zb\Gamma$-module.
\end{enumerate}
\end{proposition}
\begin{proof} The assertion (1) follows from the facts that $\ell^p(\Gamma)\subseteq C_0(\Gamma)$  and $\ell^p(\Gamma)\subseteq \ell^q(\Gamma)$ when $q<+\infty$. The assertions (2) and (3) are obvious. The assertion (4) follows from the inequalities \eqref{E-basic p-expansive 2}
and \eqref{E-basic p-expansive}.

Now we prove the assertion (5). The case $p=+\infty$ is \cite[Lemma 3.2]{LS99}. So we may assume $p<+\infty$. Suppose that $\alpha$ is $p$-expansive. By Proposition~\ref{P-basic p-expansive}.(2) we may write $\widehat{X}$
as $(\Zb\Gamma)^k/J$ for some $k\in \Nb$ and some left $\Zb\Gamma$-submodule $J$ of $(\Zb\Gamma)^k$. Denote by $P$ the canonical map $\ell^\infty(\Gamma, \Rb^k)\rightarrow ((\Rb/\Zb)^k)^\Gamma$. For each $x\in \Delta^p(X)$, by the inequalities \eqref{E-basic p-expansive} we can take $\tilde{x}\in \ell^p(\Gamma, \Rb^k)$ such that $P(\tilde{x})=x$. Since $\alpha$ is $p$-expansive, by Proposition~\ref{P-basic p-expansive}.(3) and the inequalities \eqref{E-basic p-expansive} we can find some $\varepsilon>0$ such
that if $x, y\in \Delta^p(X)$ are distinct, then $\|\tilde{x}-\tilde{y}\|_p>\varepsilon$. As $\ell^p(\Gamma, \Rb^k)$ is separable under the norm $\|\cdot \|_p$, any $\varepsilon$-separated subset of $\ell^p(\Gamma, \Rb^k)$ is countable. Therefore $\Delta^p(X)$ is countable.

To prove the assertion (6), we need the following two lemmas.

\begin{lemma} \label{L-p-homoclinic finitely generated}
Let $1\le p<+\infty$. Let $\Gamma$ act $p$-expansively on a compact abelian group $X$ by automorphisms. Assume that $\widehat{X}$ is a finitely presented left $\Zb\Gamma$-module, and write $\widehat{X}$ as $(\Zb\Gamma)^k/(\Zb\Gamma)^nA$ for some $A\in M_{n\times k}(\Zb\Gamma)$. Then $\Delta^p(X)$ is isomorphic to a $\Zb\Gamma$-submodule of
$(\Zb\Gamma)^n/(\Zb\Gamma)^kA^*$.
\end{lemma}
\begin{proof} Denote by $P$ the canonical map from $(\ell^{\infty}(\Gamma))^k=\ell^\infty(\Gamma, \Rb^k)$ to $((\Rb/\Zb)^\Gamma)^k$. For each $x\in \Delta^p(X)$, by the inequalities \eqref{E-basic p-expansive} we can take $\tilde{x}\in \ell^p(\Gamma, \Rb^k)$ such that $P(\tilde{x})=x$. Then $\tilde{x}A^*\in \ell^\infty(\Gamma, \Zb^n)\cap \ell^p(\Gamma, \Rb^n)=(\Zb\Gamma)^n$. Thus we can define a map $\varphi: \Delta^p(X)\rightarrow (\Zb\Gamma)^n/(\Zb\Gamma)^kA^*$ sending
$x$ to $\tilde{x}A^*+(\Zb\Gamma)^kA^*$.

If $a\in \ell^p(\Gamma, \Rb^k)$ satisfies $P(a)=x$, then $a-\tilde{x}\in \ell^\infty(\Gamma, \Zb^k)\cap \ell^p(\Gamma, \Rb^k)=(\Zb\Gamma)^k$, and
hence $aA^*+(\Zb\Gamma)^kA^*=\tilde{x}A^*+(\Zb\Gamma)^kA^*$. It follows easily that $\varphi$ is a $\Zb\Gamma$-module homomorphism.

Since the action is $p$-expansive, by Proposition~\ref{P-basic p-expansive}.(5) the linear map $(\ell^p(\Gamma))^k\rightarrow (\ell^p(\Gamma))^n$ sending
$a$ to $aA^*$ is injective. If $x\in \ker \varphi$, then $\tilde{x}A^*\in (\Zb\Gamma)^kA^*$, and hence $\tilde{x}\in (\Zb\Gamma)^k$,  which implies that
$x=P(\tilde{x})=e_X$. Thus $\varphi$ is injective.
\end{proof}

The next lemma is more than needed for the proof of (6), but will be useful later.

\begin{lemma} \label{L-key expansive homoclinic}
Let $k\in \Nb$, and  $A\in M_k(\Zb\Gamma)$ such that the linear map $T:(\ell^p(\Gamma))^k\rightarrow (\ell^p(\Gamma))^k$
sending $a$ to $aA^*$ is invertible for some $1\le p<+\infty$.
Set $X_A=\widehat{(\Zb\Gamma)^k/(\Zb\Gamma)^kA}$.
Then
$\Delta^p(X_A)=\Delta(X_A)=P(T^{-1}((\Zb\Gamma)^k))$ is dense in $X_A$,
where $P$ denotes the canonical map from $(\ell^{\infty}(\Gamma))^k=\ell^\infty(\Gamma, \Rb^k)$ to $((\Rb/\Zb)^\Gamma)^k$. Furthermore, $\Delta(X_A)$ is isomorphic to
$(\Zb\Gamma)^k/(\Zb\Gamma)^kA^*$ as left $\Zb\Gamma$-modules.
\end{lemma}
\begin{proof}
 From (1) and (4) of Proposition~\ref{P-basic p-homoclinic} we have $\Delta(X_A)\supseteq \Delta^p(X_A)\supseteq P(T^{-1}((\Zb\Gamma)^k))$. Let $x\in \Delta(X_A)$. Take $\tilde{x}\in ([-1/2, 1/2]^k)^\Gamma\subseteq (\ell^{\infty}(\Gamma))^k$ with $P(\tilde{x})=x$. Then the function $s \mapsto \|\tilde{x}_s\|_\infty$ on $\Gamma$ vanishes at infinity. Since $A^*\in M_k(\Zb \Gamma)$, it follows
that the function $s \mapsto \|(\tilde{x}A^*)_s\|_\infty$ on $\Gamma$ also vanishes at infinity. As $\tilde{x}A^*\in (\Zb^k)^\Gamma$, we conclude that $T(\tilde{x})=\tilde{x}A^*\in (\Zb\Gamma)^k$. Then $\tilde{x}\in T^{-1}((\Zb\Gamma)^k)$, and hence $x\in P(T^{-1}((\Zb\Gamma)^k))$.
Therefore $\Delta^p(X_A)=\Delta(X_A)=P(T^{-1}((\Zb\Gamma)^k))$.

Take $1<q\le +\infty$ with $p^{-1}+q^{-1}=1$. Note that $\ell^q(\Gamma, \Rb^k)$ can be identified with the dual space of $\ell^p(\Gamma, \Rb^k)$ naturally, as in the equation \eqref{E-pairing}.
Denote
by $T^*$ the bounded linear map $\ell^q(\Gamma, \Rb^k)\rightarrow \ell^q(\Gamma, \Rb^k)$ dual to $T$.
Let $\psi\in \widehat{X_A}$ such that $\left<\Delta(X_A), \psi\right>=1$. Write $\psi$ as $a+(\Zb\Gamma)^kA$ for some $a\in (\Zb\Gamma)^k$.
Then
\begin{align*}
(\tilde{y}((T^*)^{-1}(a))^*)_{e_\Gamma}&=(\tilde{y}((T^{-1})^*(a))^*)_{e_\Gamma}=\left<\tilde{y}, (T^{-1})^*(a)\right>\\
&=\left<T^{-1}(\tilde{y}), a\right>=(T^{-1}(\tilde{y})a^*)_{e_\Gamma}\in \Zb
\end{align*} for all $\tilde{y}\in (\Zb\Gamma)^k$. Replacing $\tilde{y}$ by $s\tilde{y}$ for all $s\in \Gamma$, we get $\tilde{y}((T^*)^{-1}(a))^*\in \ell^q(\Gamma, \Zb^k)$ for
all $\tilde{y}\in (\Zb\Gamma)^k$.
Taking $\tilde{y}$ to be the canonical basis
of $(\Zb\Gamma)^k$, we get
$(T^*)^{-1}(a)\in \ell^q(\Gamma, \Zb^k)$.
When $q<+\infty$, we get $(T^*)^{-1}a\in \ell^q(\Gamma, \Zb^k)=(\Zb\Gamma)^k$.
When $q=+\infty$, from the surjectivity of $T$ we see that $A^*$ has a left inverse in $M_k(\ell^1(\Gamma))$ and hence $A$ has a right inverse in
$M_k(\ell^1(\Gamma))$. Multiplying this right inverse to $a=T^*((T^*)^{-1}(a))=((T^*)^{-1}(a))A$, we conclude that $(T^*)^{-1}(a)\in \ell^1(\Gamma, \Rb^k)$.
Therefore we still have $ (T^*)^{-1}a\in (\Zb\Gamma)^k$.
Thus $a\in T^*((\Zb\Gamma)^k))=(\Zb\Gamma)^kA$ and hence $\psi=0_{\widehat{X_A}}$. By the Pontryagin duality $\Delta(X_A)$ is dense in
$X_A$.

Denote $\varphi$ by the map $(\Zb\Gamma)^k\rightarrow \Delta(X_A)$ sending $a$ to $P(T^{-1}(a))$.
Clearly $\varphi$ is a left $\Zb\Gamma$-module homomorphism, and $\ker \varphi\supseteq (\Zb\Gamma)^kA^*$.
If $a\in \ker \varphi$, then $T^{-1}(a)\in \ell^\infty(\Gamma, \Zb^k)\cap \ell^p(\Gamma, \Rb^k)=(\Zb\Gamma)^k$
and hence $a\in T((\Zb\Gamma)^k)=(\Zb\Gamma)^kA^*$. Therefore $\ker \varphi=(\Zb\Gamma)^kA^*$. By the first paragraph of the proof $\varphi$ is also surjective. Thus $\Delta(X_A)$ is isomorphic to $(\Zb\Gamma)^k/\ker \varphi=(\Zb\Gamma)^k/(\Zb\Gamma)^kA^*$ as left $\Zb\Gamma$-modules.
\end{proof}

When $p<+\infty$, the assertion (6) follows from Lemma~\ref{L-p-homoclinic finitely generated}, Proposition~\ref{P-basic p-expansive}.(2) and the fact that if a unital ring $R$ is left Noetherian, then every finitely generated left $R$-module is finitely presented \cite[Proposition 4.29]{Lam}. When $p=+\infty$,
the assertion (6) follows from Proposition~\ref{P-basic p-expansive}.(4), Theorem~\ref{T-algebraic characterization of expansive} and Lemma~\ref{L-key expansive homoclinic}. This finishes the proof
of Proposition~\ref{P-basic p-homoclinic}.
\end{proof}

\begin{example} \label{Ex-finite Z(f)}
Let $\Gamma, f$, and $Z(f)$ be as in Example~\ref{Ex- Z^d uniform}. When $f$ is irreducible in the factorial ring $\Zb\Gamma$ and $Z(f)$ is finite, $\Delta^1(X_f)$ is dense in $X_f$ \cite[Propositions 5.1 and 5.2, Lemma 6.3]{LSV}. Denote by $u_1, \dots, u_d$ the canonical basis of $\Zb^d$. The group
$\Delta^1(X_f)$ was calculated explicitly for $f=2d-\sum_{j=1}^d(u_j+u_j^{-1})$ in \cite[Theorem 2.4, Proposition 3.1]{SV} and for $f=2-u_1-u_2$ when $d=2$ in
\cite[Section 5]{LSV}.
\end{example}

Lind and Schmidt \cite{LS99} showed that when $\Gamma$ is finitely generated with sub-exponential growth, for any expansive action of $\Gamma$ on a compact abelian group $X$ by automorphisms, one has
$\Delta(X)=\Delta^1(X)$.
From Theorem~\ref{T-algebraic characterization of expansive}, Lemma~\ref{L-key expansive homoclinic}, and Proposition~\ref{P-basic p-homoclinic}.(1) we conclude that this holds for all $\Gamma$:

\begin{theorem} \label{T-homoclinic are 1-homoclinic}
Let $\Gamma$ act expansively on a compact abelian group $X$ by automorphisms. Then
$$ \Delta(X)=\Delta^1(X).$$
\end{theorem}

We shall need the following result in Sections~\ref{S-specification} and \ref{S-duality}.

\begin{proposition} \label{P- 1-expansive summable under metric}
Let $\Gamma$ act on a compact abelian group $X$ by automorphisms. Suppose that $\widehat{X}$ is a finitely generated left $\Zb\Gamma$-module. Then there is a compatible translation-invariant metric $\rho$ on $X$ such that $\sum_{s\in \Gamma}\rho(0_X, sx)<+\infty$ for every $x\in \Delta^1(X)$.
\end{proposition}
\begin{proof}  The case $\Gamma$ is finite is trivial. Thus we may assume that $\Gamma$ is infinite. Take a finite set $W\subseteq \widehat{X}$ generating $\widehat{X}$ as a left $\Zb\Gamma$-module. List the elements of $\Gamma$ as $s_1, s_2, \dots$.
Define a continuous translation-invariant metric $\rho$ on $X$ by
$$ \rho(x, y)=\sum_{j=1}^\infty\sum_{\varphi \in W}2^{-j}|\Psi_{x-y, \varphi}(s_j)|$$
for all $x, y\in X$, where $\Psi_{x-y, \varphi}$ is defined in Notation~\ref{N-function}. If we denote by $\tau$ the original topology on $X$, and by $\tau'$ the topology on $X$ induced by $\rho$, then the identity map $\sigma:(X, \tau)\rightarrow (X, \tau')$ is continuous. Since $(X, \tau)$ is compact and $(X, \tau')$ is Hausdorff, $\sigma$ is a homeomorphism. Thus $\rho$ is compatible.
For any $x\in \Delta^1(X)$, one has
\begin{align*}
\sum_{t\in \Gamma}\rho(0_X, tx)&=\sum_{t\in \Gamma}\sum_{j=1}^\infty\sum_{\varphi \in W}2^{-j}|\Psi_{tx, \varphi}(s_j)|\\
&=\sum_{\varphi \in W}\sum_{j=1}^\infty2^{-j}\sum_{t\in \Gamma}|\Psi_{x, \varphi}(s_jt)|\\
&=\sum_{\varphi \in W}\sum_{j=1}^\infty2^{-j}\|\Psi_{x, \varphi}\|_1\\
&=\sum_{\varphi\in W}\|\Psi_{x, \varphi}\|_1<+\infty.
\end{align*}
\end{proof}

In the rest of this section we discuss pairs of algebraic actions. Let $G_1$ and $G_2$ be discrete left $\Zb\Gamma$-modules.
Consider a  map
$\Phi: G_1\times G_2\rightarrow \Tb$ which is {\it bi-additive} in the sense
that
$$ \Phi(\varphi_1+\psi_1, \varphi_2)=\Phi(\varphi_1, \varphi_2)\Phi(\psi_1, \varphi_2) \mbox{ and } \Phi(\varphi_1, \varphi_2+\psi_2)=\Phi(\varphi_1, \varphi_2)\Phi(\varphi_1, \psi_2)$$
for all $\varphi_1, \psi_1\in G_1$ and $\varphi_2, \psi_2\in G_2$, and {\it equivariant} in the sense that
$$ \Phi(s\varphi_1, s\varphi_2)=\Phi(\varphi_1, \varphi_2)$$
for all $\varphi_1\in G_1, \varphi_2\in G_2$, and $s\in \Gamma$.  Then $\Phi$ induces $\Gamma$-equivariant
group homomorphisms $\Phi_1: G_1\rightarrow \widehat{G_2}$ and $\Phi_2: G_2\rightarrow \widehat{G_1}$ such that
$$\left<\Phi_1(\varphi_1), \varphi_2\right>=\Phi(\varphi_1, \varphi_2)=\left<\varphi_1, \Phi_2(\varphi_2)\right>$$
for all $\varphi_1\in G_1$ and $\varphi_2\in G_2$.

\begin{lemma} \label{L-pairing}
Let $G_1, G_2, \Phi, \Phi_1$, and $\Phi_2$ be as above. Then the following hold:
\begin{enumerate}
\item $\Phi_1$ is injective if and only if $\Phi_2(G_2)$ is dense in $\widehat{G_1}$.

\item For any $1\le p\le +\infty$, $\Phi_1(G_1)\subseteq \Delta^p(\widehat{G_2})$ if and only if $\Phi_2(G_2)\subseteq \Delta^p(\widehat{G_1})$.
\end{enumerate}
\end{lemma}
\begin{proof}
(1). This follows from the Pontryagin duality and the fact that $\ker \Phi_1$ consists of exactly those $\varphi_1\in G_1$ satisfying
$\left<\varphi_1, \Phi_2(G_2)\right>=1$.

(2). Let $\varphi_1\in G_1$ and $\varphi_2\in G_2$. For each $s\in \Gamma$, we have
\begin{align*}
\Psi_{\Phi_2(\varphi_2), \varphi_1}(s)&=\left<s\Phi_2(\varphi_2), \varphi_1\right>-1\\
&=\left<\Phi_2(s\varphi_2), \varphi_1\right>-1\\
&=\left<s\varphi_2, \Phi_1(\varphi_1)\right>-1\\
&=\left<\varphi_2, s^{-1}\Phi_1(\varphi_1)\right>-1\\
&=\Psi_{\Phi_1(\varphi_1), \varphi_2}(s^{-1}).
\end{align*}
Thus $\Psi_{\Phi_2(\varphi_2), \varphi_1}\in C_0(\Gamma)$ exactly when $\Psi_{\Phi_1(\varphi_1), \varphi_2}\in C_0(\Gamma)$, and when $1\le p<+\infty$,
$\Psi_{\Phi_2(\varphi_2), \varphi_1}\in \ell^p(\Gamma)$ exactly when $\Psi_{\Phi_1(\varphi_1), \varphi_2}\in \ell^p(\Gamma)$.
It follows that $\Phi_1(G_1)\subseteq \Delta^p(\widehat{G_2})$ exactly when $\Phi_2(G_2)\subseteq \Delta^p(\widehat{G_1})$.
\end{proof}

The above pairing has been studied by Einsiedler and Schmidt \cite{ES, ES02} for algebraic actions of $\Gamma=\Zb^d$ with $d\in \Nb$ on $X$, in the case
$G_1=\widehat{X}$ and $G_2=\Delta(X)$.  In view of Lemma~\ref{L-pairing},
$(G_1, G_2)$ should be thought of as a dual pair,
and the dynamic properties of the $\Gamma$-actions on $\widehat{G_1}$ and $\widehat{G_2}$ are reflected in each other.
This point of view will play a central role in Section~\ref{S-duality}.

\section{Specification Properties and Dense Homoclinic Points} \label{S-specification}

In this section, using Theorem~\ref{T-algebraic characterization of expansive}, we discuss the relation between various specification properties and having dense homoclinic points for expansive algebraic actions.
Throughout this section $\Gamma$ will be a countable discrete group.

Specification is a strong orbit tracing property. Ruelle \cite{Ruelle} studied the extension of the notion to $\Zb^d$-actions.
We take the definition of various specification properties from \cite[Definition 5.1]{LS99} (see also \cite[Remark 5.6]{LS99}), modified to the general group case.

\begin{definition} \label{D-specification}
Let $\alpha$ be a continuous $\Gamma$-action on a compact space $X$. Let $\rho$ be a compatible metric on $X$.
\begin{enumerate}
\item The action $\alpha$ has {\it weak specification} if there exists, for every $\varepsilon>0$, a nonempty finite subset $F$
of $\Gamma$ with the following property: for any finite collection $F_1, \dots, F_m$ of finite subsets of $\Gamma$ with
\begin{align} \label{E-weak specification}
FF_i\cap F_j=\emptyset  \mbox{ for } 1\le i, j\le m, i\neq j,
\end{align}
and for any collection of points $x^{(1)}, \dots, x^{(m)}$ in $X$, there exists a point $y\in X$ with
\begin{align} \label{E-specification}
\rho(sx^{(j)}, sy)\le \varepsilon \mbox{ for all } s\in F_j, 1\le j\le m.
\end{align}

\item The action $\alpha$ has {\it strong specification} if there exists, for every $\varepsilon>0$, a nonempty finite subset $F$ of $\Gamma$ with the following property: for any finite collection $F_1, \dots, F_m$ of finite subsets of $\Gamma$ satisfying (\ref{E-weak specification}) and any subgroup $\Gamma'$ of $\Gamma$ with
\begin{align} \label{E-strong specification}
FF_i\cap F_j(\Gamma'\setminus \{e_\Gamma\})=\emptyset  \mbox{ for } 1\le i, j\le m,
\end{align}
and for any collection of points $x^{(1)}, \dots, x^{(m)}$ in $X$, there exists a point $y\in X$ satisfying (\ref{E-specification}) and $sy=y$ for all $s\in \Gamma'$.

\item When $X$ is a compact group and $\alpha$ is by automorphisms of $X$, the action $\alpha$ has {\it homoclinic specification} if there exists, for
every $\varepsilon>0$, a nonempty finite subset $F$ of $\Gamma$ with the following property: for any finite subset $F_1$ of $\Gamma$ and any $x\in X$,
there exists $y\in \Delta(X)$ with
\begin{align*}
\rho(sx, sy)\le \varepsilon \mbox{ for all } s\in F_1, \\
\rho(e_X, sy)\le \varepsilon \mbox{ for all } s\in \Gamma\setminus FF_1.
\end{align*}
\end{enumerate}
\end{definition}

The following lemma is a version of \cite[Lemma 1]{Bryant}. The same argument also appeared in the proof of \cite[Theorem 10.36]{GH}.

\begin{lemma} \label{L-expansive to asymptotic}
Let $\alpha$ be an expansive continuous action of $\Gamma$ on a compact metric space $(X, \rho)$. Let $d>0$ such that
if $x, y\in X$ satisfy $\sup_{s\in \Gamma}\rho(sx, sy)\le d$, then $x=y$. Let $x, y\in X$ satisfy $\rho(sx, sy)\le d$ for all but finitely many
$s\in \Gamma$. Then $(x, y)$ is an asymptotic pair in the sense that $\rho(sx, sy)\to 0$ as $\Gamma \ni s\to \infty$.
\end{lemma}
\begin{proof} Take a finite subset $W$ of $\Gamma$ such that $\rho(sx, sy)\le d$ for every
$s\in \Gamma\setminus W$.
Suppose that $(x, y)$ is not an asymptotic pair. Then there exist $\varepsilon>0$ and  a sequence of elements $\{s_n\}_{n\in \Nb}$ in $\Gamma$
such that $\rho(s_nx, s_ny)\ge \varepsilon$ for all $n\in \Nb$ and for any finite subset $F$ of $\Gamma$ one has $s_n\not \in F$ for all sufficiently large $n\in \Nb$. Passing to a subsequence of $\{s_n\}_{n\in \Nb}$, we may assume that $s_nx$ and $s_ny$ converge to $x'$ and $y'$ in $X$ respectively, as $n\to \infty$.
Then $\rho(x', y')\ge \varepsilon$ and hence $x'\neq y'$.

Let $s\in \Gamma$. When $n\in \Nb$ is sufficiently large, one has $s s_n\not \in W$ and hence $\rho(s s_nx, s s_ny)\le d$. Letting $n\to \infty$, we obtain
$\rho(s x', s y')\le d$. By the hypothesis on $d$ we conclude that $x'=y'$, which is a contradiction. Therefore $(x, y)$ is an asymptotic pair.
\end{proof}

\begin{theorem} \label{T-specification}
Let $\alpha$ be an expansive action of $\Gamma$ on a compact abelian group $X$ by automorphisms. Then the following are equivalent:
\begin{enumerate}
\item $\alpha$ has weak specification;

\item $\alpha$ has strong specification;

\item $\alpha$ has homoclinic specification.
\end{enumerate}
Furthermore, these conditions imply
\begin{enumerate}
\item[(4)] $\Delta(X)$ is dense in $X$.
\end{enumerate}
\end{theorem}
\begin{proof} By Theorem~\ref{T-algebraic characterization of expansive} we
may assume that $X=\widehat{(\Zb\Gamma)^k/J}$ for some $k\in \Nb$ and some left $\Zb\Gamma$-submodule $J$ of $(\Zb\Gamma)^k$,
and find some $A\in M_k(\Zb\Gamma)$ being invertible in $M_k(\ell^1(\Gamma))$ such that the rows of $A$ are contained in $J$.
Denote by $W$ the union of $\{e_\Gamma\}$ and the support of $A^*$ as a $M_k(\Zb)$-valued function on $\Gamma$.
Recall the canonical metric $\rho_\infty$ on $(\Rb/\Zb)^k$ defined by \eqref{E-metric on torus} and the norm
$\|\cdot\|_\infty$ on $\ell^\infty(\Gamma, \Rb^k)$ defined by \eqref{E-infinity norm}. Let $\rho$ be a compatible metric on $X$.

(3)$\Rightarrow$(2): Let $\varepsilon>0$. Then we can find a nonempty finite subset $W_1$ of $\Gamma$ and $\|A\|_1^{-1}>\varepsilon'>0$ such that if $x, y\in X$ satisfy
$\max_{s\in W_1}\rho_\infty(x_s, y_s)\le 2\varepsilon'$, then $\rho(x, y)<\varepsilon$. Take a finite subset $W_2$ of $\Gamma$ containing $e_\Gamma$
such that $\sum_{s\in \Gamma\setminus W_2^{-1}}\|((A^*)^{-1})_s\|_1<\varepsilon'/(2\|A\|_1)$, where $\|B\|_1$ denotes the sum of the absolute values of the entries of $B$ for $B\in M_k(\Rb)$.
If $\tilde{x}, \tilde{y}\in \ell^\infty(\Gamma, \Rb^k)$ satisfy $\|\tilde{x}\|_\infty, \|\tilde{y}\|_\infty\le \|A\|_1$ and $\tilde{x}$ and $\tilde{y}$ are equal on $sW_2$ for some $s\in \Gamma$, then $\|(\tilde{x}(A^*)^{-1})_s-(\tilde{y}(A^*)^{-1})_s\|_\infty\le \varepsilon'$.
Take $\delta>0$ such that if $x, y\in X$ satisfy $\rho(x, y)\le \delta$, then $\rho_\infty(x_{e_\Gamma}, y_{e_\Gamma})\le \varepsilon'$.

By the condition (3) we can find a finite subset $W_3$ of $\Gamma$ containing $e_\Gamma$ with the following
property: for any finite subset $F_1$ of $\Gamma$ and any $x\in X$,
there exists $y\in \Delta(X)$ with
\begin{align*}
\max_{s\in F_1}\rho(sx, sy)\le \delta \mbox{ and } \sup_{s\in \Gamma \setminus W_3F_1}\rho(e_X, sy)\le \delta.
\end{align*}
By our choice of $\delta$, we then have
\begin{align*}
\max_{s\in F_1^{-1}}\rho_\infty(x_s, y_s)\le \varepsilon' \mbox{ and }
\sup_{s\in \Gamma\setminus (W_3F_1)^{-1}}\rho_\infty(0_{(\Rb/\Zb)^k}, y_s)\le \varepsilon'.
\end{align*}
Now set $F=(W_1W_2W_3^{-1}W)(W_1W_2W_3^{-1}W)^{-1}$.

Let $F_1, \dots, F_m$ be a finite collection of finite subsets of $\Gamma$ satisfying (\ref{E-weak specification}), $\Gamma'$ be a subgroup of $\Gamma$ satisfying
(\ref{E-strong specification}), and $x^{(1)}, \dots, x^{(m)}$ be points in $X$. Take $y^{(1)}, \dots, y^{(m)}\in \Delta(X)$ with
\begin{align*}
\max_{s\in F_j^{-1}W_1W_2}\rho_\infty(x^{(j)}_s, y^{(j)}_s)\le \varepsilon' \mbox{ and }
\sup_{s\in \Gamma\setminus (F_j^{-1}W_1W_2W_3^{-1})}\rho_\infty(0_{(\Rb/\Zb)^k}, y^{(j)}_s)\le \varepsilon'
\end{align*}
for each $1\le j\le m$.
Denote by $P$ the canonical map $\ell^{\infty}(\Gamma, \Rb^k)\rightarrow ((\Rb/\Zb)^k)^\Gamma$. For each $1\le j\le m$, take a lift $\tilde{x}^{(j)}$
and $\tilde{y}^{(j)}$ for $x^{(j)}$ and $y^{(j)}$ in $(([-1, 1])^k)^{\Gamma}$ respectively such that
\begin{align*}
\max_{s\in F_j^{-1}W_1W_2}\|\tilde{x}^{(j)}_s-\tilde{y}^{(j)}_s\|_\infty\le \varepsilon' \mbox{ and }
\sup_{s\in \Gamma\setminus (F_j^{-1}W_1W_2W_3^{-1})}\|\tilde{y}^{(j)}_s\|_\infty\le \varepsilon'.
\end{align*}
Then $\tilde{y}^{(j)}A^*$ belongs to $\ell^\infty(\Gamma, \Zb^k)$. Note that, for any  $s\in \Gamma\setminus (F_j^{-1}W_1W_2W_3^{-1}W)$, one has
$$\|(\tilde{y}^{(j)}A^*)_s\|_\infty\le \|A^*\|_1\cdot \sup_{t\in \Gamma\setminus (F_j^{-1}W_1W_2W_3^{-1})}\|\tilde{y}^{(j)}_t\|_\infty\le \|A\|_1\cdot \varepsilon'<1.$$
It follows that, as a $\Zb^k$-valued function on $\Gamma$, $\tilde{y}^{(j)}A^*$ has support contained in $F_j^{-1}W_1W_2W_3^{-1}W$.
We can rewrite (\ref{E-weak specification}) and (\ref{E-strong specification}) as
\begin{align*}
(F_i^{-1}W_1W_2W_3^{-1}W)\cap (F_j^{-1}W_1W_2W_3^{-1}W)=\emptyset  \mbox{ for } 1\le i, j\le m, i\neq j,
\end{align*}
and
\begin{align*}
(F_i^{-1}W_1W_2W_3^{-1}W)\cap (\Gamma'\setminus \{e_\Gamma\})(F_j^{-1}W_1W_2W_3^{-1}W)=\emptyset  \mbox{ for } 1\le i, j\le m,
\end{align*}
respectively.
Thus the elements $s\tilde{y}^{(j)}A^*$ of $\Zb^k\Gamma$ for $s\in \Gamma'$ and $1\le j\le m$ have pairwise disjoint support.
Then we have the element  $\tilde{z}:=\sum_{s\in \Gamma'} \sum_{j=1}^ms\tilde{y}^{(j)}A^*$ of $\ell^{\infty}(\Gamma, \Zb^k)$ with
$$\|\tilde{z}\|_\infty=\max_{1\le j\le m}\|\tilde{y}^{(j)}A^*\|_\infty\le \|A^*\|_1\cdot \max_{1\le j\le m}\|\tilde{y}^{(j)}\|_\infty\le \|A\|_1.$$
 Set $\tilde{y}=\tilde{z}(A^*)^{-1}\in \ell^\infty(\Gamma, \Rb^k)$
and $y=P(\tilde{y})\in ((\Rb/\Zb)^k)^\Gamma$.

We claim that $y\in X$. For each finite subset $K$ of $\Gamma'$, define $\tilde{z}_K=\sum_{s\in K} \sum_{j=1}^ms\tilde{y}^{(j)}A^*$,
$\tilde{y}_K=\tilde{z}_K(A^*)^{-1}$, and $y_K=P(\tilde{y}_K)$. Then $\|\tilde{z}_K\|_\infty\le \|A\|_1$ for every $K$.
For each $s\in \Gamma$, when $K\to \Gamma'$, we have $(\tilde{z}_K)_s\to \tilde{z}_s$, and hence
$(\tilde{y}_K)_s\to \tilde{y}_s$. It follows that, when $K\to \Gamma'$, $y_K$ converges to $y$.
Clearly $y_K=\sum_{s\in K}\sum_{j=1}^m sy^{(j)}\in X$ for each $K$. Therefore $y\in X$.

For each $s\in \Gamma'$, we have $s\tilde{z}=\tilde{z}$ and hence $sy=y$.

Fix $1\le j\le m$. Note that $\tilde{z}$ and $\tilde{y}^{(j)}A^*$ are equal on $F_j^{-1}W_1W_2W_3^{-1}W\supseteq F_j^{-1}W_1W_2$.
Since $\|\tilde{z}\|_\infty, \|\tilde{y}^{(j)}A^*\|_\infty\le \|A\|_1$, by our choice of $W_2$ we have
$\max_{s\in F_j^{-1}W_1}\|\tilde{y}_s-\tilde{y}^{(j)}_s\|_\infty\le \varepsilon'$. Thus $\max_{s\in F_j^{-1}W_1}\rho_\infty(y_s, y^{(j)}_s)\le \varepsilon'$,
and hence
\begin{eqnarray*}
\max_{s\in F_j, t\in W_1}\rho_\infty((sy)_{t}, (sx^{(j)})_{t})
&=&\max_{s\in F_j^{-1}W_1}\rho_\infty(y_s, x^{(j)}_s)\\
&\le &\max_{s\in F_j^{-1}W_1}\rho_\infty(y_s, y^{(j)}_s)+\max_{s\in F_j^{-1}W_1}\rho_\infty(y^{(j)}_s, x^{(j)}_s)\\
&\le &2\varepsilon'.
\end{eqnarray*}
By our choice of $W_1$ and $\varepsilon'$, we get $\max_{s\in F_j}\rho(sy, sx^{(j)})< \varepsilon$ as desired.

(2)$\Rightarrow$(1) and (3)$\Rightarrow$(4) are trivial.

(1)$\Rightarrow$(3): By Lemma~\ref{L-expansive to asymptotic}
there exists $\varepsilon'>0$ with the following property: if $y\in X$ satisfies $\rho(e_X, sy)\le \varepsilon'$
for all but finitely many $s\in \Gamma$, then $y\in \Delta(X)$.
 Let $\varepsilon'\ge \varepsilon>0$. Take $F$ as in the definition of weak specification. Replacing $F$ by $F\cup F^{-1}$ if necessary, we may assume that $F=F^{-1}$. Let $F_1$ be a finite subset of $\Gamma$ and $x\in X$. For each  finite subset $F_2$ of $\Gamma\setminus FF_1$, by the choice of $F$, taking $x^{(1)}=x$ and $x^{(2)}=e_X$ we can find
$y_{F_2}\in X$ such that $\max_{s\in F_1}\rho(s x, s y_{F_2})\le \varepsilon$ and $\max_{s\in F_2}\rho(s e_X, s y_{F_2})\le \varepsilon$. Note that the set of finite subsets of $\Gamma\setminus FF_1$ is partially ordered by inclusion. Take a limit point $y$ of the net $\{y_{F_2}\}_{F_2}$. Then $\max_{s\in F_1}\rho(s x, s y)\le \varepsilon$ and $\sup_{s\in \Gamma\setminus FF_1}\rho(se_X, s y)\le \varepsilon\le \varepsilon'$. By our choice of $\varepsilon'$, we conclude that $y\in \Delta(X)$.
\end{proof}

We give a new proof of the following result of Lind and Schmidt \cite[Theorem 5.2]{LS99}.

\begin{theorem} \label{T-specification for Z^d}
Suppose that $\Gamma=\Zb^d$ for some $d\in \Nb$.
Let $\alpha$ be an expansive action of $\Gamma$ on a compact abelian group $X$ by automorphisms.
Then the conditions (1), (2), (3) and (4) in Theorem~\ref{T-specification} are all equivalent.
\end{theorem}
\begin{proof}We just need to show that (4)$\Rightarrow$(3). By Theorem~\ref{T-homoclinic are 1-homoclinic} and Propositions~\ref{P-expansive to subshift} and \ref{P- 1-expansive summable under metric}, we can find a compatible translation-invariant
metric $\rho$ on $X$ such that $\sum_{s\in \Gamma}\rho(0_X, sx)<+\infty$ for all $x\in \Delta(X)$.

Since $\Zb\Gamma$ is Noetherian \cite[Corollary IV.4.2]{Lang}, from Propositions~\ref{P-basic p-expansive}.(4) and \ref{P-basic p-homoclinic}.(6)  we see that $\Delta(X)$ is a finitely generated left $\Zb\Gamma$-module.
Take $z_1, \dots, z_n\in \Delta(X)$ such that $\Delta(X)=\sum_{j=1}^n\Zb\Gamma z_j$.

By Theorem~\ref{T-algebraic characterization of expansive} $X$ is a closed $\Gamma$-invariant subgroup of $X_A:=\widehat{(\Zb\Gamma)^k/(\Zb\Gamma)^kA}$ for some
$k\in \Nb$ and some $A\in \M_k(\Zb\Gamma)$ being invertible in $M_k(\ell^1(\Gamma))$.
Treat $\Delta(X_A)$ as a discrete abelian group and consider the induced $\Gamma$-action on the Pontryagin dual $\widehat{\Delta(X_A)}$.
By Lemmas~\ref{L-key expansive homoclinic} and \ref{L-key expansive} the $\Gamma$-action on $\widehat{\Delta(X_A)}$ is expansive. By Corollary~\ref{C-characterization of expansive for abelian} there exists $f\in \Zb\Gamma$ being invertible in $\ell^1(\Gamma)$ such that $f\Delta(X_A)=0$. In particular, $f\Delta(X)=0$.

Let $\varepsilon>0$. Take a nonempty finite subset $W$ of $\Gamma$ such that $\sum_{j=1}^n\sum_{s\in \Gamma\setminus W}\rho(0_X, sz_j)<\varepsilon/(2\|f\|_1)$. Set $F=WW^{-1}$.

Let $F_1$ be a finite subset of $\Gamma$ and $x'\in X$. By the condition (4) we can take some $x\in \Delta(X)$ with $\max_{s\in F_1}\rho(sx', sx)<\varepsilon/2$.

We have $x=\sum_ja_jz_j$ for some $a_1, \dots, a_n\in \Zb\Gamma$. Note that $a_jf^{-1}$ is in $\ell^1(\Gamma)$.
Let $b_j$ be the integral part of $a_jf^{-1}$. That is, $b_j\in \Rb^\Gamma$ has integral coefficients, and $a_jf^{-1}-b_j$ has coefficients in $[-1/2, 1/2)$. Then $\|b_j\|_1\le 2\|a_jf^{-1}\|_1<+\infty$, and hence $b_j\in \Zb\Gamma$. Note that $x=\sum_j(a_j-b_jf)z_j$, and
$$\|a_j-b_jf\|_\infty\le \|a_jf^{-1}-b_j\|_\infty\cdot \|f\|_1\le \|f\|_1.$$ Thus, replacing
$a_j$ by $a_j-b_jf$ if necessary, we may assume that $\|a_j\|_\infty\le \|f\|_1$ for all $1\le j\le n$.

For each $1\le j\le n$, define $c_j\in \Zb\Gamma$ to be the same as $a_j$ on $F_1^{-1}W$ and $0$ outside of
 $F_1^{-1}W$. Set $y=\sum_jc_jz_j\in \Delta(X)$. For each $s\in F_1$, since $\rho$ is translation-invariant, we have
\begin{align*}
\rho(sx, sy)&=\rho(\sum_js(a_j-c_j)z_j, 0_X)\\
&\le \sum_j\sum_{t\in \Gamma}|(s(a_j-c_j))_t|\rho(tz_j, 0_X)\\
&=\sum_j\sum_{t\in \Gamma\setminus W}|(a_j-c_j)_{s^{-1}t}|\rho(tz_j, 0_X)\\
&\le \sum_j\sum_{t\in \Gamma\setminus W}\|a_j\|_\infty\rho(tz_j, 0_X)\\
&\le \|f\|_1\sum_j\sum_{t\in \Gamma\setminus W}\rho(tz_j, 0_X)\le \varepsilon/2,
\end{align*}
and hence
$$ \rho(sx', sy)\le \rho(sx', sx)+\rho(sx, xy)\le \varepsilon.$$

For each $s\in \Gamma\setminus FF_1$, noting that $s^{-1}W\cap F_1^{-1}W=\emptyset$, we have
\begin{align*}
\rho(sy, 0_X)&\le \sum_j\sum_{t\in \Gamma}|(sc_j)_t|\rho(tz_j, 0_X)\\
&=\sum_j\sum_{t\in \Gamma\setminus W}|(c_j)_{s^{-1}t}|\rho(tz_j, 0_X)\\
&\le \sum_j\sum_{t\in \Gamma\setminus W}\|c_j\|_\infty\rho(tz_j, 0_X)\\
&\le \|f\|_1\sum_j\sum_{t\in \Gamma\setminus W}\rho(tz_j, 0_X)\le \varepsilon/2.
\end{align*}
\end{proof}

In general, we have

\begin{conjecture} \label{Conj-specification}
Suppose that $\Gamma$ is amenable and $\Zb\Gamma$ is left Noetherian. Let $\alpha$ be an expansive action of $\Gamma$ on a compact abelian group $X$ by automorphisms.
Then the conditions (1), (2), (3) and (4) in Theorem~\ref{T-specification} are all equivalent.
\end{conjecture}

\section{IE Group} \label{S-IE}

In this section we study the local entropy theory for $\Gamma$-actions on compact groups by automorphisms.
The basics of local entropy theory are recalled in Section~\ref{SS-local}.
Throughout this section $\Gamma$ will be a countable amenable group, unless specified.

For a continuous action of $\Gamma$ on a compact space $X$,
we denote by $\cM(X, \Gamma)$ the set
of all $\Gamma$-invariant Borel probability measures on $X$.
For a compact group $X$, we denote by $\mu_X$ the normalized Haar measure on $X$, and by $e_X$ the identity element
of $X$.

\begin{lemma} \label{L-product factor}
Let $\Gamma$ act on a compact  group $X$ by automorphisms,
and let $\nu\in \cM(X, \Gamma)$. Then the product map $X\times X\rightarrow X$ sending
$(x, y)$ to $xy$ is a $\Gamma$-equivariant surjective continuous map, for $X\times X$ equipped with
the product action, and sends the $\Gamma$-invariant measures $\mu_X\times \nu$ and $\nu \times \mu_X$ to $\mu_X$.
\end{lemma}
\begin{proof} Denote the product map by $\pi$. Then $\pi_*(\mu_X\times \nu)$ is a Borel probability measure on $X$. Since $\mu_X$ is left-shift invariant,
so is $\pi_*(\mu_X\times \nu)$. Thus $\pi_*(\mu_X\times \nu)=\mu_X$. Similarly $\pi_*(\nu \times \mu_X)=\mu_X$. The other parts of the lemma are obvious.
\end{proof}

\begin{definition} \label{D-IE group}
Let $\Gamma$ act on a compact group $X$ by automorphisms. We say that a point $x\in X$ is an {\it IE-point} if $(x, e_X)\in \IE_2(X)$. We denote the set of all $\IE$-points by $\IE(X)$.
\end{definition}

\begin{theorem} \label{T-mIE=IE}
Let $\Gamma$ act on a compact group $X$ by automorphisms. Then the following hold:
\begin{enumerate}
\item  The set $\IE(X)$ is a $\Gamma$-invariant closed normal subgroup of $X$.

\item
For any $k\in \Nb$, the set $\IE_k(X)$ is a $\Gamma$-invariant (under the product action of $\Gamma$ on $X^k$) closed subgroup of the group $X^k$, and
\begin{eqnarray*}
\IE_k(X)&=&\{(x_1y, \dots, x_ky): x_1, \dots, x_k\in \IE(X), y\in X\} \\
&=&\{(yx_1, \dots, yx_k): x_1, \dots, x_k\in \IE(X), y\in X\}.
\end{eqnarray*}

\item For any $\nu\in \cM(X, \Gamma)$ and $k\in \Nb$, one has
 $\IE^{\nu}_k(X)\subseteq \IE^{\mu_X}_k(X)=\IE_k(X)$.

\item $\rh(X)>0$ if and only if $\IE(X)$ is nontrivial.

\item  Let $Y$ be a closed $\Gamma$-invariant normal subgroup of $X$ and denote by $q$ the quotient
map $X\rightarrow X/Y$. Consider the induced $\Gamma$-action on $X/Y$. Then $q(\IE(X))=\IE(X/Y)$.
\end{enumerate}
\end{theorem}
\begin{proof} Let $k\in \Nb$.

(1) and (2).
Since $\supp(\mu_X)=X$, from Theorem~\ref{T-basic IE}.(3) we have $\IE_1(X)=X$. It is clear from the definition of IE-tuples that if $1\le m<k$ and $(x_1, \dots, x_m)\in \IE_m(X)$,
then $(x_1, \dots, x_m, x_{m+1}, \dots, x_k)\in \IE_k(X)$ for $x_{m+1}=\dots=x_k=x_1$. It follows that
the length $k$ diagonal element $(x, \dots, x)$ is in $\IE_k(X)$ for every $x\in X$. In particular, $\IE_k(X)$ contains the
identity element of the group $X^k$.

Consider the product action of $\Gamma$ on $X\times X$. Denote by $\pi$ the product map
$X\times X\rightarrow X$, and by $\pi^k$ its $k$-fold $(X\times X)^k\rightarrow X^k$. Note that $\pi^k$ can be identified with the product map
of the group $X^k$.
By Theorem~\ref{T-basic IE}.(5) one has $\IE_k(X\times X)=\IE_k(X)\times \IE_k(X)$.
From Lemma~\ref{L-product factor} and Theorem~\ref{T-basic IE}.(4), one gets $\pi^k(\IE_k(X\times X))=\IE_k(X)$.
Thus, $\IE_k(X)\cdot \IE_k(X)=\IE_k(X)$. Also applying
Theorem~\ref{T-basic IE}.(4) to the inverse map $X\rightarrow X$, one gets $(\IE_k(X))^{-1}=\IE_k(X)$.
Therefore
$\IE_k(X)$ is a subgroup of $X^k$.

By Theorem~\ref{T-basic IE}.(1), the set $\IE_k(X)$ is $\Gamma$-invariant and closed.

Since $\IE_2(X)$ is a $\Gamma$-invariant closed subgroup of $X^2$, clearly $\IE(X)$ is a $\Gamma$-invariant closed subgroup of $X$.
If $x\in \IE(X)$ and $y\in X$, then $(y, y)$, $(y^{-1}, y^{-1})$ and $(x, e_X)$ are all in $\IE_2(X)$,
and hence $(yxy^{-1}, e_X)=(y, y)\cdot (x, e_X)\cdot (y^{-1}, y^{-1})\in \IE_2(X)$. Therefore $\IE(X)$ is a normal subgroup
of $X$.

Now we show that $\IE_k(X)=\{(x_1y, \dots, x_ky): x_1, \dots, x_k\in \IE(X), y\in X\}$. The case $k=1$ follows from $\IE_1(X)=X$.
Assume that $k\ge 2$. Note that, by the definition of IE-tuples, $\IE_k(X)$ is closed under taking permutation.
If $x_1, \dots, x_k\in \IE(X)$ and $y\in X$, then the $k$-tuples $(x_1, e_X, \dots, e_X), (e_X, x_2, \dots, e_X), \dots,
(e_X, e_X, \dots, x_k)$ and $(y, \dots, y)$ are all in $\IE_k(X)$, and hence
$$(x_1y, x_2y, \dots, x_ky)=(x_1, e_X, \dots, e_X)\cdot (e_X, x_2, \dots, e_X)\cdot  \dots \cdot
(e_X, e_X, \dots, x_k)\cdot (y, \dots, y)$$
 is in $\IE_k(X)$.

Note that, by the definition of IE-tuples,
if $(x_1, \dots, x_k)\in \IE_k(X)$ and $1\le m\le k$, then $(x_1, \dots, x_m)\in \IE_m(X)$.
Suppose that $(y_1, \dots, y_k)\in \IE_k(X)$. Let $2\le j\le k$. Then $(y_1, y_j)\in \IE_2(X)$, and hence
$(e_X, y_jy_1^{-1})=(y_1, y_j)\cdot (y_1^{-1}, y_1^{-1})\in \IE_2(X)$. Thus $y_jy_1^{-1}\in \IE(X)$.
Set $x_1=e_X$, $x_j=y_jy_1^{-1}$ for all $2\le j\le k$, and $y=y_1$. Then $(y_1, \dots,y_k)=(x_1y, \dots, x_ky)$.
This proves    $\IE_k(X)=\{(x_1y, \dots, x_ky): x_1, \dots, x_k\in \IE(X), y\in X\}$.
Similarly, one has $\IE_k(X)=\{(yx_1, \dots, yx_k): x_1, \dots, x_k\in \IE(X), y\in X\}$.

(3).
Let $\nu\in \cM(X, \Gamma)$.
From Theorem~\ref{T-basic measure IE}.(5) one gets $\IE^{\mu_X\times \nu}_k(X\times X)=\IE^{\mu_X}_k(X)\times \IE^{\nu}_k(X)$.
By Lemma~\ref{L-product factor} and Theorem~\ref{T-basic measure IE}.(4), one has $\pi^k(\IE^{\mu_X\times \nu}_k(X\times X))=
\IE^{\pi_*(\mu_X\times \nu)}_k(X)=\IE^{\mu_X}_k(X)$. Thus, for any
$\ox\in \IE^{\mu_X}_k(X)$ and $\oy\in \IE^{\nu}_k(X)$, one has
$\ox\cdot\oy \in \IE^{\mu_X}_k(X)$.

By Theorem~\ref{T-basic measure IE}.(3), we have $\IE^{\mu_X}_1(X)=\supp(\mu_X)=X$.
It is clear from the definition of $\mu_X$-IE-tuples that if $1\le m<k$ and $(x_1, \dots, x_m)\in \IE^{\mu_X}_m(X)$,
then $(x_1, \dots, x_m, x_{m+1}, \dots, x_k)\in \IE^{\mu_X}_k(X)$ for $x_{m+1}=\dots=x_k=x_1$. It follows that
the length $k$ diagonal element $(x, \dots, x)$ is in $\IE^{\mu_X}_k(X)$ for every $x\in X$.
In particular, $\IE^{\mu_X}_k(X)$ contains the identity element $e_{X^k}$ of $X^k$.

For any $\oy\in \IE^{\nu}_k(X)$, we have $\oy=e_{X^k}\cdot \oy\in \IE^{\mu_X}_k(X)$. This proves
$\IE^{\nu}_k(X)\subseteq \IE^{\mu_X}_k(X)$.

Now from parts (1) and (6) of Theorem~\ref{T-basic measure IE} we conclude $\IE_k(X)=\IE^{\mu_X}_k(X)$.

(4). This follows from Theorem~\ref{T-basic IE}.(2).

(5). By Theorem~\ref{T-basic IE}.(4) we have $(q\times q)(\IE_2(X))=\IE_2(X/Y)$. It follows that
$q(\IE(X))\subseteq \IE(X/Y)$. Furthermore, for any $z\in \IE(X/Y)$, there exists $(x, y)\in \IE_2(X)$
with $q(x)=z$ and $q(y)=e_{X/Y}$. Then $y\in Y$. By Assertion (2) we have
$xy^{-1}\in \IE(X)$. Thus $z=q(xy^{-1})\in q(\IE(X))$, and hence $\IE(X/Y)\subseteq q(\IE(X))$.
\end{proof}

Under the assumptions of Theorem~\ref{T-mIE=IE},
$\Gamma$ has an induced action on $X/\IE(X)$ by automorphisms.
Under the quotient map $X\rightarrow X/\IE(X)$, $X/\IE(X)$ is a topological factor of $X$.

\begin{theorem} \label{T-topological Pinsker}
Let $\Gamma$ act on a compact group $X$ by automorphisms. Denote by $q$ the quotient map $X\rightarrow X/\IE(X)$. Then the following hold:
\begin{enumerate}
\item
$X/\IE(X)$ is the largest topological
factor $Y$ of $X$ satisfying $\htopol(Y)=0$, in the sense that $\htopol(  X/\IE(X))=0$ and if $Y$ is a topological factor
of $X$ with $\htopol(Y)=0$, then there is a topological factor map $X/\IE(X)\rightarrow Y$ such that the diagram
\begin{eqnarray*}
\xymatrixcolsep{5pc}
\xymatrix{
X \ar[rd] \ar[r]^q
&X/\IE(X) \ar[d]\\
&Y}
\end{eqnarray*}
commutes.

\item $X/\IE(X)$ is also the largest topological
factor $Y$ of $X$ satisfying $\rh_{\pi_*(\mu_X)}(Y)=0$ for $\pi:X\rightarrow Y$ denoting the factor map,
in the sense that $\rh_{q_*(\mu_X)}(X/\IE(X))=0$
and if $Y$ is a topological factor
of $X$ with $\rh_{\pi_*(\mu_X)}(Y)=0$
, then there is a topological factor map $X/\IE(X)\rightarrow Y$ such that the above diagram in (1) commutes.
\end{enumerate}
\end{theorem}
\begin{proof}
(1).
Let $(x, y)\in \IE_2(X/\IE(X))$. By Theorem~\ref{T-basic IE}.(4) we can find $(\tilde{x}, \tilde{y})\in \IE_2(X)$
with $q(\tilde{x})=x$ and $q(\tilde{y})=y$. By Theorem~\ref{T-mIE=IE} we have $\tilde{x}\tilde{y}^{-1}\in \IE(X)$. Thus
$x=q(\tilde{x})=q(\tilde{x}\tilde{y}^{-1})q(\tilde{y})=y$. Therefore $\IE_2(X/\IE(X))$ consists of only diagonal elements, and hence
by Theorem~\ref{T-basic IE}.(2) one has $\htopol(X/\IE(X))=0$.

Now let $Y$ be a topological factor of $X$ such that $\htopol(Y)=0$. Denote by $\pi$ the factor map $X\rightarrow Y$. To show that
there is a topological factor map $X/\IE(X)\rightarrow Y$ making the diagram in Assertion (1) commute, it suffices to show
for any $x', y'\in X$ with $q(x')=q(y')$ one has $\pi(x')=\pi(y')$.

Let $x_1\in \IE(X)$ and $y_1\in X$. From Theorem~\ref{T-mIE=IE} we get $(y_1, x_1y_1)\in \IE_2(X)$. By Theorem~\ref{T-basic IE}.(4)
we have $(\pi\times \pi)(\IE_2(X))=\IE_2(Y)$. Thus $(\pi(y_1), \pi(x_1y_1))\in \IE_2(Y)$. Since $\htopol(Y)=0$, by Theorem~\ref{T-basic IE}.(2) the set
$\IE_2(Y)$ consists of only diagonal elements. Thus
$\pi(y_1)=\pi(x_1y_1)$.
If $x', y'\in X$ and $q(x')=q(y')$,
then $x'(y')^{-1}\in \IE(X)$ and  hence $\pi(x')=\pi((x'(y')^{-1})y')=\pi(y')$.

(2). This can be proved using arguments similar to that for Assertion (1), using Theorem~\ref{T-basic measure IE} instead of Theorem~\ref{T-basic IE}.
\end{proof}

From Theorem~\ref{T-topological Pinsker} we get

\begin{corollary} \label{C-topological CPE}
Let $\Gamma$ act on a compact group $X$ by automorphisms. Then the following are equivalent:
\begin{enumerate}
\item $\IE(X)=X$.

\item The only topological factor $Y$ of $X$ with $\htopol(Y)=0$ is the trivial factor consisting of one point.

\item The only topological factor $Y$ of $X$ with $\rh_{\pi_*(\mu_X)}(Y)=0$ for $\pi: X\rightarrow Y$ denoting the factor map is
the trivial factor consisting of one point.
\end{enumerate}
\end{corollary}

Next we show that taking the $\IE$ group is an idempotent operation.

\begin{lemma} \label{L-IE is normal}
Let $\Gamma$ act on a compact  group $X$ by automorphisms.
Let $Y$ be a closed $\Gamma$-invariant normal subgroup of $X$. Then $\IE(Y)$ is a normal subgroup of $X$.
\end{lemma}
\begin{proof} Consider the conjugation map $\pi: X\times Y\rightarrow Y$ sending $(x, y)$ to $xyx^{-1}$.
Clearly $\pi$ is surjective and $\Gamma$-equivariant for $X\times Y$ equipped with the product action.
By (4) and (5) of Theorem~\ref{T-basic IE}  we have
$(\pi\times \pi)(\IE_2(X)\times \IE_2(Y))=(\pi\times \pi) (\IE_2(X\times Y))=\IE_2(Y)$.
Let $x\in X$ and $y\in \IE(Y)$. Then $(x, x)\in \IE_2(X)$ and $(y, e_Y)\in \IE_2(Y)$.
Thus
$$(xyx^{-1}, e_Y)=(xyx^{-1}, xe_Yx^{-1})=(\pi\times \pi)((x, x)\times (y, e_Y))\in \IE_2(Y),$$
and
hence $xyx^{-1}\in \IE(Y)$.
\end{proof}

\begin{proposition} \label{P-IE of IE}
Let $\Gamma$ act on a compact group $X$ by automorphisms.
Then $\IE(\IE(X))=\IE(X)$.
\end{proposition}
\begin{proof}
By Theorem~\ref{T-mIE=IE}.(1) and Lemma~\ref{L-IE is normal} the group $\IE(\IE(X))$ is closed and normal in $X$. By Theorem~\ref{T-topological Pinsker} we have
$\rh(X/\IE(X))=\rh(\IE(X)/\IE(\IE(X)))=0$. In virtue of
Proposition~\ref{P-addition formula} one has
$$\rh(X/\IE(\IE(X)))=\rh(X/\IE(X))+\rh(\IE(X)/\IE(\IE(X)))=0.$$ By Theorem~\ref{T-topological Pinsker} we
get $\IE(\IE(X))\supseteq \IE(X)$. Thus $\IE(\IE(X))=\IE(X)$.
\end{proof}

Now we describe the relation between $\Delta^1(X)$ and $\IE(X)$ for algebraic actions.

\begin{theorem} \label{T-homoclinic to IE}
Let $\Gamma$ act on a compact abelian group $X$ by automorphisms. Suppose that $\widehat{X}$ is a finitely generated left $\Zb\Gamma$-module. Then $\Delta^1(X)\subseteq \IE(X)$.
\end{theorem}

Before giving the proof of Theorem~\ref{T-homoclinic to IE}, we use it to give a partial answer to a question of Deninger.
For any countable discrete (not necessarily amenable) group $\Gamma$, and any invertible element $f$ in the group von Neumann algebra $\cL\Gamma$ of $\Gamma$,
the Fuglede-Kadison determinant $\det_{\cL\Gamma}f$ is defined \cite{FK}, which is a positive real number.  We refer the reader to \cite{Den09, Luck}
and \cite[Section 2.2]{Li} for background on $\cL\Gamma$ and $\det_{\cL\Gamma} f$. Deninger asked \cite[Question 26]{Den09} if $f\in \Zb\Gamma$ is invertible in
$\ell^1(\Gamma)$ and has no left inverse in $\Zb\Gamma$, then whether $\det_{\cL\Gamma} f>1$. This was answered affirmatively by Deninger and Schmidt \cite[Corollary 6.7]{DS}
in the case $\Gamma$ is residually finite and amenable. Now we answer Deninger's question for all countable amenable groups:

\begin{corollary} \label{C-answer to Deninger}
Suppose that $f\in \Zb\Gamma$ is invertible in
$\ell^1(\Gamma)$ and has no left inverse in $\Zb\Gamma$. Then $\det_{\cL\Gamma} f>1$.
\end{corollary}
\begin{proof}
Let $X_f$ be as in Notation~\ref{N-principal}. Since $f$ has no left inverse in $\Zb\Gamma$, the left $\Zb\Gamma$-module $\Zb\Gamma/\Zb\Gamma f$ is nontrivial, and hence $X_f$ consists of more than one point. As $f$ is invertible in $\ell^1(\Gamma)$, by Lemma~\ref{L-key expansive homoclinic} $\Delta^1(X_f)$ is dense in $X_f$ and hence is nontrivial. In virtue of Theorem~\ref{T-homoclinic to IE}, $\IE(X_f)$ is nontrivial. By Theorem~\ref{T-topological Pinsker} one has
$\rh(X_f)>0$. Theorem 1.1 of \cite{Li} states that for any $g\in \Zb\Gamma$ being invertible in $\cL\Gamma$, one has
$\rh(X_g)=\log \det_{\cL\Gamma}g$. As invertibility in $\ell^1(\Gamma)$ implies invertibility in $\cL\Gamma$, we get
$\log \det_{\cL\Gamma}f=\rh(X_f)>0$. Therefore $\det_{\cL\Gamma} f>1$.
\end{proof}

Theorem~\ref{T-homoclinic to IE} follows from Proposition~\ref{P- 1-expansive summable under metric} and the next result, which
 is inspired by the proof of \cite[Theorem 5.1]{DS}.

\begin{proposition} \label{P-l1 to IE}
Let $\Gamma$ act on a compact group $X$ by automorphisms.
Let $x\in X$ such that
$\sum_{s \in \Gamma} \rho(s x, e_X)<+\infty$ for some
compatible translation-invariant metric $\rho$ on $X$.
Then $x\in \IE(X)$.
\end{proposition}
\begin{proof}
Let $U_1$ and $U_0$ be neighborhoods of $x$ and $e_X$ in $X$ respectively. Then
there exists
$\varepsilon>0$ such
that $U_1\supseteq \{y\in X: \rho(y, x)<\varepsilon\}$
and $ U_0\supseteq \{y\in X: \rho(y, e_X)<\varepsilon\}$.
Since $\sum_{s \in \Gamma} \rho(s x, e_X)<+\infty$,
we can find a nonempty finite subset
$K$ of $\Gamma$ such that $\sum_{s\in \Gamma\setminus K}\rho(s x, e_X)<\varepsilon$.

Let $F$ be a nonempty finite subset of $\Gamma$. By Lemma~\ref{L-separated subset} there exists $F_1\subseteq F$ with
$\frac{|F_1|}{|F|}\ge \frac{1}{2|K|+1}$ and $((F_1F_1^{-1})\setminus \{e_{\Gamma}\})\subseteq \Gamma\setminus K$. Say, $F_1=\{s_1, \dots, s_{|F_1|}\}$. For each $\sigma\in \{0, 1\}^{F_1}$, set
$$y_{\sigma}=(s_1^{-1}x)^{\sigma(s_1)}(s_2^{-1}x)^{\sigma(s_2)}\cdots (s_{|F_1|}^{-1}x)^{\sigma(s_{|F_1|})}.$$
We claim that  $s(y_{\sigma})\in U_{\sigma(s)}$ for every $s \in F_1$. Let $s \in F_1$.  Since $\rho$ is translation invariant, we have
\begin{align*}
\rho(w_1w_2\cdots w_k, z_1z_2\cdots z_k)\le \sum_{i=1}^k\rho(w_i, z_i)
\end{align*}
for all $k\in \Nb$ and $w_1, \dots, w_k, z_1, \dots, z_k\in X$.
Thus
\begin{eqnarray*}
\rho(s(y_{\sigma}), x^{\sigma(s)})
&\le & \rho((ss^{-1}x)^{\sigma(s)}, x^{\sigma(s)})+\sum_{s'\in F_1\setminus \{s\}}\rho((s(s')^{-1}x)^{\sigma(s')}, e_X)\\
&=& \sum_{s'\in F_1\setminus \{s\}}\rho((s(s')^{-1}x)^{\sigma(s')}, e_X).
\end{eqnarray*}
Since $s(s')^{-1}\in \Gamma\setminus K$ for every $s'\in F_1\setminus \{s\}$, we get
$$\rho(s(y_{\sigma}), x^{\sigma(s)})
\le  \sum_{s''\in \Gamma\setminus K}\rho(s''x, e_X)<\varepsilon.$$
Therefore $s(y_{\sigma})\in U_{\sigma(s)}$. This proves the claim. Thus $F_1$ is an independence set for $(U_1, U_0)$.
Then $(x, e_X)\in \IE_2(X)$ and hence $x\in \IE(X)$.
\end{proof}

\section{IE Group and Pinsker Algebra} \label{S-Pinsker}

Throughout this section $\Gamma$ will be a countable amenable group.

Let $(X, \cB_X, \mu)$ be a standard probability space. That is, $(X, \cB_X)$ is a standard Borel space \cite[Section 12]{Kechris} and
$\mu$ is a probability measure on $\cB_X$. Let $\Gamma$ act on $(X, \cB_X, \mu)$ via measure-preserving automorphisms.
The {\it Pinsker algebra} of this action, denoted by $\Pi(X)$ or $\Pi(X, \cB_X, \mu)$, is the $\sigma$-algebra on $X$ consisting of
$A\in \cB_X$ such that $\rh_{\mu}(\{A, X\setminus A\})=0$. For two sub-$\sigma$-algebras $\cB_1$ and $\cB_2$ of
$\cB_X$, we write $\cB_1=\cB_2 \mod \mu$ if for every $A_1\in \cB_1$ there exists $A_2\in \cB_2$ with
$\mu(A_1\Delta A_2)=0$, and vice versa.

For a compact space $X$ (recall that all compact spaces are assumed to be metrizable), we denote by $\cB_X$ the $\sigma$-algebra of Borel subsets of $X$. Note that if $X$ is a
compact space and $\mu$ is a probability measure on $\cB_X$, then $(X, \cB_X, \mu)$ is a standard probability space.

Recall that we denote by $\mu_X$ the normalized Haar measure of a compact group $X$.
Also recall that when $\Gamma$ acts on a compact space $X$ continuously, we denote by $\cM(X, \Gamma)$ the set of all $\Gamma$-invariant Borel probability measures on $X$.

The following theorem is the main result of this section, saying that $\IE(X)$ determines the Pinsker algebra with respect to $\mu_X$.

\begin{theorem} \label{T-IE to Pinsker}
Let $\Gamma$ act on a compact group $X$ by automorphisms. Denote by $q$ the quotient map
$X\rightarrow X/\IE(X)$. Then the following hold:
\begin{enumerate}
\item $q^{-1}(\cB_{X/\IE(X)})=\Pi(X, \cB_X, \mu_X) \mod \mu_X$.

\item For any $\nu \in \cM(X, \Gamma)$, one has $q^{-1}(\cB_{X/\IE(X)})\subseteq \Pi(X, \cB_X, \mu_X)\subseteq \Pi(X, \cB_X, \nu)$.
\end{enumerate}
\end{theorem}

We shall need the following result of Danilenko \cite[Theorem 0.4]{Danilenko}, which was proved first by
Glasner, Thouvenot, and Weiss \cite[Theorem 4]{GTW} in the case that the actions of $\Gamma$ on both $(X, \cB_X, \mu_X)$ and $(Y, \cB_Y, \mu_Y)$
are free and ergodic. Though Danilenko assumed $\Gamma$ to be infinite in \cite{Danilenko}, the following result holds trivially when $\Gamma$ is finite, since in such case the Pinsker algebra consists of measurable sets with measure $0$ or $1$.

\begin{theorem} \label{T-Pinsker of product}
Let $\Gamma$ act on two standard probability spaces $(X, \cB_X, \mu_X)$ and $(Y, \cB_Y, \mu_Y)$ via measure-preserving automorphisms.
For the product action of $\Gamma$ on $(X\times Y, \cB_X\times \cB_Y, \mu_X\times \mu_Y)$, one has
$\Pi(X\times Y)=\Pi(X)\times \Pi(Y) \mod \mu_X\times \mu_Y$.
\end{theorem}

\begin{lemma} \label{L-Pinsker rotation}
Let $\Gamma$ act on a compact group $X$ by automorphisms. Let $\nu\in \cM(X, \Gamma)$. Then
$x\cdot \Pi(X, \cB_X, \mu_X), \Pi(X, \cB_X, \mu_X)\cdot x\subseteq \Pi(X, \cB_X, \nu)$ for  all $x\in X$.
\end{lemma}
\begin{proof} Denote by $\pi$ the product map $X\times X\rightarrow X$ sending $(x, y)$ to $xy$. By Lemma~\ref{L-product factor} this is a measure-theoretic factor map
$(X\times X,  \cB_X\times \cB_X, \mu_X\times \nu)\rightarrow (X, \cB_X, \mu_X)$. Then $\pi^{-1}(\Pi(X, \cB_X, \mu_X))\subseteq \Pi(X\times X,  \cB_X\times \cB_X, \mu_X\times \nu)$. Let $A\in \Pi(X, \cB_X, \mu_X)$. By Theorem~\ref{T-Pinsker of product} we can find
$B\in \Pi(X, \cB_X, \mu_X)\times \Pi(X, \cB_X, \nu)$ with $(\mu_X\times \nu)(\pi^{-1}(A)\Delta B)=0$.
Denote by $\chi_{\pi^{-1}(A)}$ the characteristic function of $\pi^{-1}(A)$. Then $\chi_{\pi^{-1}(A)}\in L^1(X\times X,  \Pi(X, \cB_X, \mu_X)\times \Pi(X, \cB_X, \nu), \mu_X\times \nu)$. By the Fubini Theorem \cite[Theorem 8.8]{Rudin}, the function $y\mapsto \chi_{\pi^{-1}(A)}(x, y)=\chi_A(xy)=\chi_{x^{-1}A}(y)$ is in $L^1(X, \Pi(X, \cB_X, \nu), \nu)$ for $\mu_X$ almost
all $x\in X$. That is, there exists $E\in \cB_X$ with $\mu_X(E)=0$ such
that the function $\chi_{x^{-1}A}$ is in $L^1(X, \Pi(X, \cB_X, \nu), \nu)$ for every $x\in X\setminus E$.

Since $\supp(\mu_X)=X$, the set $X\setminus E$ is dense in $X$. Let $x_0\in X$. Since
$X$ is metrizable, we can find a sequence $\{x_n\}_{n\in \Nb}$ in $X\setminus E$ with
$x_n\to x_0$ as $n\to \infty$. For each $n\in \Nb$, since the function $\chi_{x_n^{-1}A}$ is in
$L^1(X, \Pi(X, \cB_X, \nu), \nu)$, we can find some $A_n\in \Pi(X, \cB_X, \nu)$ such that $\nu(A_n\Delta (x_n^{-1}A) )=0$.
When $n\to \infty$, since $x_n\to x_0$ and $\nu$ is regular \cite[Theorem 6.1]{Walters},
we have $\nu((x_n^{-1} A)\Delta (x_0^{-1} A))\to 0$ and hence $\nu(A_n\Delta (x_0^{-1}A))\to 0$.
Passing to a subsequence of $\{x_n\}_{n\in \Nb}$ if necessary, we may assume that $\sum_{n\in \Nb}\nu(A_n\Delta (x_0^{-1}A))<+\infty$.
It follows that $\lim_{k\to \infty}\nu((x_0^{-1}A)\Delta (\bigcup_{n\ge k}A_n))=0$ and hence $\nu((x_0^{-1}A)\Delta (\bigcap_{k\in \Nb}\bigcup_{n\ge k}A_n))=0$. Note that if $A'\in \cB_X$ satisfies $\nu(A')=0$, then $A'\in \Pi(X, \cB_X, \nu)$. Since $\bigcap_{k\in \Nb}\bigcup_{n\ge k}A_n$ is in
$\Pi(X, \cB_X, \nu)$, we conclude that $x_0^{-1}A$ is in $\Pi(X, \cB_X, \nu)$. This proves $x\cdot \Pi(X, \cB_X, \mu_X)\subseteq \Pi(X, \cB_X, \nu)$ for  all $x\in X$. Similarly, one has $\Pi(X, \cB_X, \mu_X)\cdot x\subseteq \Pi(X, \cB_X, \nu)$ for  all $x\in X$.
\end{proof}

We are ready to prove Theorem~\ref{T-IE to Pinsker}.

\begin{proof}[Proof of Theorem~\ref{T-IE to Pinsker}]
(1). By Lemma~\ref{L-Pinsker rotation} we have $x\cdot \Pi(X, \cB_X, \mu_X), \Pi(X, \cB_X, \mu_X)\cdot x\subseteq \Pi(X, \cB_X, \mu_X)$ for  all $x\in X$. Thus, by \cite[Lemma 20.4]{Schmidt}, there is a closed $\Gamma$-invariant normal subgroup
$Y$ of $X$ such that $q_1^{-1}(\cB_{X/Y})=\Pi(X, \cB_X, \mu_X) \mod \mu_X$, where $q_1$ denotes the quotient map
$X\rightarrow X/Y$. In particular, $\rh_{(q_1)_*(\mu_X)}(X/Y)=0$. By Theorem~\ref{T-topological Pinsker}, there is a surjective continuous
map $q': X/\IE(X)\rightarrow X/Y$ such that $q'\circ q=q_1$. Clearly $q'$ is a group homomorphism and hence is open.

Every continuous open surjective map between compact metrizable spaces has a Borel cross section \cite[Theorem 3.4.1]{Arveson}.
Thus we can find a Borel map
$\psi: X/Y\rightarrow X/\IE(X)$ such that $q' \circ \psi$ is the identity map on $X/Y$.
It is easily verified that the map $\phi: X/Y\times \ker q'\rightarrow X/\IE(X)$ sending $(z, y)$ to $\psi(z)y$
is an isomorphism from the measurable space $(X/Y\times \ker q', \cB_{X/Y}\times \cB_{\ker q'})$ onto
the measurable space $(X/\IE(X), \cB_{X/\IE(X)})$.
We claim that $\phi_*(\mu_{X/Y}\times \mu_{\ker q'})$ is left-translation invariant.
Let $A\in \cB_{X/Y}\times \cB_{\ker q'}$ and $(z_1, y_1)\in (X/Y)\times \ker q'$. For each $z\in X/Y$, denote by $A_z$ the set $\{y\in \ker q': (z, y)\in A\}$.
Note that $A_z\in \cB_{\ker q'}$ for every $z\in X/Y$. For any
$(z_2, y_2)\in (X/Y)\times \ker q'$, we have
$$ \phi(z_1, y_1)\phi(z_2, y_2)=\phi(z_1z_2, \psi(z_1z_2)^{-1}\psi(z_1)y_1\psi(z_2)y_2).$$
Thus, for any $z_2\in X/Y$, one has $(\phi^{-1}(\phi(z_1, y_1)\phi(A)))_{z_1z_2}=\psi(z_1z_2)^{-1}\psi(z_1)y_1\psi(z_2)A_{z_2}$ and hence
$\mu_{\ker q'}((\phi^{-1}(\phi(z_1, y_1)\phi(A)))_{z_1z_2})=\mu_{\ker q'}(A_{z_2})$. Therefore,
\begin{align*}
(\mu_{X/Y}\times \mu_{\ker q'})(\phi^{-1}(\phi(z_1, y_1)\phi(A)))&=\int_{X/Y}\mu_{\ker q'}((\phi^{-1}(\phi(z_1, y_1)\phi(A)))_{z_1z_2})\, d\mu_{X/Y}(z_1z_2)\\
&=\int_{X/Y}\mu_{\ker q'}(A_{z_2})\, d\mu_{X/Y}(z_2)\\
&=(\mu_{X/Y}\times \mu_{\ker q'})(A).
\end{align*}
This proves our claim. Therefore $\phi_*(\mu_{X/Y}\times \mu_{\ker q'})=\mu_{X/\IE(X)}$. Also note that the measures $(q_1)_*(\mu_X), (q')_*(\mu_{X/\IE(X)})$, and $q_*(\mu_X)$ are all translation invariant,
and hence $(q_1)_*(\mu_X)=(q')_*(\mu_{X/\IE(X)})=\mu_{X/Y}$ and $q_*(\mu_X)=\mu_{X/\IE(X)}$.

We claim that $q'$ is an isomorphism. Suppose that $q'$ is not injective. Then we can find disjoint nonempty open subsets $U$ and $V$ of $\ker q'$.
Since $\supp(\mu_{\ker q'})=\ker q'$, we have $0<\mu_{\ker q'}(U)<1$. By Theorem~\ref{T-topological Pinsker} we have $\rh_{q_*(\mu_X)}(X/\IE(X))=0$, and hence $q^{-1}(\cB_{X/\IE(X)})\subseteq \Pi(X, \cB_X, \mu_X)$.
Note that $(X/Y)\times U \in \cB_{X/Y}\times \cB_{\ker q'}$, and hence $\phi((X/Y)\times U)=\psi(X/Y)U$ is in $\cB_{X/\IE(X)}$.
As  $q_1^{-1}(\cB_{X/Y})=\Pi(X, \cB_X, \mu_X) \mod \mu_X$, we can find some $A\in \cB_{X/Y}$ with
$\mu_X(q_1^{-1}(A)\Delta q^{-1}(\psi(X/Y)U))=0$. Then
\begin{eqnarray} \label{E-IE to Pinsker}
\mu_{X/\IE(X)}((q')^{-1}(A)\Delta (\psi(X/Y)U))&=&q_*(\mu_X)((q')^{-1}(A)\Delta (\psi(X/Y)U))\\
\nonumber &=&\mu_X(q^{-1}((q')^{-1}(A))\Delta q^{-1}(\psi(X/Y)U))\\
\nonumber &=&\mu_X(q_1^{-1}(A)\Delta q^{-1}(\psi(X/Y)U))=0.
\end{eqnarray}
Note that
\begin{eqnarray*}
\mu_{X/\IE(X)}(\psi(X/Y)U)&=&\phi_*(\mu_{X/Y}\times \mu_{\ker q'})(\phi((X/Y)\times U))\\
&=&\mu_{X/Y}(X/Y)\cdot \mu_{\ker q'}(U)\\
&=& \mu_{\ker q'}(U)>0,
\end{eqnarray*}
and hence
\begin{eqnarray*}
\mu_{X/Y}(A)=q'_*(\mu_{X/\IE(X)})(A)=\mu_{X/\IE(X)}((q')^{-1}(A))=\mu_{X/\IE(X)}( \psi(X/Y)U)>0.
\end{eqnarray*}
Then
\begin{eqnarray*}
& & \mu_{X/\IE(X)}((q')^{-1}(A)\cap (\psi(X/Y)U))\\
&=& \mu_{X/\IE(X)}((\psi(A)\ker q')\cap (\psi(X/Y)U))\\
&=&\phi_*(\mu_{X/Y}\times \mu_{\ker q'})(\phi(A\times \ker q')\cap \phi((X/Y)\times U))\\
&=&\phi_*(\mu_{X/Y}\times \mu_{\ker q'})(\phi((A\times \ker q')\cap ((X/Y)\times U))\\
&=&\phi_*(\mu_{X/Y}\times \mu_{\ker q'})(\phi(A\times U))\\
&=&\mu_{X/Y}(A)\cdot \mu_{\ker q'}(U)\\
&<& \mu_{X/Y}(A)=\mu_{X/\IE(X)}( \psi(X/Y)U),
\end{eqnarray*}
contradict to the equality \eqref{E-IE to Pinsker}. Therefore $q'$ is an isomorphism.
Then $q^{-1}(\cB_{X/\IE(X)})=(q_1)^{-1}(\cB_{X/Y})=\Pi(X, \cB_X, \mu_X) \mod \mu_X$.

(2). Since $\Pi(X, \cB_X, \mu_X)$ contains all $A\in \cB_X$ with $\mu_X(A)=0$, from the assertion (1) we conclude that
$ q^{-1}(\cB_{X/\IE(X)})\subseteq \Pi(X, \cB_X, \mu_X)$. Taking $x=e_X$ in Lemma~\ref{L-Pinsker rotation}, we get
$\Pi(X, \cB_X, \mu_X)\subseteq \Pi(X, \cB_X, \nu)$.
\end{proof}

We say that an action of $\Gamma$ on a compact group $X$ by automorphisms has {\it CPE (completely positive entropy)} if
$\Pi(X, \cB_X, \mu_X)$ consists of Borel sets with $\mu_X$-measure $0$ or $1$. From Theorem~\ref{T-IE to Pinsker} we get

\begin{corollary} \label{C-CPE}
Let $\Gamma$ act on a compact group $X$ by automorphisms. Then
$\IE(X)=X$ if and only if the action has CPE.
\end{corollary}

In the case $\Gamma=\Zb^d$ for some $d\in \Nb$ and $X$ is abelian, the following corollary was proved
by Lind, Schmidt, and Ward \cite[Corollary 6.6]{LSW} \cite[Corollary 20.9]{Schmidt}.

\begin{corollary} \label{C-CPE extension}
Let $\Gamma$ act on a compact group $X$ by automorphisms
and let $Y$ be a closed $\Gamma$-invariant normal subgroup of $X$. Suppose that both
the restriction of the action on $Y$
and the induced action
on $X/Y$
have CPE.
Then the action
itself has CPE.
\end{corollary}
\begin{proof} By Corollary~\ref{C-CPE} we have $\IE(Y)=Y$ and $\IE(X/Y)=X/Y$, and it suffices
to show that $\IE(X)=X$.
From the definition of IE tuples we have $\IE_2(Y)\subseteq \IE_2(X)$,
and hence $Y=\IE(Y)\subseteq \IE(X)$. By Theorem~\ref{T-mIE=IE}.(5) one has
$\IE(X)/Y=\IE(X/Y)=X/Y$. Therefore $\IE(X)=X$ as desired.
\end{proof}

In the rest of this section we discuss when a $\Gamma$-action on a compact group by automorphisms has a unique maximal measure.

\begin{theorem} \label{T-Berg}
Let $\Gamma$ act on a compact group $X$ by automorphisms.
Consider the following conditions:
\begin{enumerate}
\item the action has CPE;

\item $\rh_{\nu}(X)<\rh_{\mu_X}(X)$ for every $\nu\in \cM(X, \Gamma)$ not equal to $\mu_X$.
\end{enumerate}
Then (2)$\Rightarrow$(1). If furthermore $\rh(X)<\infty$, then (1)$\Leftrightarrow$(2).
\end{theorem}

For the case $\Gamma=\Zb$, Theorem~\ref{T-Berg} was proved by Berg \cite{Berg}.
Yuzvinski\u{\i} \cite{Yuz1} showed that when $\Gamma=\Zb$ and the action has finite entropy,
the condition (1) is also equivalent to that $\mu_X$ is ergodic.

When $\Gamma=\Zb^d$ for some $d\in \Nb$, Lind, Schmidt, and Ward \cite[Theorem 6.14]{LSW} proved Theorem~\ref{T-Berg} for the case $X$ is abelian, and later Schmidt \cite[Theorem 20.15]{Schmidt} established the general case. Ledrappier \cite{Ledrappier} showed that, for $\Gamma=\Zb^2$,  the canonical $\Gamma$-action on $X=\widehat{\Zb\Gamma/J}$ is mixing with respect to $\mu_X$ and has zero entropy, where $J=2\Zb\Gamma+(1-u_1-u_2)\Zb\Gamma$ and
$u_1, u_2$ denote the canonical basis of $\Zb^2$.

For a standard probability space $(X, \cB_X, \mu)$, we say that two $\sigma$-algebras $\cB_1, \cB_2\subseteq \cB_X$ are {\it independent} if
$\mu(A\cap B)=\mu(A)\mu(B)$ for all $A\in \cB_1$ and $B\in \cB_2$. We need the following result of Danilenko \cite[Theorem 0.4]{Danilenko} (in the statement of Theorem 0.4 of
\cite{Danilenko}, the condition $\rh_\mu(X)<+\infty$ is missing).

\begin{theorem} \label{T-CPE to independent}
Let $\Gamma$ act on a standard probability space $(X, \cB_X, \mu)$ via measure-preserving automorphisms. Suppose that $\rh_\mu(X)<+\infty$.
Let $\cB_1$ and $\cB_1$ be  $\Gamma$-invariant sub-$\sigma$-algebras of $\cB_X$. Then $\cB_1$ and $\cB_2$ are independent
if and only if $\Pi(X, \cB_1, \mu)$ and $\Pi(X, \cB_2, \mu)$ are independent and
$$\rh_{\mu}(\cB_1\vee \cB_2)=\rh_{\mu}(\cB_1)+\rh_{\mu}(\cB_2).$$
\end{theorem}

We are ready to prove Theorem~\ref{T-Berg}.

\begin{proof} Assume that the condition (2) holds. By Proposition~\ref{P-addition formula} and Theorem~\ref{T-topological Pinsker}
we have
$$\rh(X)=\rh(\IE(X))+\rh(X/\IE(X))=\rh(\IE(X))=\rh_{\mu_{\IE(X)}}(\IE(X))=\rh_{\mu_{\IE(X)}}(X).$$
Thus $\mu_{\IE(X)}=\mu_X$, and hence $\IE(X)=X$. By Corollary~\ref{C-CPE} the condition (1) holds.

Now assume that $\rh(X)<\infty$ and that the condition (1) holds. Let $\nu\in \cM(X, \Gamma)$ with $\rh_{\nu}(X)=\rh_{\mu_X}(X)$. We shall show that
$\nu=\mu_X$.
Denote by $\pi_1$ and $\pi$ the first coordinate
map $X\times X\rightarrow X$ sending $(x, y)$ to $x$ and the product map $X\times X\rightarrow X$ sending $(x, y)$ to $xy$ respectively.
Set $\cB_1=\pi^{-1}_1(\cB_X)$ and $\cB_2=\pi^{-1}(\cB_X)$.
As $X$ is compact metrizable, $\cB_{X\times X}=\cB_X\times \cB_X$.
By \cite[Lemma 1.2]{Berg},
both
$\cB_1$ and $\cB_2$ are $\Gamma$-invariant sub-$\sigma$-algebras of $\cB_{X\times X}$,
and $\cB_1\vee \cB_2=\cB_X\times \cB_X$.
The condition (1) says that $\Pi(X, \cB_X, \mu_X)$ consists of elements in $\cB_X$ with
$\mu_X$-measure $0$ or $1$. Then $\Pi(X\times X, \cB_1, \mu_X\times \nu)=\pi_1^{-1}(\Pi(X, \cB_X, \mu_X))$ consists of elements
in $\cB_{X\times X}$ with $\mu_X\times \nu$-measure $0$ or $1$. Thus
$\Pi(X\times X, \cB_1, \mu_X\times \nu)$ and $\Pi(X\times X, \cB_2, \mu_X\times \nu)$ are independent
under $\mu_X\times \nu$. Note that
$$ \rh_{\mu_X\times \nu}(\cB_1\vee \cB_2)=\rh_{\mu_X\times \nu}(\cB_X\times \cB_X)=\rh_{\mu_X}(\cB_X)+\rh_{\nu}(\cB_X)=2\rh_{\mu_X}(\cB_X),$$
and by Lemma~\ref{L-product factor},
$$ \rh_{\mu_X\times \nu}(\cB_1)+\rh_{\mu_X\times \nu}(\cB_2)=\rh_{\mu_X}(\cB_X)+\rh_{\mu_X}(\cB_X)=2\rh_{\mu_X}(\cB_X).$$
Thus $$\rh_{\mu_X\times \nu}(\cB_1\vee \cB_2)=\rh_{\mu_G\times \nu}(\cB_1)+\rh_{\mu_X\times \nu}(\cB_2).$$
By Theorem~\ref{T-CPE to independent} we see that $\cB_1$ and $\cB_2$ are independent with respect to $\mu_X\times \nu$.
From \cite[Lemma 2.6]{Berg} or \cite[Lemma 20.17]{Schmidt} we conclude that $\mu_X=\nu$.
\end{proof}

\section{Duality} \label{S-duality}

Throughout this section $\Gamma$ will be a countable amenable group.

Let $\Gamma$ act on a compact abelian group $X$ by automorphisms, and $1\le p\le +\infty$.
We shall treat $\Delta^p(X)$ and its $\Gamma$-invariant subgroups $G$ as discrete abelian groups, thus
consider the induced  $\Gamma$-action on the Pontryagin dual $\widehat{\Delta^p(X)}$ and $\widehat{G}$ by automorphisms.
The pair $(\widehat{X}, G)$ will be treated as a dual pair as at the end of Section~\ref{S-p-homoclinic}.

We  first give some conditions for  $\rh(X)$ to be  bounded below by $\rh(\widehat{\Delta^p(X)})$. The definition of entropy is recalled in Section~\ref{SS-entropy}.

\begin{theorem} \label{T-bounded below}
Let $k, n\in \Nb$ and $A\in M_{n\times k}(\Zb\Gamma)$. Let $X$ be a closed $\Gamma$-invariant subgroup of $\widehat{(\Zb\Gamma)^k/(\Zb\Gamma)^nA}$. Let $1\le p<+\infty$. Suppose that one of the following conditions holds:
\begin{enumerate}
\item $p=1$ and the linear map $(\ell^p(\Gamma))^k\rightarrow (\ell^p(\Gamma))^n$ sending $a$ to $aA^*$ is injective.

\item There exists $C>0$ such that $\|a\|_p\le C\|aA^*\|_p$ for all $a\in (\ell^p(\Gamma))^k$, where the norm $\|\cdot \|_p$ is defined by the equation \eqref{E-infinity norm}.
\end{enumerate}
Then
$$ \rh(X)\ge \rh(\widehat{\Delta^p(X)}).$$
\end{theorem}

To prove Theorem~\ref{T-bounded below}, we need the following lemma, of which the case $p=2$ appeared in \cite[Lemma 5.1]{Li}.

\begin{lemma} \label{L-ball}
Let $1\le p<+\infty$. There exists some universal constant $C_p>0$ such that for any
$\lambda>1$,
there is some $\delta>0$ so that for any nonempty finite
set $Y$, any positive integer $n$ with $|Y|\le \delta n$, and any $M\ge 1$ one has
$$ |\{x\in \Zb^Y: \|x\|_p\le M\cdot n^{1/p}\}|\le
C_p\lambda^{n}M^{|Y|}.$$
\end{lemma}
\begin{proof}
Let $\delta>0$ be a small number less than $e^{-1}$ which we
shall determine in a moment.
Let $Y$ be a nonempty finite set, $n$ be a positive integer
with $|Y|\le \delta n$, and  $M\ge 1$. For each $x\in \Zb^Y$, denote $\{z\in
\Rb^Y: 0\le z_{y}-x_{y}\le 1 \mbox{ for all } y \in Y\}$ by
$D_x$. Denote $\{x\in \Zb^Y: \|x\|_p\le M\cdot n^{1/p}\}$ by $S$
and denote the union of $D_x$ for all $x\in S$ by $D_S$. Then the
(Euclidean) volume of $D_S$ is equal to $|S|$. Note that for any $z\in D_S$, say $z\in D_x$, one has
$$\|z\|_p\le \|x\|_p+\|z-x\|_p\le M\cdot n^{1/p}+n^{1/p}\le 2Mn^{1/p}.$$

A simple calculation shows that the function
$\varsigma(t):=(n/t)^{t/p}$ is increasing for $0<t\le ne^{-1}$.
The volume of the unit ball of $\Rb^Y$ under $\| \cdot \|_p$ is $\frac{(2/p)^{|Y|}(\Gamma(1/p))^{|Y|}}{(|Y|/p)\Gamma(|Y|/p)}$
\cite[page 394]{Wang}, where $\Gamma$ denotes the gamma function. By Stirling's formula \cite[page 423]{Lang99} there exists some
constant $C'>0$ such that $\Gamma(t)\ge C'\sqrt{2\pi} t^{t-1/2}e^{-t}$ for
all $t\ge 1/p$. Thus the volume of $D_S$ is no bigger than
\begin{align*}
\frac{(2/p)^{|Y|}(\Gamma(1/p))^{|Y|}(2Mn^{1/p})^{|Y|}}{(|Y|/p)\Gamma(|Y|/p)}&\le \frac{(2/p)^{|Y|}(\Gamma(1/p))^{|Y|}(2Mn^{1/p})^{|Y|}}{(|Y|/p)C' \sqrt{2\pi} (|Y|/p)^{|Y|/p-1/2}e^{-|Y|/p}}\\
&\le |Y|^{-1/2}C_p\tilde{C}^{|Y|}(n/|Y|)^{|Y|/p}M^{|Y|}\\
&\le C_p\tilde{C}^{|Y|}\varsigma(|Y|)M^{|Y|}\\
&\le C_p\tilde{C}^{\delta n}\varsigma(\delta n) M^{|Y|}=C_p\tilde{C}^{\delta n}\delta^{-\delta n/p}M^{|Y|},
\end{align*}
where $C_p=\sqrt{p/(2\pi)}/C'$ and $\tilde{C}=\max(4 e^{1/p}p^{(1-p)/p}\Gamma(1/p), 1)$. Take $\delta>0$ so
small that $\tilde{C}^{\delta}\delta^{-\delta/p}\le \lambda$. Then the
volume of $D_S$ is no bigger than $C_p\lambda^{n}M^{|Y|}$. Consequently,
$|S|\le C_p\lambda^{n}M^{|Y|}$.
\end{proof}

Let $\Gamma$ act on a compact abelian group $X$ by automorphisms.
For any nonempty finite subset $E$ of $\widehat{X}$, the function $F\mapsto \log |\sum_{s\in F}s^{-1}E|$ defined on the set of nonempty finite subsets of $\Gamma$ satisfies
the conditions of the Ornstein-Weiss lemma \cite[Theorem 6.1]{LW}, thus $\frac{\log |\sum_{s\in F}s^{-1}E|}{|F|}$ converges to some real number $c$, denoted by $\lim_F\frac{\log |\sum_{s\in F}s^{-1}E|}{|F|}$, when $F$ becomes more and more left invariant. That is,
for any $\varepsilon>0$, there exist a nonempty finite subset $K$ of $\Gamma$ and $\delta>0$ such that for any
nonempty finite subset $F$ of $\Gamma$ satisfying $|KF\setminus F|\le \delta |F|$ one has
$|\frac{\log |\sum_{s\in F}s^{-1}E|}{|F|}-c|<\varepsilon$.
 We need the following beautiful result of Peters \cite[Theorem 6]{Peters}:

\begin{theorem} \label{T-Peters}
Let $\Gamma$ act on a compact abelian group $X$ by automorphisms. Then
$$ \rh(X)=\sup_E\lim_F\frac{\log |\sum_{s\in F}s^{-1}E|}{|F|},$$
where $E$ ranges over all nonempty finite subsets of $\widehat{X}$.
\end{theorem}

In \cite{Peters}, Theorem~\ref{T-Peters} was stated and proved only for the case $\Gamma=\Zb$, but the proof there works for general countable amenable groups.

We are ready to prove Theorem~\ref{T-bounded below}.

\begin{proof}[Proof of Theorem~\ref{T-bounded below}]
Fix a compatible translation-invariant  metric $\rho$ on $X$. Denote by $K$ the support of $A$ as a $M_{n\times k}(\Zb)$-valued function on $\Gamma$.
When $\Gamma$ is finite and acts on a compact space $Y$ continuously, one has $\htopol(Y)=|Y|/|\Gamma|$ when $Y$ is a finite set and $\htopol(Y)=+\infty$ otherwise. Thus we may assume that $\Gamma$ is infinite.

By Theorem~\ref{T-Peters} it suffices to show
$$ \lim_F\frac{\log |\sum_{s\in F}s^{-1}E|}{|F|}\le \rh(X)+\delta$$
for every nonempty finite subset $E$ of $\Delta^p(X)$ and every $\delta>0$. Fix such $E$ and $\delta$.
Recall the canonical metric $\rho_\infty$ on $(\Rb/\Zb)^k$ defined in \eqref{E-metric on torus}.
Take $\varepsilon>0$ such that
for any $x\in X$ with $\rho(x, 0_X)\le \varepsilon$ one has $\rho_\infty(x_{e_\Gamma}, 0_{(\Rb/\Zb)^k})\le (2\|A\|_1)^{-1}$.
It suffices to show
$$|\sum_{s\in F}s^{-1}E|\le N_{\rho, F, \varepsilon}(X) \exp(\delta|F|)$$
for all sufficiently left invariant nonempty finite subsets $F$ of $\Gamma$.

Set $E'=E-E\subseteq \Delta^p(X)$. Denote by $B_{F, \varepsilon}$ the set of all $x\in X$ satisfying
$\max_{s\in F}\rho(sx, 0_X)\le \varepsilon$.
Take a maximal $(\rho, F, \varepsilon)$-separated subset $V_F$ of $\sum_{s\in F}s^{-1}E$. Then for any $x\in \sum_{s\in F}s^{-1}E$, since $\rho$ is translation-invariant, one can find some $y\in V_F$ with $x-y\in B_{F, \varepsilon}$. Note that
$x-y \in \sum_{s\in F}s^{-1}E'$. It follows that
$$ |\sum_{s\in F}s^{-1}E|\le |V_F| |B_{F, \varepsilon}\cap \sum_{s\in F}s^{-1}E'|\le N_{\rho, F, \varepsilon}(X)|B_{F, \varepsilon}\cap \sum_{s\in F}s^{-1}E'|.$$
Thus it suffices to show
\begin{align} \label{E-bounded below}
 |B_{F, \varepsilon}\cap \sum_{s\in F}s^{-1}E'|\le \exp(\delta|F|)
\end{align}
for all sufficiently left invariant nonempty finite subsets $F$ of $\Gamma$.

Denote by $P$ the canonical projection map $\ell^\infty(\Gamma, \Rb^k)\rightarrow ((\Rb/\Zb)^k)^\Gamma$.
For each $w\in E'$, take $\tilde{w}\in \ell^\infty(\Gamma, \Rb^k)$ with $P(\tilde{w})=w$ and
$\|\tilde{w}_s\|_\infty=\rho_\infty(w_s, 0_{(\Rb/\Zb)^k})$ for all $s\in \Gamma$.
Since $w\in \Delta^p(X)$, by Proposition~\ref{P-basic p-homoclinic}.(4) one has $\tilde{w}\in \ell^p(\Gamma, \Rb^k)$. Set $\tilde{E}=\{\tilde{w}: w\in E'\}$.
 For each $\tilde{w}\in \tilde{E}$, one has $\tilde{w}A^*\in \ell^\infty(\Gamma, \Zb^n)\cap \ell^p(\Gamma, \Rb^n)=\Zb^n \Gamma$. Denote by $K_1$ the finite subset $\bigcup_{\tilde{w}\in \tilde{E}}\supp(\tilde{w}A^*)$ of $\Gamma$. Note that for any nonempty finite subset $F$ of $\Gamma$ and any $\tilde{x}\in \sum_{s\in F}s^{-1}\tilde{E}$, one has
 $\supp(\tilde{x}A^*)\subseteq F^{-1}K_1$.

Let $F$ be a nonempty finite subset of $\Gamma$. Let $x\in B_{F, \varepsilon}\cap \sum_{s\in F}s^{-1}E'$. Take $x'\in \ell^\infty(\Gamma, \Rb^k)$ with $P(x')=x$ and $\|x'_s\|_\infty=\rho_\infty(x_s, 0_{(\Rb/\Zb)^k})$ for all $s\in \Gamma$. Since $x\in \Delta^p(X)$, by Proposition~\ref{P-basic p-homoclinic}.(4) one has $x'\in \ell^p(\Gamma, \Rb^k)$.
Set $F':=\{s\in F: s^{-1}K\subseteq F^{-1}\}$.
As $x\in B_{F, \varepsilon}$, by our choice of $\varepsilon$ one has $\|x'_{t}\|_\infty=\rho_\infty(x_t, 0_{(\Rb/\Zb)^k})\le (2\|A\|_1)^{-1}$ for every $t\in F^{-1}$, and hence $$\|(x'A^*)_{s}\|_\infty\le (\max_{t\in F^{-1}}\|x'_{t}\|_\infty) \|A^*\|_1\le 1/2$$
for all $s\in (F')^{-1}$. Since $x\in X$,  $x' A^*$ has integral coefficients. Thus $x'A^*=0$ on $(F')^{-1}$.

Since $P(\tilde{E})=E'$, we can find $\tilde{x}\in \sum_{s\in F}s^{-1}\tilde{E}$ with $P(\tilde{x})=x$.
Define $\bar{x}\in \ell^p(\Gamma, \Rb^k)$ to be the same as $x'$ on $F^{-1}$ and the same as $\tilde{x}$ on $\Gamma \setminus F^{-1}$.
Then $\bar{x}A^*=x'A^*=0$ on $(F')^{-1}$. Also, $\bar{x}A^*=\tilde{x}A^*$ on $\Gamma\setminus (F^{-1}K^{-1})$, and hence $\bar{x}A^*=\tilde{x}A^*=0$ on $\Gamma\setminus (F^{-1}(K_1\cup K^{-1}))$. Therefore $\supp(\bar{x}A^*)\subseteq (F^{-1}(K_1\cup K^{-1}))\setminus (F')^{-1}$.

Since
$P(\tilde{x})=P(x')=x$, we have $P(\bar{x})=x\in X$, and hence $\bar{x}A^*$ has integral coefficients.

Now we separate two cases.

Assume first that the condition (1) holds. Set $D=\sum_{\tilde{w}\in \tilde{E}}\|\tilde{w}\|_1$.
Note that $\|\tilde{x}\|_\infty\le D$ and hence
$$ \|\bar{x}\|_\infty\le \max(\|\tilde{x}\|_\infty, \|x'\|_\infty)\le D+1.$$
Thus
$$ \|\bar{x}A^*\|_\infty\le \|\bar{x}\|_\infty \cdot \|A\|_1\le (D+1)\|A\|_1.$$
Then the number of possible $\bar{x}A^*$ is at most $(2(D+1)\|A\|_1+1)^{n|(F^{-1}(K_1\cup K^{-1}))\setminus (F')^{-1}|}$. Since the map $(\ell^1(\Gamma))^k\rightarrow (\ell^1(\Gamma))^n$ sending $a$ to $aA^*$ is injective, the number of possible $\bar{x}$ is also bounded above by $(2(D+ 1)\|A\|_1+1)^{n|(F^{-1}(K_1\cup K^{-1}))\setminus (F')^{-1}|}$. As $P(\bar{x})=x$, we obtain
$$ |B_{F, \varepsilon}\cap \sum_{s\in F}s^{-1}E'|\le (2(D+1)\|A\|_1+1)^{n|(F^{-1}(K_1\cup K^{-1}))\setminus (F')^{-1}|}.$$

When $F$ is sufficiently left invariant, the right hand side of the above inequality is bounded above by $\exp(\delta |F|)$, and hence \eqref{E-bounded below} holds.

Next we assume that the condition (2) holds. Set $D'=\sum_{\tilde{w}\in \tilde{E}}\|\tilde{w}A^*\|_1$. Note that $\|\tilde{x}A^*\|_\infty\le D'$, and hence
$\|\tilde{x}\|_p\le C\|\tilde{x}A^*\|_p\le CD'|F^{-1}K_1|^{1/p}$. Define $\hat{x}\in \ell^p(\Gamma, \Rb^k)$ to be the same as $\bar{x}$ on $((F^{-1}(K_1\cup K^{-1}))\setminus (F')^{-1})K$ and $0$ on all other points of $\Gamma$. Then $\hat{x}A^*=\bar{x}A^*$ on $(F^{-1}(K_1\cup K^{-1}))\setminus (F')^{-1}$. Note that
\begin{align*}
\|\hat{x}\|_p&\le \|\tilde{x}\|_p+|((F^{-1}(K_1\cup K^{-1}))\setminus (F')^{-1})K|^{1/p}\\
& \le CD'|F^{-1}K_1|^{1/p}+|((F^{-1}(K_1\cup K^{-1}))\setminus (F')^{-1})K|^{1/p}.
\end{align*}
Since $\bar{x}A^*$ has support in $(F^{-1}(K_1\cup K^{-1}))\setminus (F')^{-1}$, we get
\begin{align*}
\|\bar{x}A^*\|_p&\le \|\hat{x}A^*\|_p\le \|\hat{x}\|_p \|A^*\|_1\\
&\le (CD'|F^{-1}K_1|^{1/p}+|((F^{-1}(K_1\cup K^{-1}))\setminus (F')^{-1})K|^{1/p})\|A\|_1\\
&\le (2CD'+1)\|A\|_1|F|^{1/p},
\end{align*}
when $F$ is sufficiently left invariant. By Lemma~\ref{L-ball}, for any $\lambda>1$,  when $F$ is sufficiently left invariant, the number of $y\in \Zb^n\Gamma$ with support in $(F^{-1}(K_1\cup K^{-1}))\setminus (F')^{-1}$ and $\|y\|_p\le (2CD'+1)\|A\|_1|F|^{1/p}$ is at most $$C_p^n\lambda^{n|F|}((2CD'+1)\|A\|_1)^{n|(F^{-1}(K_1\cup K^{-1}))\setminus (F')^{-1}|}.$$
Since $\Gamma$ is infinite, it follows that when $F$ is sufficiently left invariant, the number of $\bar{x}A^*$ is at most $\exp(\delta |F|)$.
As in the first case, one concludes that the inequality \eqref{E-bounded below} holds.
\end{proof}

\begin{question} \label{Q-p-expansive}
Could one weaken the conditions (1) and (2) of Theorem~\ref{T-bounded below} to that the linear map
$(\ell^p(\Gamma))^k\rightarrow (\ell^p(\Gamma))^n$ sending $a$ to $aA^*$ is injective?
\end{question}

From Theorems~\ref{T-finite entropy vs 1-expansive} and \ref{T-bounded below}, and Proposition~\ref{P-basic p-expansive}.(5) we get

\begin{corollary} \label{C-finitely presented to 1-homoclinic}
Let $\Gamma$ act on a compact abelian group $X$ by automorphisms such that $\widehat{X}$ is a finitely presented left $\Zb\Gamma$-module. Then
$$ \rh(X)\ge \rh(\widehat{\Delta^1(X)}).$$
\end{corollary}

From Theorems~\ref{T-algebraic characterization of expansive}, \ref{T-homoclinic are 1-homoclinic} and \ref{T-bounded below}, we get

\begin{corollary} \label{C-expansive dual}
Let $\Gamma$ act on a compact abelian group $X$ expansively by automorphisms.
Then
$$ \rh(X)\ge \rh(\widehat{\Delta(X)}).$$
\end{corollary}

Let $A\in M_k(\Zb\Gamma)$ for some $k\in \Nb$. Let $1<p, q<+\infty$ with $p^{-1}+q^{-1}=1$. One may identify $\ell^q(\Gamma, \Rb^k)$ with the dual space
of $\ell^p(\Gamma, \Rb^k)$ naturally, as using the pairing in \eqref{E-pairing}.
For the bounded linear map $T:\ell^p(\Gamma, \Rb^k)\rightarrow \ell^p(\Gamma, \Rb^k)$ sending $a$ to $aA^*$, its dual $T^*: \ell^q(\Gamma, \Rb^k)\rightarrow \ell^q(\Gamma, \Rb^k)$ sends $b$ to $bA$. Thus $T$ is invertible exactly when $T^*$ is invertible.
From Theorem~\ref{T-bounded below} and Lemma~\ref{L-key expansive homoclinic} we get

\begin{corollary} \label{C-key p-expansive entropy}
Let $k\in \Nb$, and  $A\in M_k(\Zb\Gamma)$
such that the linear map $(\ell^p(\Gamma))^k\rightarrow (\ell^p(\Gamma))^k$
sending $a$ to $aA^*$ is invertible for some $1< p<+\infty$.
Set $X_A=\widehat{(\Zb\Gamma)^k/(\Zb\Gamma)^kA}$ and $X_{A^*}=\widehat{(\Zb\Gamma)^k/(\Zb\Gamma)^kA^*}$.
Then
$$ \rh(X_A)=\rh(\widehat{\Delta^p(X_A)})=\rh(X_{A^*}).$$
\end{corollary}

\begin{remark} \label{R-invertible}
Since $\Gamma$ is amenable, by a result of Herz \cite[Theorem C]{Herz71} \cite[Theorem 5]{Herz73}, for any $k\in \Nb$, and any $1\le p\le q\le 2$ or $2\le q\le p<+\infty$, every bounded linear map
$\ell^p(\Gamma, \Rb^k)\rightarrow \ell^p(\Gamma, \Rb^k)$ commuting with the left-shift action of $\Gamma$ can be thought of a bounded linear map
 $\ell^q(\Gamma, \Rb^k)\rightarrow \ell^q(\Gamma, \Rb^k)$. It follows that, for any $A\in M_k(\Zb\Gamma)$ and $1\le p\le q\le 2$ or $2\le q\le p<+\infty$,
 if the linear map  $\ell^p(\Gamma, \Rb^k)\rightarrow \ell^p(\Gamma, \Rb^k)$ sending $a$ to $aA^*$ is invertible, then so is the linear map
  $\ell^q(\Gamma, \Rb^k)\rightarrow \ell^q(\Gamma, \Rb^k)$ sending $a$ to $aA^*$.
\end{remark}

\begin{remark} \label{R-dual}
 For $A\in M_k(\Zb\Gamma)$, note that the linear map $(\ell^2(\Gamma))^k\rightarrow (\ell^2(\Gamma))^k$ sending $a$ to $aA^*$ is invertible exactly when $A$ is invertible in $M_k(\cL\Gamma)$, where $\cL\Gamma$ denotes the group von Neumann algebra of $\Gamma$.
In \cite{Li} it is shown that, when $f\in \Zb\Gamma$ is invertible in $\cL\Gamma$, $\rh(X_f)$ can be calculated using the Fuglede-Kadison determinant of $f$, and as a consequence, $\rh(X_f)=\rh(X_{f^*})$ \cite[Corollary 9.2]{Li}.
Corollary~\ref{C-key p-expansive entropy} yields a new proof of this consequence, and in turn a new proof of \cite[Theorem 1.1]{Li}.
\end{remark}

Recall our convention of CPE before Corollary~\ref{C-CPE}.

\begin{theorem} \label{T-dual entropy for Noetherian 1-homoclinic}
Suppose that $\Zb\Gamma$ is left Noetherian.
 Let $\Gamma$ act on a compact abelian group $Y$ $1$-expansively by automorphisms such that
 $\Delta^1(Y)$ is dense in $Y$. Let $X$ be a closed $\Gamma$-invariant subgroup of $Y$. Then the following hold:
\begin{enumerate}
\item For any $\Gamma$-invariant subgroup $G$ of $\Delta^1(X)$ with $\overline{G}=\overline{\Delta^1(X)}$, one has
$\rh(X)=\rh(\widehat{G})$.

\item $\Delta^1(X)$ is a dense subgroup of $\IE(X)$.

\item The action $\Gamma \curvearrowright X$ has positive entropy if and only if $\Delta^1(X)$ is nontrivial.

\item The action $\Gamma \curvearrowright X$ has CPE if and only if $\Delta^1(X)$ is dense in $X$.
\end{enumerate}
\end{theorem}
\begin{proof} (1). We show first $\rh(Y)=\rh(\widehat{G_1})$ for any $\Gamma$-invariant subgroup $G_1$ of $\Delta^1(Y)$ satisfying
$\overline{G_1}=Y$. Denote by $\Tb$ the unit circle in $\Cb$. The canonical pairing $Y\times \widehat{Y}\rightarrow \Tb$ restricts to
a pairing $G_1\times \widehat{Y}\rightarrow \Tb$ which is bi-additive and equivariant in the sense defined before Lemma~\ref{L-pairing}.
Since $G_1$ is dense in $Y$, by Lemma~\ref{L-pairing}, the induced $\Gamma$-equivariant group homomorphism $\Phi:\widehat{Y}\rightarrow \widehat{G_1}$ is injective and maps $\widehat{Y}$ into $\Delta^1(\widehat{G_1})$.

Since the $\Gamma$-action on $Y$ is $1$-expansive, by Proposition~\ref{P-basic p-expansive}.(2) and Proposition~\ref{P-basic p-homoclinic}.(6) both $\widehat{Y}$
and $\Delta^1(Y)$  are finitely generated left $\Zb\Gamma$-modules.
As $\Zb\Gamma$ is left Noetherian, every left finitely generated $\Zb\Gamma$-module is Noetherian and finitely presented \cite[Proposition 4.29]{Lam}.
Thus both $\widehat{Y}$ and $G_1$ are finitely presented left $\Zb\Gamma$-modules. In virtue of Corollary~\ref{C-finitely presented to 1-homoclinic},
we have
$$ \rh(Y)\ge \rh(\widehat{\Delta^1(Y)})\ge \rh(\widehat{G_1}),$$
and
$$ \rh(\widehat{G_1})\ge \rh(\widehat{\Delta^1(\widehat{G_1\quad})})\ge \rh(\widehat{\widehat{Y}})=\rh(Y).$$
Therefore $\rh(Y)=\rh(\widehat{\Delta^1(Y)})=\rh(\widehat{G_1})$ as desired.

Next we show  $\rh(X)=\rh(\widehat{\Delta^1(X)})$.
 As above, both $\widehat{X}$ and $\widehat{Y/X}$ are finitely presented left $\Zb\Gamma$-modules. By Proposition~\ref{P-basic p-homoclinic}.(3) the quotient map $Y\rightarrow Y/X$ induces an embedding $\Delta^1(Y)/\Delta^1(X)\hookrightarrow \Delta^1(Y/X)$. In virtue of Corollary~\ref{C-finitely presented to 1-homoclinic} we have
$$ \rh(X)\ge \rh(\widehat{\Delta^1(X)}),$$
and
$$ \rh(Y/X)\ge \rh(\widehat{\Delta^1(Y/X)})\ge \rh(\widehat{\Delta^1(Y)/\Delta^1(X)}).$$
From Proposition~\ref{P-addition formula} we then obtain
\begin{align*}
\rh(Y)=\rh(X)+\rh(Y/X)\ge \rh(\widehat{\Delta^1(X)})+\rh(\widehat{\Delta^1(Y)/\Delta^1(X)})=\rh(\widehat{\Delta^1(Y)}).
\end{align*}
From the last paragraph we have $\rh(Y)=\rh(\widehat{\Delta^1(Y)})$. Since the $\Gamma$-action on $Y$ is $1$-expansive and $\widehat{Y}$ is a finitely presented left $\Zb\Gamma$-module,
 by Theorem~\ref{T-finite entropy vs 1-expansive} one has $\rh(Y)<+\infty$. Thus we conclude that $\rh(X)= \rh(\widehat{\Delta^1(X)})$.

Finally we show $\rh(\widehat{\Delta^1(X)})=\rh(\widehat{G})$. For this purpose we may assume that $\overline{\Delta^1(X)}=\overline{G}=X$.
Since the $\Gamma$-action on $Y$ is $1$-expansive, its restriction on $X$ is also $1$-expansive.
From the first part of the proof we conclude that $\rh(\widehat{G})=\rh(X)=\rh(\widehat{\Delta^1(X)})$.

(2). By Theorem~\ref{T-homoclinic to IE} we have $\Delta^1(X)\subseteq \IE(X)$. From Assertion (1) we have
$$ \rh(X)=\rh(\widehat{\Delta^1(X)})=\rh(\overline{\Delta^1(X)}).$$
In the above we have seen that $\rh(X)\le \rh(Y)<+\infty$. Thus, by Proposition~\ref{P-addition formula} we have
$$ \rh(X/\overline{\Delta^1(X)})=\rh(X)-\rh(\overline{\Delta^1(X)})=0.$$
In virtue of Theorem~\ref{T-topological Pinsker}.(1), we conclude that $\IE(X)\subseteq \overline{\Delta^1(X)}$. Therefore $\IE(X)=\overline{\Delta^1(X)}$.

The assertion (3) follows from the assertion (2) and Theorem~\ref{T-mIE=IE}.(4). The assertion (4) follows from the assertion (2) and Corollary~\ref{C-CPE}.
\end{proof}

\begin{example} \label{Ex-dense 1-homoclinic}
Let $\Gamma=\Zb^d$ for some $d\in \Nb$ with $d\ge 2$. Let $f\in \Zb\Gamma$ be irreducible such that
$Z(f)$ (as defined in Examples~\ref{Ex- Z^d uniform}) is nonempty but finite. For instance, one may take $f$ as $2d-\sum_{j=1}^d(u_j+u_j^{-1})$
or $d-\sum_{j=1}^du_j$ for $u_1, \dots, u_d$ being the canonical basis of $\Zb^d$.
As pointed out in Examples~\ref{Ex- Z^d uniform} and \ref{Ex-finite Z(f)}, $\alpha_f$ (as defined in Notation~\ref{N-principal}) is $1$-expansive and $\Delta^1(X_f)$ is dense in $X_f$.
On the other hand, $\alpha_f$ is not expansive \cite[Theorem 3.9]{Schmidt90}.
\end{example}

From Theorems~\ref{T-algebraic characterization of expansive}, \ref{T-homoclinic are 1-homoclinic}, \ref{T-dual entropy for Noetherian 1-homoclinic}, parts (1) and (4) of Proposition~\ref{P-basic p-expansive}, and Lemma~\ref{L-key expansive homoclinic} we have

\begin{corollary} \label{C-homoclinic dense in IE for Noetherian}
Suppose that $\Zb\Gamma$ is left Noetherian.
 Let $\Gamma$ act on a compact abelian group $X$ expansively by automorphisms.
Then the following hold:
\begin{enumerate}
\item For any $\Gamma$-invariant subgroup $G$ of $\Delta(X)$ with $\overline{G}=\overline{\Delta(X)}$, one has
$\rh(X)=\rh(\widehat{G})$.

\item $\Delta(X)$ is a dense subgroup of $\IE(X)$.

\item The action has positive entropy if and only if $\Delta(X)$ is nontrivial.

\item The action has CPE if and only if $\Delta(X)$ is dense in $X$.
\end{enumerate}
\end{corollary}

Theorem~\ref{T-main} follows from Corollary~\ref{C-homoclinic dense in IE for Noetherian} and the fact that when $\Gamma$ is polycyclic-by-finite, $\Zb\Gamma$ is left Noetherian \cite{Hall} \cite[Theorem 10.2.7]{Passman}.



\begin{thebibliography}{999}
\Small

\bibitem{AF}
F. W. Anderson and K. R.  Fuller. {\it Rings and Categories of Modules}. Second edition. Graduate Texts in Mathematics, 13. Springer-Verlag, New York, 1992.

\bibitem{Arveson}
W. Arveson. {\it An Invitation to $C^*$-algebras}.
Graduate Texts in Mathematics, No. 39. Springer-Verlag, New York-Heidelberg, 1976.

\bibitem{Berg}
K. R. Berg. Convolution of invariant measures, maximal entropy.
{\it Math. Systems Theory}  {\bf 3}  (1969),  146--150.

\bibitem{BG}
V. Bergelson and A. Gorodnik. Ergodicity and mixing of non-commuting epimorphisms.
{\it Proc. Lond. Math. Soc. (3)}  {\bf 95}  (2007),  no. 2, 329--359.

\bibitem{Bhattacharya}
S. Bhattacharya. Expansiveness of algebraic actions on connected groups. {\it Trans. Amer. Math. Soc.} {\bf 356} (2004), no. 12, 4687--4700.

\bibitem{Blanchard}
F. Blanchard. A disjointness theorem involving topological entropy.
{\it Bull.\ Soc.\ Math.\ France} {\bf 121} (1993), no. 4, 465--478.

\bibitem{BGH}
F. Blanchard, E. Glasner, and B. Host. A variation on the variational principle and applications to entropy pairs.
{\it Ergod.\ Th.\ Dynam. Sys.} {\bf 17} (1997), no. 1, 29--43.

\bibitem{BGKM}
F. Blanchard, E. Glasner, S. Kolyada, and A. Maass. On Li-Yorke pairs. {\it J. Reine Angew. Math.} {\bf 547} (2002), 51--68.

\bibitem{BHMMR}
F. Blanchard, B. Host, A. Maass, S. Martinez, and D. J. Rudolph. Entropy pairs for a measure. {\it Ergod.\ Th.\ Dynam. Sys.} {\bf 15} (1995), no. 4, 621--632.

\bibitem{BHR}
F. Blanchard, B. Host, and S. Ruette.
Asymptotic pairs in positive-entropy systems.
{\it  Ergod.\ Th.\ Dynam.\ Sys.} {\bf 22} (2002), no. 3, 671--686.

\bibitem{Bow4}
L. Bowen. Entropy for expansive algebraic actions of residually finite groups. {\it Ergod.\ Th.\ Dynam. Sys.} {\bf 31} (2011), no. 3, 703--718.

\bibitem{Bryant}
B. F. Bryant. On expansive homeomorphisms.  {\it Pacific J. Math.}  {\bf 10}  (1960), 1163--1167.

\bibitem{Chou}
C. Chou. Elementary amenable groups. {\it Illinois J. Math.} {\bf 24} (1980), no. 3, 396--407.

\bibitem{Danilenko}
A. I. Danilenko. Entropy theory from the orbital point of view.
{\it Monatsh.\ Math.}  {\bf 134}  (2001), no. 2, 121--141.

\bibitem{Den}
C. Deninger. Fuglede-Kadison determinants and entropy
for actions of discrete amenable groups.
{\it J.\ Amer.\ Math.\ Soc.} {\bf 19} (2006), 737--758.

\bibitem{Den09}
C. Deninger. Mahler measures and Fuglede--Kadison determinants.  {\it M\"{u}nster J. Math.}  {\bf 2}  (2009), 45--63.

\bibitem{DS}
C. Deninger and K. Schmidt. Expansive algebraic actions of discrete
residually finite amenable groups and their entropy.
{\it Ergod.\ Th.\ Dynam.\ Sys.} {\bf 27} (2007), 769--786.

\bibitem{ER}
M. Einsiedler and H. Rindler. Algebraic actions of the discrete Heisenberg group and other non-abelian groups.
{\it Aequationes Math.}  {\bf 62}  (2001), no. 1-2, 117--135.


\bibitem{ES}
M. Einsiedler and K. Schmidt. The adjoint action of an expansive algebraic $\Zb^d$-action.
{\it Monatsh. Math.}  {\bf 135}  (2002),  no. 3, 203--220.

\bibitem{ES02}
M. Einsiedler and K. Schmidt.  Irreducibility, homoclinic points and adjoint actions of algebraic $\Zb^d$-actions of rank one.
In: {\it Dynamics and Randomness (Santiago, 2000)},  pp. 95--124, Nonlinear Phenom. Complex Systems, 7, Kluwer Acad. Publ., Dordrecht, 2002.

\bibitem{EW05}
M. Einsiedler and T. Ward.  Entropy geometry and disjointness for zero-dimensional algebraic actions.
{\it J. Reine Angew. Math.}  {\bf 584}  (2005), 195--214.


\bibitem{Elek03}
G. Elek. On the analytic zero divisor conjecture of Linnell.
{\it Bull. London Math. Soc.}  {\bf 35}  (2003),  no. 2, 236--238.

\bibitem{FK}
B. Fuglede and R. V. Kadison. Determinant theory in finite factors.
{\it Ann. of Math. (2)}  {\bf 55} (1952), 520--530.

\bibitem{Glasner}
E. Glasner. {\it Ergodic Theory via Joinings.} American Mathematical
Society, Providence, RI, 2003.

\bibitem{GTW}
E. Glasner, J.-P. Thouvenot, and B. Weiss. Entropy theory without a past.
{\it Ergod.\ Th.\ Dynam.\ Sys.} {\bf 20} (2000), no. 5, 1355--1370.

\bibitem{GY}
E. Glasner and X. Ye. Local entropy theory. {\it Ergod.\ Th.\ Dynam. Sys.} {\bf 29} (2009), no. 2, 321--356.

\bibitem{GH}
W. H. Gottschalk and G. A. Hedlund. {\it Topological Dynamics}.
American Mathematical Society Colloquium Publications, Vol. 36. American Mathematical Society, Providence, R. I., 1955.

\bibitem{Hall}
P. Hall. Finiteness conditions for soluble groups.  {\it Proc. London Math. Soc. (3)}  {\bf 4}  (1954), 419--436.


\bibitem{Herz71}
C. Herz. The theory of $p$-spaces with an application to convolution operators. {\it Trans. Amer. Math. Soc.} {\bf 154} (1971), 69--82.

\bibitem{Herz73}
C. Herz. Harmonic synthesis for subgroups. {\it Ann. Inst. Fourier (Grenoble)} {\bf 23} (1973), no. 3, 91--123.

\bibitem{HLSY}
W. Huang, S. M. Li, S. Shao, and X. Ye. Null systems and sequence entropy pairs.
{\it Ergod.\ Th.\ Dynam.\ Sys.} {\bf 23} (2003), no. 5, 1505--1523.

\bibitem{HMRY}
W. Huang, A. Maass, P. P. Romagnoli, and X. Ye. Entropy pairs and a local Abramov formula for a measure theoretical entropy of open covers.
{\it Ergod.\ Th.\ Dynam.\ Sys.} {\bf 24} (2004), no. 4, 1127--1153.

\bibitem{HY06}
W. Huang and X. Ye. A local variational relation and applications. {\it Israel J. Math.} {\bf 151} (2006), 237--279.

\bibitem{HY09}
W. Huang and X. Ye. Combinatorial lemmas and applications to dynamics. {\it Adv. Math.} {\bf 220} (2009), no. 6, 1689--1716.

\bibitem{HYZ}
W. Huang, X. Ye, and G. Zhang. Local entropy theory for a countable discrete amenable group action. {\it J. Funct. Anal.} {\bf 261} (2011), no. 4, 1028--1082.

\bibitem{KKS}
A. Katok, S. Katok, and K. Schmidt.
Rigidity of measurable structure for $\Zb^d$-actions by automorphisms of a torus.
{\it Comment. Math. Helv.} {\bf 77} (2002), no. 4, 718--745.

\bibitem{KS}
A. Katok and R. J. Spatzier. Invariant measures for higher-rank hyperbolic abelian actions.
{\it Ergod.\ Th.\ Dynam.\ Sys.} {\bf 16} (1996), no. 4, 751--778.

\bibitem{Kechris}
A. S. Kechris. {\it Classical Descriptive Set Theory}.
Graduate Texts in Mathematics, 156. Springer-Verlag, New York, 1995.

\bibitem{Ind}
D. Kerr and H. Li. Independence in topological and $C^*$-dynamics.
{\it Math.\ Ann.} {\bf 338} (2007), no. 4, 869--926.

\bibitem{MInd}
D. Kerr and H. Li. Combinatorial independence in measurable dynamics.
{\it J. Funct. Anal.}  {\bf 256}  (2009),  no. 5, 1341--1386.

\bibitem{KerLi10}
D. Kerr and H. Li. Entropy and the variational principle for actions of sofic groups. {\it Invent. Math.} {\bf 186} (2011), no. 3, 501--558.

\bibitem{KS00}
B. Kitchens and K.  Schmidt. Isomorphism rigidity of irreducible algebraic $\Zb^d$-actions.
{\it Invent. Math.}  {\bf 142}  (2000),  no. 3, 559--577.

\bibitem{Lam}
T. Y. Lam. {\it Lectures on Modules and Rings}.
Graduate Texts in Mathematics, 189. Springer-Verlag, New York, 1999.

\bibitem{Lang99}
S. Lang. {\it Complex Analysis}. Fourth edition. Graduate Texts in Mathematics, 103. Springer-Verlag, New York, 1999.

\bibitem{Lang}
S. Lang. {\it Algebra}. Revised third edition. Graduate Texts in Mathematics, 211. Springer-Verlag, New York, 2002.

\bibitem{Ledrappier}
F. Ledrappier. Un champ markovien peut \^{e}tre d'entropie nulle et m\'{e}langeant.
{\it C. R. Acad. Sci. Paris S\'{e}r. A-B} {\bf 287} (1978), no. 7, A561--A563.

\bibitem{Li}
H. Li. Compact group automorphisms, addition formulas and Fuglede-Kadison determinants. {\it Ann. of Math. (2)} {\bf 176} (2012), no. 1, 303--347.

\bibitem{Lind77}
D.  Lind. The structure of skew products with ergodic group automorphisms.
{\it Israel J. Math.}  {\bf 28}  (1977), no. 3, 205--248.


\bibitem{LS99}
D. Lind and K. Schmidt. Homoclinic points of algebraic $\Zb^d$-actions.
{\it J. Amer. Math. Soc.}  {\bf 12}  (1999),  no. 4, 953--980.

\bibitem{LS02}
D. Lind and K. Schmidt.  Symbolic and algebraic dynamical systems.
In: {\it Handbook of Dynamical Systems, Vol. 1A,}  pp. 765--812, North-Holland, Amsterdam, 2002.

\bibitem{LSV}
D. Lind, K. Schmidt, and E. Verbitskiy.
Entropy and growth rate of periodic points of algebraic $\Zb^d$-actions. In: {\it Dynamical Numbers: Interplay between Dynamical Systems and Number Theory}, pp. 195--211, Contemp. Math., 532, Amer. Math. Soc., Providence, RI, 2010.

\bibitem{LSW}
D. Lind, K. Schmidt, and T. Ward. Mahler measure and
entropy for commuting automorphisms of compact groups.
{\it Invent.\ Math.}  {\bf 101}  (1990), 593--629.

\bibitem{LW}
E. Lindenstrauss and B. Weiss. Mean topological dimension.
{\it Israel J.\ Math.} {\bf 115} (2000), 1--24.

\bibitem{Linnell91}
P. A. Linnell. Zero divisors and group von Neumann algebras. {\it Pacific J. Math.} {\bf 149} (1991), no. 2, 349--363.

\bibitem{Linnell92}
P. A. Linnell. Zero divisors and $L^2(G)$. {\it C. R. Acad. Sci. Paris S\'{e}r. I Math.} {\bf 315} (1992), no. 1, 49--53.

\bibitem{Linnell93}
P. A. Linnell. Division rings and group von Neumann algebras.
{\it Forum Math.}  {\bf 5}  (1993),  no. 6, 561--576.

\bibitem{Linnell98}
P. A. Linnell. Analytic versions of the zero divisor conjecture.
In: {\it Geometry and Cohomology in Group Theory (Durham, 1994)}, pp. 209--248, London Math. Soc. Lecture Note Ser., 252, Cambridge Univ. Press, Cambridge, 1998.

\bibitem{LP}
P. A. Linnell and M. J. Puls. Zero divisors and $L^p(G)$, II. {\it New York J. Math.} {\bf 7} (2001), 49--58 (electronic).

\bibitem{Luck}
W. L\"{u}ck. {\it $L^2$-Invariants: Theory and Applications to Geometry and $K$-theory}.
Springer-Verlag, Berlin, 2002.

\bibitem{MT78}
G. Miles and R. K. Thomas. Generalized torus automorphisms are Bernoullian.
In: {\it Studies in Probability and Ergodic Theory},  pp. 231--249, Adv. in Math. Suppl. Stud., 2, Academic Press, New York-London, 1978.

\bibitem{Miles}
R. Miles. Expansive algebraic actions of countable abelian groups. {\it Monatsh. Math.} {\bf 147} (2006), no. 2, 155--164.

\bibitem{MW}
R. Miles and T. Ward. Orbit-counting for nilpotent group shifts.  {\it Proc. Amer. Math. Soc.}  {\bf 137}  (2009),  no. 4, 1499--1507.

\bibitem{JMO}
J. Moulin Ollagnier. {\it Ergodic Theory and Statistical Mechanics.}
Lecture Notes in Math., 1115. Springer, Berlin, 1985.

\bibitem{OW}
D. S. Ornstein and B. Weiss. Entropy and isomorphism theorems for
actions of amenable groups.  {\it J. Analyse Math.} {\bf 48} (1987),
1--141.

\bibitem{Passman}
D. S. Passman. {\it The Algebraic Structure of Group Rings}.
Pure and Applied Mathematics. Wiley-Interscience [John Wiley \& Sons], New York-London-Sydney, 1977.

\bibitem{Peters}
J. Peters. Entropy on discrete abelian groups.
{\it Adv. in Math.} {\bf 33} (1979), no. 1, 1--13.

\bibitem{Pisier}
G. Pisier. {\it The Volume of Convex Bodies and Banach Space Geometry}. Cambridge Tracts in Mathematics, 94. Cambridge University Press, Cambridge, 1989.

\bibitem{Puls}
M. J. Puls. Zero divisors and $L^p(G)$. {\it Proc. Amer. Math. Soc.} {\bf 126} (1998), no. 3, 721--728.

\bibitem{ell1}
H. P. Rosenthal. A characterization of Banach spaces containing $l_1$.
{\it Proc.\ Nat.\ Acad.\ Sci.\ USA} {\bf 71} (1974), 2411--2413.


\bibitem{Rudin}
W. Rudin. {\it Real and Complex Analysis}. Third edition. McGraw-Hill Book Co., New York, 1987.

\bibitem{RS}
D. J. Rudolph and K. Schmidt. Almost block independence and Bernoullicity of $Z^d$-actions by automorphisms of compact abelian groups.
{\it Invent. Math.}  {\bf 120}  (1995),  no. 3, 455--488.

\bibitem{Ruelle}
D. Ruelle. Statistical mechanics on a compact set with $Z^{v}$ action satisfying expansiveness and specification.
{\it Trans. Amer. Math. Soc.} {\bf 187} (1973), 237--251.

\bibitem{Schmidt90}
K. Schmidt. Automorphisms of compact abelian groups and affine varieties.
{\it Proc. London Math. Soc. (3)}  {\bf 61}  (1990),  no. 3, 480--496.

\bibitem{Schmidt}
K. Schmidt. {\it Dynamical Systems of Algebraic Origin}.
Progress in Mathematics, 128. Birkh\"{a}user Verlag, Basel, 1995.

\bibitem{Schmidt95}
K. Schmidt. The cohomology of higher-dimensional shifts of finite type. {\it Pacific J. Math.} {\bf 170} (1995), no. 1, 237--269.

\bibitem{Schmidt01}
K. Schmidt. The dynamics of algebraic $\Zb^d$-actions.
In: {\it European Congress of Mathematics, Vol. I (Barcelona, 2000)},  pp. 543--553, Progr. Math., 201, Birkh\"{a}user, Basel, 2001.

\bibitem{SV}
K. Schmidt and E. Verbitskiy. Abelian sandpiles and the harmonic model. {\it Comm. Math. Phys.} {\bf 292} (2009), no. 3, 721--759.


\bibitem{SW}
K. Schmidt and T. Ward. Mixing automorphisms of compact groups and a theorem of Schlickewei.
{\it Invent. Math.}  {\bf 111}  (1993),  no. 1, 69--76.

\bibitem{Segal}
D. Segal. {\it Polycyclic Groups}.
Cambridge Tracts in Mathematics, 82. Cambridge University Press, Cambridge, 1983.

\bibitem{Walters}
P. Walters. {\it An Introduction to Ergodic Theory.} Graduate Texts in
Mathematics, 79. Springer-Verlag, New York, Berlin, 1982.

\bibitem{Wang}
X. Wang. Volumes of generalized unit balls. {\it Math. Mag.} {\bf 78} (2005), no. 5, 390--395.

\bibitem{Yuz1}
S. A. Yuzvinski\u{\i}. Metric properties of the endomorphisms of compact groups. (Russian)
{\it Izv. Akad. Nauk SSSR Ser. Mat.}  {\bf 29}  (1965), 1295--1328.
Translated in {\it Amer. Math. Soc. Transl. (2)} {\bf 66} (1968), 63--98.

\bibitem{Yuz2}
S. A. Yuzvinski\u{\i}. Computing the entropy of a group of endomorphisms. (Russian)  {\it Sibirsk. Mat. \^{Z}.}  {\bf 8}  (1967), 230--239.
Translated in {\it Siberian Math. J.} {\bf 8} (1967), 172--178.


\end{thebibliography}
\end{document}